\documentclass[reqno,12pt]{amsart}

\usepackage[all]{xypic}
\usepackage{enumerate}
\usepackage{hyperref}
\usepackage{a4}
\usepackage{amssymb}

\DeclareMathOperator*{\LIM}{LIM}
\DeclareMathOperator{\sdet}{sdet}
\DeclareMathOperator{\img}{img}

\DeclareMathOperator{\rk}{rk}
\DeclareMathOperator{\gr}{gr}
\DeclareMathOperator{\tr}{tr}
\DeclareMathOperator{\str}{str}
\DeclareMathOperator{\Ind}{Ind}

\newcommand\RS{{\rm{RS}}}
\newcommand\goe{\mathfrak g}

\newcommand\llangle{\langle\!\langle}
\newcommand\rrangle{\rangle\!\rangle}
\DeclareMathOperator{\Tor}{Tor}

\DeclareMathOperator{\id}{id}
\DeclareMathOperator{\GL}{GL}
\DeclareMathOperator{\Aut}{Aut}
\DeclareMathOperator{\eend}{end}
\DeclareMathOperator{\der}{der}
\newcommand\Z{\mathbb Z}
\newcommand\N{\mathbb N}

\newcommand\C{\mathbb C}

\theoremstyle{plain}
  \newtheorem{theorem}{Theorem}[section]
  \newtheorem{corollary}[theorem]{Corollary}
  \newtheorem{lemma}[theorem]{Lemma}
  \newtheorem{proposition}[theorem]{Proposition}
  
\theoremstyle{definition}
  \newtheorem{example}[theorem]{Example}
\theoremstyle{remark}
  \newtheorem{remark}[theorem]{Remark}

\begin{document}

\title[Analytic torsion of generic rank two distributions]{Analytic torsion of generic rank two distributions in dimension five}

\author{Stefan Haller}

\address{Stefan Haller,
         Department of Mathematics,
         University of Vienna,
         Oskar-Morgenstern-Platz 1,
         1090 Vienna,
	Austria.}

\email{stefan.haller@univie.ac.at}

\begin{abstract}
We propose an analytic torsion for the Rumin complex associated with generic rank two distributions on closed 5-manifolds.
This torsion behaves as expected with respect to Poincar\'e duality and finite coverings.
We establish anomaly formulas, expressing the dependence on the sub-Riemannian metric and the 2-plane bundle in terms of integrals over local quantities.
For certain nilmanifolds we are able to show that this torsion coincides with the Ray--Singer analytic torsion, up to a constant.
\end{abstract}

\keywords{Analytic torsion; Rumin complex; Rockland complex; Generic rank two distribution; (2,3,5) distribution; sub-Riemannian geometry}

\subjclass[2010]{58J52 (primary) and 53C17, 58A30, 58J10, 58J42}

\maketitle


\section{Introduction}\label{S:intro}

The Ray--Singer analytic torsion \cite{RS71} of a closed smooth manifold is a zeta regularized graded determinant of its de~Rham complex and computes the Reidemeister torsion \cite{Ch77,Ch79,M78,BZ92}.
Rumin and Seshadri \cite{RS12} have recently introduced an analytic torsion for the Rumin complex of contact manifolds \cite{R90,R94,R00} and showed that it coincides with the Ray--Singer torsion for 3-dimensional CR Seifert manifolds equipped with a unitary representation \cite[Theorem~4.2]{RS12}.
In this paper we propose an analytic torsion of the Rumin complex associated with another filtered geometry \cite{R99,R01,R05} known as generic rank two distributions in dimension five \cite{C10}.

A generic rank two distribution on a 5-manifold $M$ is a smooth rank two subbundle $\mathcal D$ in the tangent bundle $TM$ such that Lie brackets of sections of $\mathcal D$ span a rank three subbundle $[\mathcal D,\mathcal D]$ in $TM$ and triple brackets span all of the tangent bundle.
In other words, $\mathcal D$ is a bracket generating distribution with growth vector $(2,3,5)$.
Clearly, this is a $C^2$-open condition on the 2-plane bundle $\mathcal D$, whence the name generic.

If $\mathcal D$ is a generic rank two distribution, then the Lie bracket of vector fields induces an algebraic (Levi) bracket on the associated graded bundle 
\[
	\mathfrak tM=\mathfrak t^{-3}M\oplus\mathfrak t^{-2}M\oplus\mathfrak t^{-1}M,
\]
where $\mathfrak t^{-1}M=\mathcal D$, $\mathfrak t^{-2}M=[\mathcal D,\mathcal D]/\mathcal D$ and $\mathfrak t^{-3}M=TM/[\mathcal D,\mathcal D]$.
This turns $\mathfrak tM$ into a bundle of graded nilpotent Lie algebras called the bundle of osculating algebras.
Its fibers are all isomorphic to the graded nilpotent Lie algebra $\goe=\goe_{-3}\oplus\goe_{-2}\oplus\goe_{-1}$ with graded basis $X_1,X_2\in\goe_{-1}$, $X_3\in\goe_{-2}$, $X_4,X_5\in\goe_{-3}$ and brackets
\[
	[X_1,X_2]=X_3,\qquad[X_1,X_3]=X_4,\qquad[X_2,X_3]=X_5.
\]

A basic example of a generic rank two distribution is obtained by equipping the simply connected Lie group $G$ corresponding to the Lie algebra $\goe$ with the $G$-invariant 2-plane bundle obtained by translating $\goe_{-1}$.

The study of generic rank two distributions reaches back to Cartan's ``five variables paper'' \cite{C10} from 1910 where Cartan constructs a canonical connection and derives a non-trivial local invariant, the harmonic curvature tensor which is a section of $S^4\mathcal D^*$.
A generic rank two distribution is locally diffeomorphic to the invariant rank two distribution on the nilpotent Lie group $G$ if and only if Cartan's curvature tensor vanishes.

In modern terminology, generic rank two distributions can equivalently be described as regular normal parabolic geometries \cite{CS09} of type $(G_2,P)$, where $P$ is a particular parabolic subgroup in the split real form of the exceptional Lie group $G_2$, see also \cite{S08}. 
Their geometry has many intriguing features.
The most symmetric example, the flat model, is the homogeneous space $G_2/P$ which is locally diffeomorphic to the nilpotent Lie group $G$ and has an underlying manifold diffeomorphic to $S^2\times S^3$, see \cite{S06}.
In general, the symmetry group of a generic rank two distribution is a Lie group of dimension at most 14 and is subject to further restrictions \cite[Theorem~7]{CN09}.
Nurowski \cite{N05} constructed a conformal metric of signature $(2,3)$ with conformal holonomy $G_2$ which is naturally associated with a generic rank two distribution, see also \cite{S08,CSa09,HS09,SW17,SW17b}.
Bryant and Hsu have shown that generic rank two distributions admit many rigid curves \cite{BH93}, i.e., curves which are isolated in the path space of curves tangent to $\mathcal D$ with fixed endpoints.
In the same paper Bryant and Hsu also discuss the mechanical system of a surface rolling without slipping and twisting on another surface.
This gives rise to a 5-dimensional configuration space equipped with a rank two distribution encoding the no slipping and twisting condition, which is generic iff the Gaussian curvatures are disjoint \cite[Section 4.4]{BH93}.
If both surfaces are round spheres and the ratio of their radii is 3:1 then the universal covering of the configuration space is diffeomorphic to $G_2/P$, see \cite{C10,BM09}.
Let us also mention the following recent studies of generic rank two distributions \cite{AN14,BHN18,GPW17,HS11,LNS17}.

By virtue of Gromov's h-principle for open manifolds \cite[Section~2.2.2]{G86}, it is well understood \cite[Theorem~2]{DH19} which open 5-manifolds admit globally defined generic rank two distributions.
For closed manifolds the situation appears to be more subtle but the underlying topological problem has been settle, see \cite[Theorem~1(b)]{DH19}.
The question to what extent generic rank two distributions abide by an h-principle \cite{G86,E02} has served as a motivation for us to study the analytic torsion of the associated Rumin complex, cf.~\cite{E89,G91,BEM15,V09,CPdPP17,CdPP20,dPV20}.

The Rumin complex \cite{R99,R01,R05} associated with a generic rank two distribution on a 5-manifold $M$ is a natural complex of higher order differential operators,
\[
	\Gamma^\infty(E^0)\xrightarrow{D_0}
	\Gamma^\infty(E^1)\xrightarrow{D_1}
	\Gamma^\infty(E^2)\xrightarrow{D_2}
	\Gamma^\infty(E^3)\xrightarrow{D_3}
	\Gamma^\infty(E^4)\xrightarrow{D_4}
	\Gamma^\infty(E^5)
\]
where $E^q=\mathcal H^q(\mathfrak tM)$ denotes the vector bundle obtained by taking the fiberwise Lie algebra cohomology of the bundle of osculating algebras $\mathfrak tM$.
The ranks of the bundles $E^q$ are $1,2,3,3,2,1$ and the Heisenberg orders of the operators $D_q$ are $1,3,2,3,1$, see \cite[Section~5]{BEN11} or \cite[Example~4.24]{DH17}.
There exist injective differential operators $L_q\colon\Gamma^\infty(E^q)\to\Omega^q(M)$, embedding the Rumin complex as a subcomplex in the de~Rham complex, whose image can be characterized using differential operators.
The operators $L_q$ intertwine the Rumin differential with the de~Rham differential, $dL_q=L_{q+1}D_q$, and induce isomorphisms in cohomology.
In particular, the Rumin complex computes the cohomology of $M$.

Let us emphasize that the Rumin differentials $D_q$ are natural, i.e.\ independent of any further choices, and so is the map induced on cohomology by $L_q$, see Lemma~\ref{L:Sindep}.
The operators $L_q$, however, depend on the choice of a sub-Riemannian metric $\tilde g$ on $\mathcal D$ and a splitting of the filtration, $S\colon\mathfrak tM\to TM$.
By the latter we mean a filtration preserving isomorphism inducing the identity on the associated graded, that is, $S|_{\mathfrak t^{-1}M}=\id$, $S(\mathfrak t^{-2}M)\subseteq[\mathcal D,\mathcal D]$, $S|_{\mathfrak t^{-2}M}=\id$ mod $\mathcal D$, and $S|_{\mathfrak t^{-3}M}=\id$ mod $[\mathcal D,\mathcal D]$.
Using parabolic geometry, one can construct a variant of the Rumin complex $D_q$ with splitting operators $L_q$ which are entirely natural.
Indeed, the Rumin complex is isomorphic to the curved BGG sequence \cite{CSS01} associated with the trivial representation of $G_2$.
Writing down the BGG operators explicitly, however, requires the determination of the canonical regular normal Cartan connection associated with the rank two distribution \cite{C10}, a quite involved procedure, see \cite{S08}.
While this more natural perspective might prove to be insightful when studying the sub-Riemannian limit, it does not appear to be helpful for the purpose of this paper, and we will therefore not adopt it here.

Suppose $F$ is a flat complex vector bundle over $M$ and let $D^F_q$ denote the differential complex obtained by twisting the Rumin complex with $F$.
The definition of the analytic torsion requires a sub-Riemannian metric \cite{G96,MG02}, i.e., a fiberwise Euclidean inner product $\tilde g$ on $\mathcal D$, as well as a fiberwise Hermitian metric $h$ on $F$.
These choices give rise to fiberwise Hermitian metrics on the vector bundles $E^q$ and permit to define formal adjoints of the Rumin differentials, denoted by $(D^F_q)^*$.
The essential part of the analytic torsion is a zeta regularized graded determinant which may be expressed in the form
\[
	\log\sdet\sqrt{(D^F)^*D^F}=-\frac12\frac\partial{\partial s}\Big|_{s=0}\sum_q(-1)^q\tr\left(\left((D^F_q)^*D^F_q\right)^{-s}\right)
\]
where the complex powers are understood to vanish on the kernels of the operators, see Remark~\ref{R:Z2}.
For our rigorous definition, however, we will proceed as in Ray--Singer \cite{RS71} or Rumin--Seshadri \cite{RS12} and rewrite this in terms of zeta functions associated with hypoelliptic Rumin--Seshadri type operators 
\[
	\Delta^F_q=\bigl((D^F_{q-1}D^F_{q-1})^*\bigr)^{a_{q-1}}+\bigl((D_q^F)^*D_q^F\bigr)^{a_q},
\]
see Section~\ref{SS:anator}.
For appropriate choices of the numbers $a_q$, the operators $\Delta_q^F$ are Rockland of Heisenberg order $2\kappa$ and admit a parametrix in the Heisenberg calculus \cite{M82,EY19,DH17}.
The zeta function $\tr\bigl((\Delta_q^F)^{-s}\bigr)$ converges for $\Re(s)>10/2\kappa$ and extends to a meromorphic function on the entire complex plane which is holomorphic at zero, see \cite[Theorem~2]{DH20}.
Correspondingly, as $t\to0$, the heat kernel admits an asymptotic expansion \cite[Theorem~1]{DH20} of the form
\begin{equation}\label{E:heatk235}
	e^{-t\Delta^F_q}(x,x)\sim\sum_{j=0}^\infty t^{(j-10)/2\kappa}p_{q,j}^F(x)
\end{equation}
where $p^F_{q,j}\in\Gamma^\infty\bigl(\eend(\mathcal H^q(\mathfrak tM)\otimes F)\otimes|\Lambda|\bigr)$ are locally computable, and $p_{q,j}^F=0$ for all odd $j$.
Here $|\Lambda|$ denotes the line bundle of 1-densities over $M$.

Following \cite{BZ92} we incorporate the zero eigenspaces and consider the analytic torsion as a norm $\|-\|^{\sdet(H^*(M;F))}_{\mathcal D,\tilde g,h}$ on the graded determinant line \cite{KM76}
\[
	\sdet(H^*(M;F)):=\bigotimes_q\bigl(\det H^q(M;F)\bigr)^{(-1)^q},
\]
where $H^q(M;F)$ denotes the de~Rham cohomology of $M$ with coefficients in the flat bundle $F$.
This analytic torsion behaves as expected with respect to finite coverings and Poincar\'e duality, see Theorem~\ref{T:PD}.
Moreover, this definition leads to a simple and local anomaly formula, describing the metric dependence, which is analogous to the corresponding formulas for the Ray--Singer torsion \cite[Theorem~0.1]{BZ92} and the Rumin--Seshadri torsion \cite[Corollary~3.7]{RS12}.

To state this formula, note that restriction provides natural isomorphisms
\[
	\Aut(\mathfrak tM)=\Aut(\mathcal D)\qquad\text{and}\qquad\der(\mathfrak tM)=\eend(\mathcal D),
\]
where $\Aut(\mathfrak tM)$ denotes the bundle of fiberwise graded Lie algebra automorphisms and $\der(\mathfrak tM)$ denotes the bundle of fiberwise derivations of degree zero.
Hence, each $A\in\Gamma^\infty(\Aut(\mathcal D))$ induces $\mathcal H^q(A)\in\Gamma^\infty(\Aut(\mathcal H^q(\mathfrak tM)))$ and correspondingly, each $\dot A\in\Gamma^\infty(\eend(\mathcal D))$ induces $\mathcal H(\dot A)\in\Gamma^\infty(\eend(\mathcal H^q(\mathfrak tM)))$.

\begin{theorem}\label{T:var235}
Suppose $\tilde g_u$ is a family of sub-Riemannian metrics on $\mathcal D$ and $h_u$ is a family of fiberwise Hermitian metrics on $F$, both depending smoothly on a real parameter $u$.
Then
\begin{multline*}
	\tfrac\partial{\partial u}\log\|-\|^{\sdet(H^*(M;F))}_{\mathcal D,\tilde g_u,h_u}
	=\frac12\sum_q(-1)^q\int_M\tr\Bigl(\bigl(\mathcal H^q(\dot g_u)+\tfrac52\tr(\dot g_u)+\dot h_u\bigr)p_{q,u,10}^F\Bigr)
\end{multline*}
where $\dot g_u:=\tilde g^{-1}_u\frac\partial{\partial u}\tilde g_u\in\Gamma^\infty(\eend(\mathcal D))$, $\dot h_u:=h^{-1}_u\frac\partial{\partial u}h_u\in\Gamma^\infty(\eend(F))$, and $p^F_{q,u,10}\in\Gamma^\infty\bigl(\eend(\mathcal H^q(\mathfrak tM)\otimes F)\otimes|\Lambda|\bigr)$ denotes the constant term in the heat kernel expansions associated with the Rumin--Seshadri operator $\Delta^F_{u,q}$, cf.~\eqref{E:heatk235}.
\end{theorem}

Of course, the Ray--Singer torsion of a 5-manifold is independent of the Riemannian metric and the Hermitian metric on $F$, for there are no constant terms in the small time asymptotic expansion of the heat trace of the Hodge Laplacians in odd dimensions.
The relevant analogue in the heat trace asymptotics for Rockland operators on filtered manifolds is the homogeneous dimension, which happens to be even (ten) for generic rank two distributions in dimension five.
Hence, there is no apparent analytical reason, why the corresponding analytic torsion should be independent of $\tilde g$ and $h$.
The situation is similar for the Rumin--Seshadri analytic torsion of contact manifolds which have even homogeneous dimension too.
On contact $3$-manifolds, by adding (potentially vanishing) local correction terms, Rumin and Seshadri were able to turn their torsion into a CR-invariant \cite[Corollary~3.8(3)]{RS12} which coincides with the Ray--Singer torsion according to \cite{AQ19}.
We will not compute the local quantities appearing in the anomaly formula in Theorem~\ref{T:var235} more explicitly in this paper.

In view of Gray's stability theorem \cite[Theorem~2.2.2]{G08}, a smooth deformation of a contact structure can always be absorbed by adjusting the sub-Riemannian metric.
Hence, the infinitesimal change of the Rumin--Seshadri analytic torsion under a smooth deformation of the contact structure can be expressed as an integral over a local quantity.
For generic rank two distributions the situation is fundamentally different due to the local invariant provided by Cartan's curvature tensor.
Even a small perturbation of the 2-plane bundle will in general result in a distribution which is not even locally diffeomorphic to the original one.
Nevertheless, the torsion changes by a local quantity.
More precisely, we have

\begin{theorem}\label{T:var.all235}
Suppose $\Theta_u\in\Gamma^\infty(\Aut(TM))$ depends smoothly on a real parameter $u$ such that $\Theta_0=\id_{TM}$.
Consider the 2-plane bundle $\mathcal D_u=\Theta_u(\mathcal D)$ equipped with the sub-Riemannian metric $\tilde g_u=(\Theta_u)_*\tilde g$.
Then, provided $u$ is sufficiently small, $\mathcal D_u$ is a generic rank two distribution and we have
\[
	\tfrac\partial{\partial u}\big|_{u=0}\log\|-\|^{\sdet(H^*(M;F))}_{\mathcal D_u,\tilde g_u,h}=\frac12\int_M\alpha
\]
where the density $\alpha$ is locally computable.
More precisely, $\alpha(x)$ can be computed from the germ of $(M,\mathcal D,\tilde g,F,h,\dot\Theta|_{\mathcal D})$ at $x$, where $\dot\Theta=\frac\partial{\partial u}\big|_{u=0}\Theta_u\in\Gamma^\infty(\eend(TM))$.
\end{theorem}

A more explicit formula for the density $\alpha$ will be provided in the proof presented in Section~\ref{SS:proof}, see \eqref{E:dotAF235}, \eqref{E:tpj}, and \eqref{E:alpha}.

We will compute the torsion for generic rank two distributions on some nilmanifolds, up to a constant.
To formulate this result, let $G$ denote the simply connected Lie group with Lie algebra $\goe$ and let $\mathcal D$ denote a right invariant generic rank two distribution on $G$.
We equip $\mathcal D$ with a right invariant sub-Riemannian metric $\tilde g$.
Since the structure constants of $\goe$ are rational, the group $G$ admits lattices \cite[Theorem 2.12]{R72}.
If $\Gamma$ is a lattice in $G$, then $\mathcal D$ and $\tilde g$ descend to the nilmanifold $G/\Gamma$.
The induced structures on $G/\Gamma$ will be denoted by $\mathcal D_\Gamma$ and $\tilde g_\Gamma$, respectively.
We have the following comparison with the Ray--Singer torsion \cite{RS71,BZ92} which will be denoted by $\|-\|^{\sdet(H^*(G/\Gamma;F))}_\RS$.

\begin{theorem}\label{T:intro}
There exists a constant $c>0$ such that
\[
	\|-\|^{\sdet(H^*(G/\Gamma;F))}_{\mathcal D_\Gamma,\tilde g_\Gamma,h}=c\cdot\|-\|^{\sdet(H^*(G/\Gamma;F))}_\RS
\]
for every right invariant generic rank two distribution $\mathcal D$ on $G$, every lattice $\Gamma$ in $G$ which is generated by two elements, and every flat line bundle $F$ over $G/\Gamma$ with parallel fiberwise Hermitian metric $h$.
\end{theorem}

Actually, we will establish a slightly stronger result in Theorem~\ref{T:main} below.

Since the nilmanifold $G/\Gamma$ admits a free circle action, its Ray--Singer torsion can be read off the corresponding Gysin sequence, see \cite[Corollary~0.9]{LST98}.

To prove Theorem~\ref{T:intro}, we will consider the lattice $\Gamma_0$ in $G$ generated by two elements $\exp(X_1)$ and $\exp(X_2)$ where $X_1,X_2$ is a basis of $\goe_{-1}$.
The group of graded automorphisms $\Aut(\goe)\cong\GL(2,\mathbb R)$ acts on $G$ by diffeomorphisms preserving $\mathcal D_0$, the right invariant rank two distribution obtained by translating $\goe_{-1}$.
The subgroup preserving the lattice, $\Aut(\goe,\log\Gamma_0)\cong\GL(2,\mathbb Z)$, acts on the nilmanifold $G/\Gamma_0$ by filtration preserving diffeomorphisms.
Moreover, the latter group has dense orbits when acting on the space of unitary characters of $\Gamma_0$, that is, the space of flat line bundles over $G/\Gamma_0$ which admit a parallel Hermitian metric.
Exploiting this fact, and using several other results established here, including Theorems~\ref{T:var235} and \ref{T:var.all235}, we are able to give a proof of Theorem~\ref{T:intro}.

We expect that the statement in Theorem~\ref{T:intro} remains true with $c=1$ for arbitrary lattices $\Gamma$ and all flat bundles $F$ with parallel Hermitian metric $h$.
This expectation is motivated by a recent result of Albin and Quan \cite[Corollary~3]{AQ19}, asserting that the quotient of the Rumin--Seshadri torsion and the Ray--Singer torsion can be expressed in terms of an integral over a local quantity.
If their analysis of the sub-Riemannian limit can be generalized to generic rank two distributions, the aforementioned strengthening of Theorem~\ref{T:intro} would follow at once.

The remaining part of this paper is organized as follows.
In Section~\ref{S:torRc} we define an analytic torsion for Rockland differential complexes over general closed filtered manifolds.
We establish a metric anomaly formula for this torsion in Theorem~\ref{T:var} and describe the variation of the torsion through a deformation of the underlying filtration in Theorem~\ref{T:var.all}.
Some basic properties including compatibility with finite coverings and duality are collected in Section~\ref{SS:elementary}.

In Section~\ref{S:torRumin} we apply this general framework to the Rumin complex associated with certain filtered manifolds.
The resulting anomaly formulas are contained in Theorem~\ref{T:varg} and Theorem~\ref{T:def.filt}, compatibility with Poincar\'e duality is formulated in Theorem~\ref{T:PD}.
In order for the Rumin complex to be Rockland, it is necessary to assume that the osculating algebras have pure cohomology.
This apparently very restrictive assumption will be discussed in Section~\ref{SS:pure}.
We only know of three types of filtered manifolds with this property: trivially filtered manifolds giving rise to the Ray--Singer torsion, contact manifolds giving rise to the Rumin--Seshadri torsion, and generic rank two distributions in dimension five giving rise to the torsion proposed in this paper.

In Section~\ref{S:five} we specialize to generic rank two distributions in dimension five and prove the results stated in the introduction.

\subsection*{Acknowledgments}
The author is indebted to Shantanu Dave for countless encouraging discussions.
He gratefully acknowledges the support of the Austrian Science Fund (FWF): P31663-N35 and Y963-N35.

\section{Analytic torsion of Rockland complexes}\label{S:torRc}

To every Rockland complex of differential operators over a closed filtered manifold one can associate an analytic torsion.
For the de~Rham complex this specializes to the classical Ray--Singer torsion \cite{RS71,BZ92}.
For the Rumin complex of a contact manifold it specializes to the Rumin--Seshadri analytic torsion \cite{RS12}.

The purpose of this section is to provide a rigorous definition of the analytic torsion of general Rockland complexes and to discuss some basic properties.
We will show that this torsion is compatible with duality and finite coverings, see Section~\ref{SS:elementary}.
Moreover, we will establish an anomaly formula which shows that the dependence on the metric is given by an integral over a local quantity, see Theorem~\ref{T:var}.
We will also study the variation of the torsion during a deformation of the underlying filtration, and show that in certain situations, this can be expressed in terms of local quantities also, cf.~Theorem~\ref{T:var.all}.

In the Sections~\ref{S:torRumin} and \ref{S:five} will apply these results to Rumin complexes associated with certain filtered manifolds.

\subsection{Differential operators on filtered manifolds}

Recall that a filtered manifold is a smooth manifold $M$ together with a filtration of the tangent bundle $TM$ by smooth subbundles,
$$
TM=T^{-r}M\supseteq\cdots\supseteq T^{-2}M\supseteq T^{-1}M\supseteq T^0M=0,
$$
which is compatible with the Lie bracket of vector fields in the following sense: If $X$ is a smooth section of $T^pM$ and $Y$ is a smooth sections of $T^qM$ then the Lie bracket $[X,Y]$ is a section of $T^{p+q}M$.
Putting $\mathfrak t^pM:=T^pM/T^{p+1}M$, the Lie bracket induces a vector bundle homomorphism $\mathfrak t^pM\otimes\mathfrak t^qM\to\mathfrak t^{p+q}M$ referred to as Levi bracket.
This turns the associated graded vector bundle $\mathfrak tM:=\bigoplus_p\mathfrak t^pM$ into a bundle of graded nilpotent Lie algebras called the bundle of osculating algebras.
The Lie algebra structure on the fiber $\mathfrak t_xM=\bigoplus_p\mathfrak t^p_xM$ depends smoothly on the base point $x\in M$, but will in general not be locally trivial.
In particular, the Lie algebras $\mathfrak t_xM$ might be non-isomorphic for different $x\in M$.
The simply connected nilpotent Lie group with Lie algebra $\mathfrak t_xM$ is called osculating group at $x$ and will be denoted by $\mathcal T_xM$.

Using negative degrees, we are following a convention common in parabolic geometry, see \cite{CS09,M93,M02} for instance.
The other convention, where everything is concentrated in positive degrees, is the one that has been adopted in \cite{EY19,EY17}.

The filtration on $M$ gives rise to a Heisenberg filtration on differential operators which is compatible with composition and transposition.
The basic idea is to consider differentiation along a vector field tangent to $T^{-k}M$ as a an operator of Heisenberg order at most $k$.
Suppose $E$ and $F$ are two vector bundles over $M$.
A differential operator $A\colon\Gamma^\infty(E)\to\Gamma^\infty(F)$ of Heisenberg order at most $k$ has a Heisenberg principal (co)symbol, encoding the highest order derivatives in the Heisenberg sense at $x\in M$,
\[
	\sigma^k_x(A)\in\mathcal U_{-k}(\mathfrak t_xM)\otimes\hom(E_x,F_x)
\]
where $\mathcal U_{-k}(\mathfrak t_xM)$ denotes the degree $-k$ part of the universal enveloping algebra of the graded nilpotent Lie algebra $\mathfrak t_xM=\bigoplus_p\mathfrak t^p_xM$.
Equivalently, the Heisenberg principal symbol may be regarded as a left invariant differential operator
\[
	\sigma^k_x(A)\colon C^\infty(\mathcal T_xM,E_x)\to C^\infty(\mathcal T_xM,F_x)
\]
on the osculating group $\mathcal T_xM$ which is homogeneous of degree $k$ with respect to the grading automorphism.
The Heisenberg principal symbol is compatible with composition and transposition of differential operators.

A differential operator $A\colon\Gamma^\infty(E)\to\Gamma^\infty(F)$ of Heisenberg order at most $k$ is said to satisfy the Rockland condition \cite{R78} if 
\[
	\pi(\sigma^k_x(A))\colon\mathcal H_\infty\otimes E_x\to\mathcal H_\infty\otimes F_x
\]
is injective for every $x\in M$ and every non-trivial irreducible unitary Hilbert space representation $\pi$ of the osculating group $\mathcal T_xM$, where $\mathcal H_\infty$ denotes the subspace of smooth vectors, cf.~\cite{K04}.
The Rockland theorem asserts that in this situation the operator admits a left parametrix which is of order $-k$ in an appropriate Heisenberg calculus of pseudodifferential operators adapted to the filtration.
In particular, Rockland operators are hypoelliptic.

This result has a long history.
For trivially filtered manifolds, that is if $TM=T^{-1}M$, it reduces to the classical, elliptic case.
In this situation all irreducible unitary representations of the (abelian) osculating group are one dimensional, and the Rockland condition at $x\in M$ becomes the familiar condition that the principal symbol of the operator is invertible at every $0\neq\xi\in T_x^*M$, provided $\rk(E)=\rk(F)$.
Helffer--Nourrigat \cite{HN78,HN79,HN85} proved maximal hypoellipticity for left invariant scalar Rockland differential operators on graded nilpotent Lie groups, thus confirming a conjecture due to Rockland \cite{R78}.
For contact and (more generally) Heisenberg manifolds, a pseudodifferential calculus has been developed independently by Beals--Greiner \cite{BG88} and Taylor \cite{T84}, see also \cite{P08}.
Rockland theorems for Heisenberg manifolds can be found in \cite[Theorem~8.4]{BG88} or \cite[Theorem~5.4.1]{P08}.
These investigations can be traced back to the work of Kohn \cite{K65}, Boutet de Monvel \cite{B74}, and Folland--Stein \cite{FS74} on CR manifolds, cf.\ the introduction of \cite{BG88} for further historical comments.
A pseudodifferential calculus for general filtered manifolds was first described by Melin \cite{M82}.
In his unpublished manuscript Melin shows \cite[Theorem~7.2]{M82} that a scalar Rockland differential operator admits a parametrix in his calculus.
More recent constructions of pseudodifferential calculi on filtered manifolds are based on the Heisenberg tangent groupoid \cite{EY17,EY19,CP19,HH18,M21} and the idea of essential homogeneity introduced in \cite{DS14}.
In \cite[Theorem~2.5(d)]{CGGP92} one finds a scalar Rockland theorem for graded nilpotent Lie groups which suffices to study the flat models in parabolic geometry as well as topologically stable \cite{P16} structures like contact and Engel manifolds.
For general systems of (pseudo)differential operators on arbitrary filtered manifolds the Rockland theorem may be found in \cite[Theorem~A]{DH17}.

Using the parametrix of the heat operator and the Heisenberg pseudodifferential calculus, one can extend the results on the structure of complex powers \cite{MP49,S67,BGS84,P08} to (formally) selfadjoint positive Rockland differential operators over general filtered manifolds, see \cite[Corollary~2]{DH19}.
In particular, the zeta function of such an operator admits a meromorphic extension to the entire complex plane which is holomorphic at zero.
Hence, its derivative at zero may be used to define regularized determinants for these operators.

\subsection{Rockland complexes}\label{SS:Rockland}

Consider a finite complex of differential operators over a closed filtered manifold $M$,
\begin{equation}\label{E:ED}
	\cdots\to\Gamma^\infty(E^{q-1})\xrightarrow{D_{q-1}}
	\Gamma^\infty(E^q)\xrightarrow{D_q}
	\Gamma^\infty(E^{q+1})\to\cdots
\end{equation}
which are of Heisenberg order $k_q\geq1$, respectively.
Here $E^q$ are smooth vector bundles over $M$, only finitely many of these vector bundles are non-zero, and $D_qD_{q-1}=0$ for all $q$.
We assume the sequence is Rockland \cite[Definition~2.14]{DH17}, i.e., the Heisenberg principal symbol sequence
\[
	\cdots\to 
	C^\infty(\mathcal T_xM,E^{q-1}_x)\xrightarrow{\sigma^{k_{q-1}}_x(D_{q-1})}
	C^\infty(\mathcal T_xM,E^q_x)\xrightarrow{\sigma^{k_q}_x(D_q)}
	C^\infty(\mathcal T_xM,E^{q-1}_x)\to\cdots
\]
becomes exact in every non-trivial irreducible unitary representation of $\mathcal T_xM$.

Fix a volume density $\mu$ on $M$ and fiber wise Hermitian inner products $h_q$ on the bundles $E^q$, and consider the associated $L^2$ inner product on $\Gamma^\infty(E^q)$,
\begin{equation}\label{E:L2IP}
	\llangle\phi,\psi\rrangle_{E^q}:=\int_Mh_q(\phi,\psi)\mu,
	\qquad\phi,\psi\in\Gamma^\infty(E^q).
\end{equation}
Let $D_q^*\colon\Gamma^\infty(E^{q+1})\to\Gamma^\infty(E^q)$ denote the formal adjoint of $D_q$, that is
\begin{equation}\label{E:deltai}
	\llangle D_q^*\phi,\psi\rrangle_{E^q}
	=\llangle\phi,D_q\psi\rrangle_{E^{q+1}}
\end{equation}
for all $\phi\in\Gamma^\infty(E^{q+1})$ and $\psi\in\Gamma^\infty(E^q)$.

Fix natural numbers $a_q\in\N$ such that 
\begin{equation}\label{E:kiai}
	k_{q-1}a_{q-1}=k_qa_q
\end{equation} 
for all $q$, and put
\begin{equation}\label{E:kkiai}
	\kappa:=k_qa_q.
\end{equation}
Then the Rumin--Seshadri operators,
\begin{equation}\label{E:Deltai}
	\Delta_q\colon\Gamma^\infty(E^q)\to\Gamma^\infty(E^q),\qquad
	\Delta_q:=(D_{q-1}D_{q-1}^*)^{a_{q-1}}+(D_q^*D_q)^{a_q}.
\end{equation}
are all of Heisenberg order $2\kappa$.
These operators generalize the classical Hodge Laplacians as well as the Laplacians associated with the Rumin complex in \cite[Section 2.3]{RS12} which are of Heisenberg order two or four.

\begin{lemma}[{\cite[Lemma~2.18]{DH17}}]
The Rumin--Seshadri operator $\Delta_q$ is Rockland.
\end{lemma}

It will be convenient to consider the operator $\Delta:=\bigoplus_q\Delta_q$ acting on sections of the graded vector bundle $E:=\bigoplus_qE^q$.
Note that $\Delta$ has Heisenberg order $2\kappa$.

\begin{lemma}[{\cite[Lemma~1]{DH20}}]\label{L:Desa}
The Rumin--Seshadri operator $\Delta$ is essentially self-adjoint with compact resolvent on $L^2(E)$.
Moreover, $\Delta\geq0$.
\end{lemma}

Using the spectral theorem \cite[Section~VI§5.3]{K95} we obtain a strongly differentiable semi-group $e^{-t\Delta}$ for $t\geq0$.
Let 
\begin{equation}\label{E:n}
	n:=-\sum_pp\cdot\rk(\mathfrak t^pM)
\end{equation}
denote the homogeneous dimension of $M$.
Note that $n\geq0$ since with our convention the grading of $\mathfrak tM=\bigoplus_p\mathfrak t^pM$ is concentrated in negative degrees.
We will denotes the line bundle of $1$-densities on $M$ by $|\Lambda|$.

\begin{lemma}[{\cite[Theorem~1]{DH20}}]\label{L:asyexp}
In this situation $e^{-t\Delta}$ is a smoothing operator for each $t>0$, and the corresponding heat kernels $k_t\in\Gamma^\infty(E\boxtimes(E^*\otimes|\Lambda|))$ depend smoothly on $t>0$.
Furthermore, as $t\to0$, we have an asymptotic expansion
\begin{equation}\label{E:heatkernasymp}
k_t(x,x)\sim\sum_{j=0}^\infty t^{(j-n)/2\kappa}p_j(x)
\end{equation}
where $p_j\in\Gamma^\infty(\eend(E)\otimes|\Lambda|)$.
Moreover, $p_j(x)=0$ for all odd $j$.
\end{lemma}


For $A\in\Gamma^\infty(\eend(E))$ we thus obtain an asymptotic expansion as $t\to0$,
\[
	\str\bigl(Ae^{-t\Delta}\bigr)
	\sim\sum_{j=0}^\infty t^{(j-n)/2\kappa}\int_M\str\bigl(Ap_j\bigr),
\]
where $\str$ denotes the graded trace.
In particular,
\begin{equation}\label{E:strAetD}
	\LIM_{t\to0}\str\bigl(Ae^{-t\Delta}\bigr)
	=\int_M\str(Ap_n)
\end{equation}
where $\LIM$ denotes the constant term in the asymptotic expansion.

\begin{lemma}\label{L:Euler}
The graded heat trace $\str\bigl(e^{-t\Delta}\bigr)$ is constant in $t$ and
\begin{equation}\label{E:lit}
	\chi(E,D)
	=\LIM_{t\to0}\str\bigl(e^{-t\Delta}\bigr)
	=\int_M\str(p_n).
\end{equation}
Here $\chi(E,D)$ denotes the Euler characteristics of the complex in \eqref{E:ED}.
\end{lemma}

\begin{proof}
We adapt the classical argument.
Defining $\sigma_q\colon\Gamma^\infty(E^q)\to\Gamma^\infty(E^{q-1})$ by
\begin{equation}\label{E:sigma}
	\sigma_q
	:=(D_{q-1}^*D_{q-1})^{a_{q-1}-1}D_{q-1}^*
	=D_{q-1}^*(D_{q-1}D_{q-1}^*)^{a_{q-1}-1}
\end{equation}
we obtain the graded commutator relation
\begin{equation}\label{E:DDs}
	\Delta=[D,\sigma].
\end{equation}
Combining this with $[D,e^{-t\Delta}]=0$ and the fact that the graded trace vanishes on graded commutators, we see that the graded heat trace is constant in $t$: 
\begin{multline}\label{E:lit1}
	\tfrac\partial{\partial t}\str\bigl(e^{-t\Delta}\bigr)
	=-\str\bigl(\Delta e^{-t\Delta}\bigr)
	\\=-\str\bigl([D,\sigma]e^{-t\Delta}\bigr)
	=-\str\bigl([D,\sigma e^{-t\Delta}]\bigr)
	=0.
\end{multline}
Moreover, with respect to the trace norm we have \cite[Lemma~5]{DH20}
\begin{equation}\label{E:limetDP}
	\lim_{t\to\infty}e^{-t\Delta}
	=P
\end{equation}
where $P$ denotes the spectral (orthogonal) projection onto the kernel of $\Delta$.
In view of Lemma~\ref{L:Desa}, this kernel is finite dimensional.
By Hodge theory \cite[Corollary~2.20]{DH17}, the inclusion $\img(P)\subseteq\Gamma^\infty(E)$ induces an isomorphism in cohomology.
Hence, 
\begin{equation}\label{E:lit2}
	\lim_{t\to\infty}\str\bigl(e^{-t\Delta}\bigr)
	=\str(P)
	=\chi\bigl(\img(P),D\bigr)
	=\chi(E,D).
\end{equation}
Putting $A=\id$ in \eqref{E:strAetD} and combining the resulting equation with \eqref{E:lit1} and \eqref{E:lit2} we obtain \eqref{E:lit}.
\end{proof}

\subsection{Analytic torsion}\label{SS:anator}

For $\lambda\geq0$, let $P_\lambda$ denote the spectral projection corresponding to the eigenvalues at most $\lambda$, and let $Q_\lambda:=\id-P_\lambda$ denote the complementary projection.
Fix numbers $N_q$ such that 
\begin{equation}\label{E:Niik}
	N_{q+1}-N_q=k_q
\end{equation}
for all $q$.
We let $N\in\Gamma(\eend(E))$ denote the operator given by multiplication with $N_q$ on $E^q$.
For $\lambda\geq0$ and $\Re(s)>n/2\kappa$, we put
\begin{equation}\label{E:zetalambda}
	\zeta_\lambda(s)
	:=\str\bigl(NQ_\lambda\Delta^{-s}\bigr)
	=\frac1{\Gamma(s)}\int_0^\infty t^{s-1}\str\bigl(NQ_\lambda e^{-t\Delta}\bigr)dt.
\end{equation}

According to \cite[Corollary~2]{DH20} this zeta function admits a meromorphic continuation to the entire complex plane which is holomorphic at $s=0$.
Moreover,
\begin{equation}\label{E:zeta0}
	\zeta_\lambda(0)
	=\LIM_{t\to0}\str\bigl(NQ_\lambda e^{-t\Delta}\bigr)
	=\LIM_{t\to0}\str\bigl(Ne^{-t\Delta}\bigr)-\chi'_\lambda
\end{equation}
where $\LIM$ denotes the constant term in the asymptotic expansion, and
\[
	\chi'_\lambda
	:=\str(NP_\lambda)=\sum_q(-1)^qN_q\rk(P_{\lambda,q}).
\]
Putting $A=N$ in \eqref{E:strAetD} and using \eqref{E:zeta0} we obtain
\begin{equation}\label{E:zeta00}
	\zeta_\lambda(0)
	=-\chi'_\lambda+\int_M\str(Np_n).
\end{equation}

Recall that $\img(P_\lambda)$ is a finite dimensional subcomplex of $(\Gamma^\infty(E),D)$.
Moreover, the inclusion induces a canonical isomorphism in cohomology,
\[
	H^*\bigl(\img(P_\lambda),D\bigr)=H^*(E,D).
\]
The torsion of finite dimensional complexes, see \cite[Section~a]{BGS88} or \cite[Section~Ia]{BZ92}, provides a canonical isomorphism of graded determinant lines \cite{KM76}
\[
	\sdet(\img(P_\lambda))
	=\sdet\bigl(H^*\bigl(\img(P_\lambda),D\bigr)\bigr).
\]
Combining the latter two identifications, we obtain a canonical isomorphism
\begin{equation}\label{E:toriso}
	\sdet(\img(P_\lambda))
	=\sdet(H^*(E,D)).
\end{equation}
The $L^2$ inner product on $\Gamma^\infty(E)$ restricts to a graded inner product on $\img(P_\lambda)$ and induces an inner product on the graded determinant line $\sdet(\img(P_\lambda))$.
Via \eqref{E:toriso} this corresponds to an inner product on the line $\sdet (H^*(E,D))$. 
We will denote the corresponding norm by $\|-\|^{\sdet(H^*(E,D))}_{[0,\lambda],h,\mu}$.
Following \cite{BGS88,BZ92}, we define the \emph{analytic torsion of the Rockland complex $(E,D)$} to be the norm 
\begin{equation}\label{E:defRSmetric}
	\|-\|^{\sdet(H^*(E,D))}_{h,\mu}
	:=\exp\bigl(-\tfrac1{2\kappa}\zeta'_\lambda(0)\bigr)\cdot\|-\|^{\sdet(H^*(E,D))}_{[0,\lambda],h,\mu}
\end{equation}
on the graded determinant line 
\[
	\sdet(H^*(E,D))=\bigotimes_q\bigl(\det H^q(E,D)\bigr)^{(-1)^q}.
\]

A priori, the analytic torsion of a Rockland complex $(E,D)$ depends on the graded fiber wise Hermitian inner product $h$ on $E$, the volume density $\mu$ on $M$, the choice of numbers $a_q$ satisfying \eqref{E:kiai}, the choice of numbers $N_q$ satisfying \eqref{E:Niik}, and the choice of a spectral cutoff $\lambda\geq0$.
It is fairly easy to see that the analytic torsion is actually independent of $a_q$, $N_q$ and $\lambda$, see Lemmas~\ref{L:lambda}, \ref{L:Ni} and \ref{L:ai} below.
The dependence on $h_q$ and $\mu$ will be discussed in Section~\ref{SS:var}, see Theorem~\ref{T:var} below.

\begin{lemma}\label{L:lambda}
The analytic torsion of a Rockland complex, see \eqref{E:defRSmetric}, does not depend on the choice of the spectral cutoff $\lambda\geq0$.
\end{lemma}

\begin{proof}
Suppose $0\leq\lambda_1\leq\lambda_2$ and consider the finite rank spectral projection $P_{(\lambda_1,\lambda_2]}:=P_{\lambda_2}-P_{\lambda_1}$.
Clearly, $C:=\img(P_{(\lambda_1,\lambda_2]})$ is a finite dimensional graded vector space invariant under $D$, $D^*$, and $\Delta$.
Moreover, $\ker\bigl(D|_{C^q}\bigr)=\img(D|_{C^{q-1}}\bigr)$ and $\ker\bigl(D^*|_{C^q}\bigr)=\img(D^*|_{C^{q+1}}\bigr)$.
From the decomposition of complexes 
$$
\img(P_{\lambda_2})=\img(P_{\lambda_1})\oplus C
$$ 
we obtain
\begin{equation}\label{E:lambda1}
\frac{\|-\|^{\sdet(H^*(E,D))}_{[0,\lambda_2],h,\mu}}{\|-\|^{\sdet(H^*(E,D))}_{[0,\lambda_1],h,\mu}}
=T(C,D)
\end{equation}
where the right hand side denotes the torsion of the finite dimensional acyclic complex $(C,D)$, see \cite[Proposition~1.5]{BZ92}.

Consider the decomposition of graded vector spaces
$$
C=K\oplus L
$$
where $K^q:=\img\bigl(P_{(\lambda_1,\lambda_2]}D_{q-1}\bigr)$ and $L^q:=\img\bigl(P_{(\lambda_1,\lambda_2]}D_q^*\bigr)$.
A well known expression for the torsion of an acyclic complex, see \cite[Proposition~2.7]{RS12}, gives
\begin{equation}\label{E:lambda2}
	T(C,D)
	=\sdet\bigl(D^*D|_L\bigr)^{-1/2}.
\end{equation}

Clearly, $\det\bigl(DD^*|_{K^q}\bigr)=\det\bigl(D^*D|_{L^{q-1}}\bigr)$ and thus
\begin{multline*}
	\det\bigl(\Delta|_{C^q}\bigr)
	=\det\bigl(DD^*|_{K^q}\bigr)^{a_{q-1}}\det\bigl(D^*D|_{L^q}\bigr)^{a_q}
	\\=\det\bigl(D^*D|_{L^{q-1}}\bigr)^{a_{q-1}}\det\bigl(D^*D|_{L^q}\bigr)^{a_q}.
\end{multline*}
Using \eqref{E:Niik} and \eqref{E:kkiai} this gives
\[
	\det\bigl(\Delta|_{C^q}\bigr)^{N_q}
	=\det\bigl(D^*D|_{L^{q-1}}\bigr)^\kappa
	\det\bigl(D^*D|_{L^{q-1}}\bigr)^{N_{q-1}a_{q-1}}
	\det\bigl(D^*D|_{L^q}\bigr)^{N_qa_q}.
\]
A telescoping argument yields
\begin{equation}\label{E:telescope}
	\prod_q\det\bigl(\Delta|_{C^q}\bigr)^{(-1)^qN_q}
	=\sdet\bigl(D^*D|_L\bigr)^{-\kappa}.
\end{equation}

Using $Q_{\lambda_2}-Q_{\lambda_1}=-P_{(\lambda_1,\lambda_2]}$ we obtain, see \eqref{E:zetalambda},
\[
	\zeta_{\lambda_2}(s)-\zeta_{\lambda_1}(s)
	=-\str\bigl(NP_{(\lambda_1,\lambda_2]}\Delta^{-s}\bigr).
\]
Hence,
\[
	\zeta'_{\lambda_2}(s)-\zeta'_{\lambda_1}(s)
	=\str\bigl(NP_{(\lambda_1,\lambda_2]}\Delta^{-s}\log\Delta\bigr)
\]
and
\[
	\zeta'_{\lambda_2}(0)-\zeta'_{\lambda_1}(0)
	=\str\bigl(NP_{(\lambda_1,\lambda_2]}\log\Delta\bigr)
	=\sum_q(-1)^qN_q\log\det\bigl(\Delta|_{C^q}\bigr).
\]
Consequently,
\begin{equation}\label{E:lambda4}
	\frac{\exp\bigl(-\tfrac1{2\kappa}\zeta'_{\lambda_2}(0)\bigr)}{\exp\bigl(-\tfrac1{2\kappa}\zeta'_{\lambda_1}(0)\bigr)}
	=\left(\prod_q\det\bigl(\Delta|_{C^q}\bigr)^{(-1)^qN_q}\right)^{-1/2\kappa}.
\end{equation}
Combining \eqref{E:lambda1}, \eqref{E:lambda2}, \eqref{E:telescope}, and \eqref{E:lambda4} we obtain 
\[
	\frac{\exp\bigl(-\tfrac1{2\kappa}\zeta'_{\lambda_2}(0)\bigr)\cdot\|-\|^{\sdet(H^*(E,D))}_{[0,\lambda_2],h,\mu}}{\exp\bigl(-\tfrac1{2\kappa}\zeta'_{\lambda_1}(0)\bigr)\cdot\|-\|^{\sdet(H^*(E,D))}_{[0,\lambda_1],h,\mu}}=1,
\]
whence the lemma.
\end{proof}

\begin{lemma}\label{L:Ni}
The analytic torsion of a Rockland complex, see \eqref{E:defRSmetric}, does not depend on the choice of numbers $N_q$, see \eqref{E:Niik}.
\end{lemma}

\begin{proof}
If $\tilde N_q$ is another choice such that $\tilde N_{q+1}-\tilde N_q=k_q$, then there exists a constant $c$ such that $\tilde N_q-N_q=c$ for all $q$.
Recall that the graded commutator relations $\Delta=[D,\sigma]$, see \eqref{E:DDs}, and $[D,e^{-t\Delta}]=0=[Q_\lambda,D]$ yield
\begin{multline*}
	\tfrac\partial{\partial t}\str\bigl(Q_\lambda e^{-t\Delta}\bigr)
	=-\str\bigl(Q_\lambda\Delta e^{-t\Delta}\bigr)
	\\=-\str\bigl(Q_\lambda[D,\sigma]e^{-t\Delta}\bigr)
	=-\str\bigl([D,Q_\lambda\sigma e^{-t\Delta}]\bigr)
	=0.
\end{multline*}
Since $\lim_{t\to\infty}\tr\bigl(Q_{\lambda,q}e^{-t\Delta_q}\bigr)=0$, see \eqref{E:limetDP}, we conclude $\str\bigl(Q_\lambda e^{-t\Delta}\bigr)=0$ for all $t>0$ and all $\lambda\geq0$.
Hence, for $\Re(s)>n/2\kappa$ we obtain
\[
	\str\bigl(\tilde NQ_\lambda\Delta^{-s}\bigr)
	-\str\bigl(NQ_\lambda\Delta^{-s}\bigr)
	=\frac c{\Gamma(s)}\int_0^\infty t^{s-1}\str\bigl(Q_\lambda e^{-t\Delta}\bigr)dt=0.
\]
By analytic continuation, this remains true for $s=0$, whence the lemma.
\end{proof}

\begin{lemma}\label{L:ai}
The analytic torsion of a Rockland complex, see \eqref{E:defRSmetric}, does not depend on the choice of numbers $a_q$, see \eqref{E:kiai}.
\end{lemma}

\begin{proof}
Suppose $\tilde a_q\in\N$ is another choice of positive integers satisfying \eqref{E:kiai}, that is, $k_{q-1}\tilde a_{q-1}=k_q\tilde a_q$ for all $q$.
Then there exist positive integers $r,\tilde r\in\N$ such that $ra_q=\tilde r\tilde a_q$ for all $q$. 
Indeed, we may use $r:=\tilde a_{q_0}$ and $\tilde r:=a_{q_0}$ for some fixed $q_0$.
W.l.o.g.\ we may assume $\tilde r=1$ and thus
$$
\tilde a_q=ra_q
$$
for all $q$.
Writing $\tilde\kappa:=k_q\tilde a_q$ we find, see \eqref{E:kkiai}, 
\begin{equation}\label{E:QQQ:kappa}
	\tilde\kappa=r\kappa.
\end{equation}
Denoting the corresponding Laplacian by $\tilde\Delta_q:=(D_{q-1}D^*_{q-1})^{\tilde a_{q-1}}+(D^*_qD_q)^{\tilde a_q}$, see \eqref{E:deltai}, and using $(D^*)^2=0=D^2$, we find
\begin{equation}\label{E:QQQ:Delta}
	\tilde\Delta_q=\Delta_q^r.
\end{equation}
Using the spectral cutoff
\[
	\tilde\lambda:=\lambda^r
\] 
and denoting the spectral projections of $\tilde\Delta$ by $\tilde Q$, we have
\begin{equation}\label{E:QQQ:Q}
	\tilde Q_{\tilde\lambda}=Q_\lambda.
\end{equation}
Using $(\Delta^r)^{-s}=\Delta^{-rs}$, \eqref{E:QQQ:Delta} and \eqref{E:QQQ:Q} we find, see \eqref{E:zetalambda},
\[
	\tilde\zeta_{\tilde\lambda}(s)=\zeta_\lambda(rs)
\]
for all $s$.
Differentiating and using \eqref{E:QQQ:kappa} we find
\begin{equation}\label{E:QQQ:zetap}
	\tfrac1{\tilde\kappa}\tilde\zeta'_{\tilde\lambda}(0)
	=\tfrac1\kappa\zeta'_\lambda(0).
\end{equation}
In view of \eqref{E:QQQ:Q} we have $\img(\tilde P_{\tilde\lambda})=\img(P_\lambda)$ and thus
\begin{equation}\label{E:QQQ:norm}
	\widetilde{\|-\|}^{\sdet(H^*(E,D))}_{[0,\tilde\lambda],h,\mu}
	=\|-\|^{\sdet(H^*(E,D))}_{[0,\lambda],h,\mu}.
\end{equation}
Combining \eqref{E:QQQ:zetap} with \eqref{E:QQQ:norm} and Lemma~\ref{L:lambda} we see that the analytic torsion with the choices $a_q$ and $\tilde a_q$ coincide, cf.~\eqref{E:defRSmetric}.
\end{proof}

\begin{remark}\label{R:Z2}
The analytic torsion in \eqref{E:defRSmetric} may be expressed in a way which only uses the $\mathbb Z_2$-grading of the Rockland complex and does not require the choice of numbers $a_q$ as in \eqref{E:kiai} and $N_q$ as in \eqref{E:Niik}.
More precisely, for $\lambda\geq0$ we have
\begin{equation}\label{E:zetaZ2}
	\tfrac1{2\kappa}\zeta'_\lambda(0)=-\tfrac12\tfrac\partial{\partial s}\big|_{s=0}\str_\lambda\bigl((D^*D)^{-s}\bigr)=:\log\sdet_\lambda\sqrt{D^*D}.
\end{equation}
Here $\str_\lambda\bigl((D^*D)^{-s}\bigr)$ denotes the graded trace, disregarding all eigenvalues $\leq\lambda$, in particular, disregarding the eigenvalue zero which will have infinite multiplicity in general.
To see this, note that $D^2=0$ gives, see \eqref{E:Deltai}
\[
	\tr(Q_\lambda\Delta_q^{-s})=\tr_\lambda\bigl((D^*_qD_q)^{-a_qs}\bigr)+\tr_\lambda\bigl((D_{q-1}D_{q-1}^*)^{-a_{q-1}s}\bigr),\qquad\Re(s)>n/2\kappa.
\]
Moreover, since $D_{q-1}D_{q-1}^*$ and $D^*_{q-1}D_{q-1}$ have the same spectrum 
\[
	\tr_\lambda\bigl((D_{q-1}D_{q-1}^*)^{-a_{q-1}s}\bigr)=\tr_\lambda\bigl((D^*_{q-1}D_{q-1})^{-a_{q-1}s}\bigr),\qquad\Re(s)>n/2\kappa.
\]
Using the latter two displayed equations, and proceeding by induction on $q$, one readily concludes, from the corresponding fact for $\tr(Q_\lambda\Delta_q^{-s})$, that the functions $\tr_\lambda\bigl((D^*_qD_q)^{-a_qs}\bigr)$ and $\tr_\lambda\bigl((D_{q-1}D^*_{q-1})^{-a_qs}\bigr)$ admit meromorphic extensions which are holomorphic at $s=0$.
Using \eqref{E:Niik} we obtain
\[
	-\str(NQ_\lambda\Delta^{-s})=\sum_q(-1)^q\tr_\lambda\bigl(k_q(D^*_qD_q)^{-a_qs}\bigr).
\]
Using \eqref{E:zetalambda} and \eqref{E:kkiai} this gives
\begin{equation*}\label{E:zetaZ2a}
	\tfrac1{\kappa}\zeta_\lambda'(s)=\sum_q(-1)^q\tr_\lambda\bigl((D_q^*D_q)^{-a_qs}\log D^*_qD_q\bigr).
\end{equation*}
Similarly,
\begin{equation*}\label{E:zetaZ2b}
	-\tfrac\partial{\partial s}\str_\lambda\bigl((D^*D)^{-s}\bigr)=\sum_q(-1)^q\tr_\lambda\bigl((D_q^*D_q)^{-s}\log D^*_qD_q\bigr).
\end{equation*}
Combining the latter two equations with the obvious equality
$$
	\tr_\lambda\bigl((D_q^*D_q)^{-a_qs}\log D^*_qD_q\bigr)\Big|_{s=0}=\tr_\lambda\bigl((D_q^*D_q)^{-s}\log D^*_qD_q\bigr)\Big|_{s=0},
$$
obtained via analytic continuation, we arrive at the formula in \eqref{E:zetaZ2}.
\end{remark}

\begin{remark}\label{R:signs}
The analytic torsion of a Rockland complex remains unchanged if the differentials $D_q$ are replaced with $\pm D_q$ where the signs may depend on $q$.
\end{remark}

\subsection{Dependence on the metric}\label{SS:var}

Suppose the volume density $\mu_u$ on $M$ and the graded Hermitian inner product $h_u$ on $E$ depend smoothly on a real parameter $u$.
Let $D^*_u$ and $\Delta_u$ denote the corresponding family of operators.
For real $u$ and $v$, we let $G_{v,u}\in\Gamma^\infty(\Aut(E))$,
\begin{equation}\label{E:Gdef}
	G_{v,u}:=h_v^{-1}h_u\frac{\mu_u}{\mu_v}
\end{equation}
denote the unique vector bundle automorphism such that, cf.~\eqref{E:L2IP}, 
\begin{equation}\label{E:inneru}
	\llangle\phi,\psi\rrangle_u
	=\llangle G_{v,u}\phi,\psi\rrangle_v
\end{equation}
for all $\phi,\psi\in\Gamma^\infty(E)$.
Clearly, $G_{v,u}$ preserves the decomposition $E=\bigoplus_qE^q$.
Using \eqref{E:deltai} we immediately see that $D^*_u$ varies by conjugation, i.e.,
\begin{equation}\label{E:deltau}
	D^*_u
	=G_{v,u}^{-1}D^*_vG_{v,u}.
\end{equation}

The aim of this section is to establish the following variational formula generalizing well known results for the Ray--Singer torsion \cite{RS71,BZ92} and the Rumin--Seshadri torsion, see \cite[Corollary~3.7]{RS12}.

\begin{theorem}\label{T:var}
In this situation we have
\[
	\tfrac\partial{\partial u}\log\|-\|^{\sdet(H^*(E,D))}_{h_u,\mu_u}
	=\frac12\int_M\str\bigl((\dot h_u+\dot\mu_u)p_{u,n}\bigr).
\]
where $\dot h_u:=h_u^{-1}\frac\partial{\partial u}h_u$, $\dot\mu_u:=\mu_u^{-1}\frac\partial{\partial u}\mu_u$, and $p_{u,n}\in\Gamma^\infty(\eend(E)\otimes|\Lambda|)$ denotes the constant term in the heat kernel expansion associated with the Rumin--Seshadri operator $\Delta_u$, see \eqref{E:heatkernasymp}.
\end{theorem}

\begin{remark}
If the homogeneous dimension $n$ is odd, then Theorem~\ref{T:var} implies that the analytic torsion of a Rockland complex does not depend on the choice of $\mu$ and $h$, for in this case $p_{u,n}=0$ according to Lemma~\ref{L:asyexp}.
In the classical case this has already been observed by Ray and Singer \cite{RS71}. 
\end{remark}

Let $\dot G_u\in\Gamma^\infty(\eend(E))$ be defined by
\begin{equation}\label{E:dGdhdm}
	\dot G_u
	:=G_{v,u}^{-1}\tfrac\partial{\partial u}G_{v,u}
	=\tfrac\partial{\partial w}\big|_{w=u}G_{u,w}
	=\dot h_u+\dot\mu_u,
\end{equation}
where the middle equality is a consequence of the cocycle relation $G_{w,v}G_{v,u}=G_{w,u}$, and the last equality follows by differentiating \eqref{E:Gdef}.

The following can be traced back to the original work of Ray--Singer \cite[Theorem~2.1]{RS71}, see also \cite[Theorem~5.6]{BZ92} or \cite[Section~3.2]{RS12}.

\begin{lemma}\label{L:var}
In this situation we have
\begin{equation}\label{E:var}
	\tfrac\partial{\partial u}\str\bigl(Ne^{-t\Delta_u}\bigr)
	=\kappa t\tfrac\partial{\partial t}\str\bigl(\dot G_ue^{-t\Delta_u}\bigr).
\end{equation}
\end{lemma}

\begin{proof}
Let us define $\sigma_{u,q}\colon\Gamma^\infty(E^q)\to\Gamma^\infty(E^{q-1})$ by
\begin{equation}\label{E:sigmai}
	\sigma_{u,q}
	:=(D^*_{u,q-1}D_{q-1})^{a_{q-1}-1}D^*_{u,q-1}
	=D^*_{u,q-1}(D_{q-1}D^*_{u,q-1})^{a_{q-1}-1}
\end{equation}
and $s_{u,q}\colon\Gamma^\infty(E^q)\to\Gamma^\infty(E^{q+1})$ by
\begin{equation}\label{E:sui}
	s_{u,q}
	:=(D_qD^*_{u,q})^{a_q-1}D_q=D_q(D^*_{u,q}D_q)^{a_q-1}.
\end{equation}
Hence, we have graded commutator relations
\begin{equation}\label{E:Dddss}
	\Delta_u=[D,\sigma_u]=[s_u,D^*_u].
\end{equation}
Using Duhamel's formula one shows 
\begin{equation}\label{E:trDui}
	\tfrac\partial{\partial u}\tr\bigl(e^{-t\Delta_{u,q}}\bigr)
	=-t\tr\bigl(\dot\Delta_{u,q}e^{-t\Delta_{u,q}}\bigr)
\end{equation}
where $\dot\Delta_u:=\frac\partial{\partial u}\Delta_u$, see \cite[Lemma~3.5]{RS12} or \cite[Corollary~2.50]{BGV92}.
Consequently,
\begin{equation}\label{E:var1}
	\tfrac\partial{\partial u}\str\bigl(Ne^{-t\Delta_u}\bigr)
	=-t\str\bigl(N\dot\Delta_ue^{-t\Delta_u}\bigr).
\end{equation}
Put $\dot\sigma_u:=\frac\partial{\partial u}\sigma$ and let $k$ denote the operator given by multiplication with $k_q$ on $\Gamma^\infty(E^q)$.
Then we have graded commutator relations $\dot\Delta_u=[D,\dot\sigma_u]$ and $[N,D]=Dk$, see \eqref{E:Dddss} and \eqref{E:Niik}, respectively.
Combining this with $[D,e^{-t\Delta_u}]=0$ and the fact that the graded trace vanishes on graded commutators we obtain
\begin{equation}\label{E:var2}
	\str\bigl(N\dot\Delta_ue^{-t\Delta_u}\bigr)
	=\str\bigl(Dk\dot\sigma_ue^{-t\Delta_u}\bigr).
\end{equation}
Put $\dot D^*_u:=\frac\partial{\partial u}D^*_u$.
From \eqref{E:sigmai} we get
\[
	\dot\sigma_{u,q}
	=\sum_{r=0}^{a_{q-1}-1}(D^*_{u,q-1}D_{q-1})^r\dot D^*_{u,q-1}(D_{q-1}D^*_{u,q-1})^{a_{q-1}-1-r}.
\]
Since $[D_{q-1}D^*_{u,q-1},e^{-t\Delta_{u,q}}]=0$ and the trace vanishes on commutators this yields 
\[
	\tr\bigl(D_{q-1}k_{q-1}\dot\sigma_{u,q}e^{-t\Delta_{u,q}}\bigr)
	=\kappa\tr\bigl(s_{u,{q-1}}\dot D^*_{u,q-1}e^{-t\Delta_{u,q}}\bigr),
\]
see also \eqref{E:kkiai} and \eqref{E:sui}.
Consequently,
\begin{equation}\label{E:var3}
	\str\bigl(Dk\dot\sigma_ue^{-t\Delta_u}\bigr)
	=\kappa\str\bigl(s_u\dot D^*_ue^{-t\Delta_u}\bigr).
\end{equation}
Differentiating \eqref{E:deltau} and using \eqref{E:dGdhdm}, we obtain $\dot D^*_u=[D^*_u,\dot G_u]$.
Combining this with $[D^*_u,e^{-t\Delta_u}]=0$ and $[s_u,D^*_u]=\Delta_u$, see \eqref{E:Dddss}, and using the fact that the graded trace vanishes on graded commutators we obtain
\begin{equation}\label{E:var4}
	\str\bigl(s_u\dot D^*_ue^{-t\Delta_u}\bigr)
	=\str\bigl(\dot G_u\Delta_ue^{-t\Delta_u}\bigr).
\end{equation}
Clearly,
\begin{equation}\label{E:var5}
	\tfrac\partial{\partial t}\str\bigl(\dot G_ue^{-t\Delta_u}\bigr)
	=-\str\bigl(\dot G_u\Delta_ue^{-t\Delta_u}\bigr).
\end{equation}
Combining equations \eqref{E:var1}, \eqref{E:var2}, \eqref{E:var3}, \eqref{E:var4} and \eqref{E:var5}, we obtain \eqref{E:var}.
\end{proof}

The following generalizes \cite[Theorem~3.4]{RS12}.

\begin{lemma}
For $\lambda\geq0$ not in the spectrum of $\Delta_u$ and $\Re(s)>n/2\kappa$ we have
\begin{equation}\label{E:varzeta}
	\tfrac\partial{\partial u}\tfrac1\kappa\zeta_{u,\lambda}(s)
	=-\frac s{\Gamma(s)}\int_0^\infty t^{s-1}\str\bigl(\dot G_uQ_{u,\lambda}e^{-t\Delta_u}\bigr)dt.
\end{equation}
Moreover, 
\begin{equation}\label{E:dduzeta}
	\tfrac\partial{\partial u}\tfrac1\kappa\zeta_{u,\lambda}(0)
	=0
\end{equation} 
and
\begin{equation}\label{E:varzetap}
	\tfrac\partial{\partial u}\tfrac1\kappa\zeta'_{u,\lambda}(0)
	=\str\bigl(\dot G_uP_{u,\lambda}\bigr)-\int_M\str\bigl(\dot G_up_{u,n}\bigr).
\end{equation}
\end{lemma}

\begin{proof}
Since $\Delta_{u,q}$ commutes with $Q_{u,\lambda,q}$, Duhamel's formula gives
\begin{equation}\label{E:trQD}
	\tfrac\partial{\partial u}\tr\bigl(Q_{u,\lambda,q}e^{-t\Delta_{u,q}}\bigr)
	=\tr\bigl(\dot Q_{u,\lambda,q}e^{-t\Delta_{u,q}}\bigr)
	-t\tr\bigl(Q_{u,\lambda,q}\dot\Delta_{u,q}e^{-t\Delta_{u,q}}\bigr).
\end{equation}
Differentiating $Q_{u,\lambda,q}=Q_{u,\lambda,q}^2$, we get $\dot Q_{u,\lambda,q}=\dot Q_{u,\lambda,q}Q_{u,\lambda,q}+Q_{u,\lambda,q}\dot Q_{u,\lambda,q}$ and then $Q_{u,\lambda,q}\dot Q_{u,\lambda,q}Q_{u,\lambda,q}=0$.
Hence, $\dot Q_{u,\lambda,q}=Q_{u,\lambda,q}\dot Q_{u,\lambda,q}P_{u,\lambda,q}+P_{u,\lambda,q}\dot Q_{u,\lambda,q}Q_{u,\lambda,q}$ and thus
\[
	\tr\bigl(\dot Q_{u,\lambda,q}e^{-t\Delta_{u,q}}\bigr)
	=0.
\]
Combining this with \eqref{E:trQD}, we obtain the following analogue of \eqref{E:trDui}
\begin{equation}\label{E:trQD2}
	\tfrac\partial{\partial u}\tr\bigl(Q_{u,\lambda,q}e^{-t\Delta_{u,q}}\bigr)
	=-t\tr\bigl(Q_{u,\lambda,q}\dot\Delta_{u,q}e^{-t\Delta_{u,q}}\bigr).
\end{equation}

Proceeding exactly as in the proof of Lemma~\ref{L:var} and using the fact that $Q_{u,\lambda}$ commutes with $\Delta_u$, $D$ and $D^*_u$, one obtains
\begin{equation}\label{E:varlambda}
	\tfrac\partial{\partial u}\str\bigl(NQ_{u,\lambda}e^{-t\Delta_u}\bigr)
	=\kappa t\tfrac\partial{\partial t}\str\bigl(\dot G_uQ_{u,\lambda}e^{-t\Delta_u}\bigr).
\end{equation}
Using \eqref{E:zetalambda} and \eqref{E:varlambda} it is straight forward to derive \eqref{E:varzeta}.
Using the fact that 
\[
	\frac1{\Gamma(s)}\int_0^\infty t^{s-1}\str\bigl(\dot G_uQ_{u,\lambda}e^{-t\Delta_u}\bigr)dt
\]
is holomorphic at $s=0$, the equation \eqref{E:dduzeta} follows at once.
Furthermore, we get
\[
	\tfrac\partial{\partial u}\tfrac1\kappa\zeta'_{u,\lambda}(0)
	=-\LIM_{t\to0}\str\bigl(\dot G_uQ_{u,\lambda}e^{-t\Delta_u}\bigr)
	=\str\bigl(\dot G_uP_{u,\lambda}\bigr)-\LIM_{t\to0}\str\bigl(\dot G_ue^{-t\Delta_u}\bigr).
\]
Combining this with \eqref{E:strAetD} we obtain \eqref{E:varzetap}.
\end{proof}

The following can be proved as in the classical case, see for instance \cite[p~56]{BH07}. 

\begin{lemma}\label{L:varfinnorm}
If $\lambda\geq0$ is not contained in the spectrum of $\Delta_u$, then
\[
	\tfrac\partial{\partial u}\log\|-\|^{\sdet(H^*(E,D))}_{[0,\lambda],h_u,\mu_u}
	=\tfrac12\str(\dot G_uP_{u,\lambda}).
\]
\end{lemma}

Combining Lemma~\ref{L:varfinnorm} with Equations~\eqref{E:varzetap} and \eqref{E:dGdhdm}, we obtain Theorem~\ref{T:var}, see \eqref{E:defRSmetric}.

\begin{remark}\label{R:twist}
Twisting a Rockland complex as in \eqref{E:ED} with a flat vector bundle $F$ we obtain a new Rockland complex:
$$
\cdots\to\Gamma^\infty(E^{q-1}\otimes F)\xrightarrow{D_{q-1}^F}
\Gamma^\infty(E^q\otimes F)\xrightarrow{D_q^F}
\Gamma^\infty(E^{q+1}\otimes F)\to\cdots.
$$
Suppose $h^F$ is a parallel fiber wise Hermitian inner product on $F$ and equip $E^q\otimes F$ with the Hermitian metric $h_q\otimes h^F$.
Locally, the flat bundle $F$ and the parallel Hermitian metric $h^F$ can be trivialized simultaneously.
Hence, locally, we have
\[
	D^F=D\oplus\dotsb\oplus D,\quad
	(D^F)^*=D^*\oplus\dotsb\oplus D^*,\quad
	\Delta^F=\Delta\oplus\dotsm\oplus\Delta.
\]
Using locality of the heat kernel expansion, we conclude
\[
	p_j^F=p_j\otimes\id_F
\]
for all $j\in\mathbb N_0$, where $p_j^F\in\Gamma^\infty\bigl(\eend(E\otimes F)\otimes|\Lambda|\bigr)$ and $p_j\in\Gamma^\infty\bigl(\eend(E)\otimes|\Lambda|\bigr)$ denote the local quantities in the heat kernel expansions of the Rumin--Seshadri operators $\Delta^F$ and $\Delta$, respectively, see \eqref{E:heatkernasymp}.
\end{remark}

\subsection{Deformation of the filtration}\label{SS:deffil}

In this section we consider a family of closed filtered manifolds $M_u$ smoothly depending on a real parameter $u$.
More explicitly, we assume the underlying manifold $M$ remains fixed, and the filtration subbundles $T^pM_u\subseteq TM$ depend smoothly on $u$.
Suppose we have a smooth family of differential complexes $D_u$ acting between sections of vector bundles $E^q_u$ which depend smoothly on $u$,
\[
	\cdots\to\Gamma^\infty(E_u^{q-1})\xrightarrow{D_{q-1,u}}
	\Gamma^\infty(E_u^q)\xrightarrow{D_{q,u}}
	\Gamma^\infty(E_u^{q+1})\to\cdots
\]
We assume that each $D_u$ is a Rockland complex with respect to the filtration on $M_u$ and that the Heisenberg order $k_q$ of $D_{q,u}$ is independent of $u$.
We consider graded Hermitian metrics $h_u$ on $E_u=\bigoplus_qE^q_u$ and volume densities $\mu_u$ on $M$ depending smoothly on $u$, and aim at describing the dependence of the torsion $\|-\|^{\sdet(H^*(E_u,D_u))}_{h_u,\mu_u}$ on the parameter $u$, cf.~\eqref{E:defRSmetric}.

We will add the subscript $u$ to the notation used in the preceding sections to indicate the dependence on this parameter.
In particular, $\Delta_u$ denotes the Rumin--Seshadri operator associated with $E_u,D_u,h_u,\mu_u$, see \eqref{E:Deltai}, and $k_{t,u}(x,y)$ will denote the (smooth) kernel of $e^{-t\Delta_u}$.
We take the numbers $a_q$ satisfying \eqref{E:kiai} and the numbers $N_q$ as in \eqref{E:Niik} to be constant in $u$.

Suppose we have differential operators $A_{q,v,u}\colon\Gamma^\infty(E^q_u)\to\Gamma^\infty(E^q_v)$ which depend smoothly on real parameters $u$ and $v$ such that
\begin{equation}\label{E:ADDA}
	A_{v,u}D_u=D_vA_{v,u}.
\end{equation}
We assume that $A_{v,u}$ induces an isomorphism in cohomology,
\begin{equation}\label{E:HAu}
	H(A_{v,u})\colon H^*(E_u,D_u)\to H^*(E_v,D_v),
\end{equation}
and that $A_{u,u}=\id$.
In this situation we have

\begin{theorem}\label{T:var.all}
With respect to the isomorphism of determinant lines
\[
	\sdet(H(A_{v,u}))\colon\sdet(H^*(E_u,D_u))\to\sdet(H^*(E_v,D_v))
\]
induced by \eqref{E:HAu}, we have
\begin{equation}\label{E:varDuall}
	\tfrac\partial{\partial v}\big|_{v=u}\log\left((\sdet(H(A_{v,u})))^*\|-\|^{\sdet(H^*(E_v,D_v))}_{h_v,\mu_v}\right)
	=\frac12\int_M\str(\tilde p_{n,u})
\end{equation}
where $\tilde p_{n,u}\in\Gamma^\infty(\eend(E_u)\otimes|\Lambda|)$ is the constant term in the asymptotic expansion 
\begin{equation}\label{E:asym.all}
	\bigl((\dot A_u+\dot A^*_u+\dot h_u+\dot\mu_u)k_{t,u}\bigr)(x,x)\sim\sum_{j=-r_u}^\infty t^{(j-n)/2\kappa}\tilde p_{j,u}(x).
\end{equation}
Here $\dot\mu_u=\mu_u^{-1}\frac\partial{\partial u}\mu_u$, $\dot h_u=h_u^{-1}\nabla_{\frac{\partial}{\partial u}}h_u$, and $\dot A_u=\nabla_{\frac{\partial}{\partial v}}|_{v=u}A_{v,u}$ with respect to an auxiliary linear connection on the bundle $\bigsqcup_uE_u$ and $r_u$ denotes the Heisenberg order of $\dot A_u$ with respect to the filtration on $M_u$.
In particular, $\tilde p_{n,u}$ is locally computable.
More precisely, $\tilde p_{n,u}(x)$ can be computed from $(\dot A_u,\dot h_u,\dot\mu_u)$ at $x$ and the germ of $(M_u,E_u,D_u,h_u,\mu_u)$ at $x$.
Furthermore, $\tilde p_{n,u}=0$ if the homogeneous dimension $n$ is odd.
\end{theorem}

In the remaining part of this section we will give a proof of this theorem.
We begin by addressing the existence of the asymptotic expansion in \eqref{E:asym.all}.

\begin{lemma}\label{L:asymA}
Suppose $B$ is a differential operator, acting on sections of $E_u$, which is of Heisenberg order at most $r_u$ with respect to the filtration on $M_u$, and let $(Bk_{t,u})(x,y)$ denote the (smooth) kernel of $Be^{-t\Delta_u}$.
Then we have an asymptotic expansion, as $t\to0$,
\begin{equation}\label{E:asymAu}
	(Bk_{t,u}\bigr)(x,x)\sim\sum_{j=-r_u}^\infty t^{(j-n)/2\kappa}\,\tilde p^B_{j,u}(x)
\end{equation}
where $\tilde p^B_{j,u}\in\Gamma^\infty(\eend(E_u)\otimes|\Lambda|)$ are locally computable.
More precisely, $\tilde p^B_{j,u}(x)$ can be computed from the coefficients of $B$ at $x$ and the germ of $(M_u,E_u,D_u,h_u,\mu_u)$ at $x$.
Furthermore, $\tilde p_{j,u}^B(x)=0$ for all odd $j$.
\end{lemma}

\begin{proof}
According to \cite[Section~4]{DH20}, the heat operator $\Delta_u+\frac\partial{\partial t}$ admits a parametrix $Q$ in the Heisenberg calculus on $M\times\mathbb R$ whose kernel $k_Q$ satisfies
\[
	k_Q(x,s;y,s-t)=\begin{cases}k_{t,u}(x,y)&\text{for $t>0$, and}\\0&\text{for $t<0$.}\end{cases}
\]
Here $M\times\mathbb R$ is considered as a filtered manifold such that the vector field $\frac\partial{\partial t}$ has Heisenberg order $2\kappa$.
Hence, $BQ$ is a pseudodifferential operator of Heisenberg order $r-2\kappa$ in the Heisenberg calculus on $M\times\mathbb R$.
For its kernel $k_{BQ}$ we find
\[
	k_{BQ}(x,s;y,s-t)=\begin{cases}(Bk_{t,u})(x,y)&\text{for $t>0$, and}\\0&\text{for $t<0$.}\end{cases}
\]
Since $BQ$ is in the Heisenberg calculus, its kernel $k_{BQ}$ admits an asymptotic expansion along the diagonal.
Proceeding exactly as in \cite[Section~4]{DH20}, this leads to the asymptotic expansion in \eqref{E:asymAu}.
\end{proof}

Using parallel transport with respect to the auxiliary connection to identify the bundles $E_u$ with a single vector bundle, we may w.l.o.g.\ assume $E_u=E$ for all $u$.
Then $\dot h_u=h_u^{-1}\frac\partial{\partial u}h_u$ and $\dot A_u=\frac\partial{\partial v}|_{v=u}A_{v,u}$.
Differentiating \eqref{E:ADDA}, we obtain
\begin{equation}\label{E:dADDA}
        \dot D_u=\dot A_uD_u-D_u\dot A_u.
\end{equation}
where $\dot D_u:=\tfrac\partial{\partial u}D_u$.
By the chain rule,
\begin{multline*}
	\tfrac\partial{\partial v}\big|_{v=u}\log\left((\sdet(H(A_{v,u})))^*\|-\|^{\sdet(H^*(E,D_v))}_{h_v,\mu_v}\right)
	\\=\tfrac\partial{\partial v}\big|_{v=u}\log\left((\sdet(H(A_{v,u})))^*\|-\|^{\sdet(H^*(E,D_v))}_{h_u,\mu_u}\right)
	\\+\tfrac\partial{\partial v}\big|_{v=u}\log\|-\|^{\sdet(H^*(E,D_u))}_{h_v,\mu_v}.
\end{multline*}
In view of Theorem~\ref{T:var} we may thus assume w.l.o.g.\ that the Hermitian metrics and the volume densities are independent of $u$, that is, $h_u=h$ and $\mu_u=\mu$.

\begin{lemma}\label{L:vara}
If $\lambda\geq0$ is not in the spectrum of $\Delta_u$ then
\begin{equation}\label{E:vara}
	\tfrac\partial{\partial u}\str\bigl(NQ_{\lambda,u}e^{-t\Delta_u}\bigr)
	=\kappa t\tfrac\partial{\partial t}\str\bigl((\dot A_u+\dot A_u^*)Q_{\lambda,u}e^{-t\Delta_u}\bigr).
\end{equation}
\end{lemma}

\begin{proof}
Defining $\sigma_{u,q}\colon\Gamma^\infty(E^q)\to\Gamma^\infty(E^{q-1})$ by
\begin{equation}\label{E:sigmaia}
        \sigma_{u,q}
        :=(D^*_{u,q-1}D_{u,q-1})^{a_{q-1}-1}D^*_{u,q-1}
        =D^*_{u,q-1}(D_{u,q-1}D^*_{u,q-1})^{a_{q-1}-1},
\end{equation}
we have the graded commutator relations
\begin{equation}\label{E:Dddssa}
        \Delta_u=[D_u,\sigma_u].
\end{equation}
Using Duhamel's formula one shows as in \eqref{E:trQD2} 
\begin{equation*}
	\tfrac\partial{\partial u}\tr\bigl(Q_{\lambda,u,q}e^{-t\Delta_{u,q}}\bigr)
	=-t\tr\bigl(Q_{\lambda,u,q}\dot\Delta_{u,q}e^{-t\Delta_{u,q}}\bigr)
\end{equation*}
where $\dot\Delta_u:=\frac\partial{\partial u}\Delta_u$.
Consequently,
\begin{equation}\label{E:var1a}
	\tfrac\partial{\partial u}\str\bigl(NQ_{\lambda,u}e^{-t\Delta_u}\bigr)
	=-t\str\bigl(NQ_{\lambda,u}\dot\Delta_ue^{-t\Delta_u}\bigr).
\end{equation}

Put $\dot\sigma_u:=\frac\partial{\partial u}\sigma$ and let $k$ denote the operator given by multiplication with $k_q$ on $\Gamma^\infty(E^q)$.
Then we have graded commutator relations $\dot\Delta_u=[\dot D_u,\sigma_u]+[D_u,\dot\sigma_u]$, $[N,D_u]=D_uk$ and $[N,\sigma_u]=-k\sigma_u$, see \eqref{E:Dddssa}, \eqref{E:Niik} and \eqref{E:sigmai}, respectively.
Combining this with $[D_u,e^{-t\Delta_u}]=0=[\sigma_u,e^{-t\Delta_u}]$ and $[D_u,Q_{\lambda,u}]=0=[\sigma_u,Q_{\lambda,u}]$ and the fact that the graded trace vanishes on graded commutators, we obtain
\begin{multline}\label{E:var2a}
	-\str\bigl(NQ_{\lambda,u}\dot\Delta_ue^{-t\Delta_u}\bigr)
	=\str\bigl(kQ_{\lambda,u}(\sigma_u\dot D_u+\dot\sigma_uD_u)e^{-t\Delta_u}\bigr)
\\	=\str\bigl(kQ_{\lambda,u}\tfrac\partial{\partial u}(\sigma_uD_u)e^{-t\Delta_u}\bigr).
\end{multline}
Clearly,
\begin{multline*}
        \tfrac\partial{\partial u}(\sigma_{u,q+1}D_{u,q})
        =\tfrac\partial{\partial u}(D_{u,q}^*D_{u,q})^{a_q}
        \\=\sum_{r=0}^{a_q-1}(D_{u,q}^*D_{u,q})^r\bigl(\dot D_{u,q}^*D_{u,q}+D_{u,q}^*\dot D_{u,q}\bigr)(D_{u,q}^*D_{u,q})^{a_q-r-1}.
\end{multline*}
As $[D_{u,q}^*D_{u,q},e^{-t\Delta_{u,q}}]=0=[D_{u,q}^*D_{u,q},Q_{\lambda,u,q}]$,
this yields
\[
	\tr\bigl(k_qQ_{\lambda,u,q}\tfrac\partial{\partial u}(\sigma_{u,q+1}D_{u,q})e^{-t\Delta_{u,q}}\bigr)
	=\kappa\tr\bigl(Q_{\lambda,u,q}(\dot D_{u,q}^*\sigma_{u,q+1}^*+\sigma_{u,q+1}\dot D_{u,q})e^{-t\Delta_{u,q}}\bigr),
\]
see also \eqref{E:kkiai}.
Consequently,
\begin{equation}\label{E:var3a}
	\str\bigl(kQ_{\lambda,u}\tfrac\partial{\partial u}(\sigma_uD_u)e^{-t\Delta_u}\bigr)
	=\kappa\str\bigl(Q_{\lambda,u}(\dot D_u^*\sigma_u^*+\sigma_u\dot D_u)e^{-t\Delta_u}\bigr).
\end{equation}
Using \eqref{E:dADDA}, $[D_u,e^{-t\Delta_u}]=0=[D_u,Q_{\lambda,u}]$, $[\sigma_u,D_u]=\Delta_u$ and the fact that the graded trace vanishes on graded commutators, we obtain $\str(Q_{\lambda,u}\sigma_u\dot D_ue^{-t\Delta_u})=-\str(\dot A_uQ_{\lambda,u}\Delta_ue^{-t\Delta_u})$.
Similarly, $\str(Q_{\lambda,u}\dot D_u^*\sigma_u^*e^{-t\Delta_u})=-\str(\dot A_u^*Q_{\lambda,u}\Delta_ue^{-t\Delta_u})$ and thus
\begin{equation}\label{E:var4a}
	\str\bigl(Q_{\lambda,u}(\dot D_u^*\sigma_u^*+\sigma_u\dot D_u)e^{-t\Delta_u}\bigr)
	=-\str\bigl((\dot A_u+\dot A_u^*)Q_{\lambda,u}\Delta_ue^{-t\Delta_u}\bigr).
\end{equation}
Clearly,
\begin{equation}\label{E:var5a}
	\tfrac\partial{\partial t}\str\bigl((\dot A_u+\dot A_u^*)Q_{\lambda,u}e^{-t\Delta_u}\bigr)
	=-\str\bigl((\dot A_u+\dot A_u^*)Q_{\lambda,u}\Delta_ue^{-t\Delta_u}\bigr).
\end{equation}
Combining equations \eqref{E:var1a}, \eqref{E:var2a}, \eqref{E:var3a}, \eqref{E:var4a} and \eqref{E:var5a}, we obtain \eqref{E:vara}.
\end{proof}

\begin{lemma}
For $\lambda\geq0$ not in the spectrum of $\Delta_u$ and $\Re(s)>(n+r_u)/2\kappa$ we have
\begin{equation}\label{E:varzetaa}
        \tfrac\partial{\partial u}\tfrac1\kappa\zeta_{\lambda,u}(s)
        =-\frac s{\Gamma(s)}\int_0^\infty t^{s-1}\str\bigl((\dot A_u+\dot A_u^*)Q_{\lambda,u}e^{-t\Delta_u}\bigr)dt.
\end{equation}
Moreover,
\begin{equation}\label{E:dduzetaa}
        \tfrac\partial{\partial u}\tfrac1\kappa\zeta_{\lambda,u}(0)
        =0
\end{equation}
and
\begin{equation}\label{E:varzetapa}
        \tfrac\partial{\partial u}\tfrac1\kappa\zeta'_{\lambda,u}(0)
        =\str\bigl((\dot A_u+\dot A_u^*)P_{u,\lambda}\bigr)-\int_M\str(\tilde p_{n,u})
\end{equation}
where $\tilde p_{n,u}$ is the constant term in the asymptotic expansion \eqref{E:asym.all}.
\end{lemma}

\begin{proof}
Equation~\eqref{E:varzetaa} follows from \eqref{E:zetalambda} and \eqref{E:vara} via partial integration.
From Lemma~\ref{L:asymA} we obtain the asymptotic expansion
\begin{equation}\label{E:asymAA}
	\str\bigl((\dot A_u+\dot A_u^*)e^{-t\Delta_u}\bigr)dt\sim\sum_{j=-r_u}^\infty t^{(j-n)/2\kappa}\int_M\str(\tilde p_{j,u}),
\end{equation}
as $t\to0$.
In particular,
\[
        \frac1{\Gamma(s)}\int_0^\infty t^{s-1}\str\bigl((\dot A_u+\dot A_u^*)Q_{\lambda,u}e^{-t\Delta_u}\bigr)dt
\]
is holomorphic at $s=0$.
Hence, \eqref{E:dduzetaa} follows from \eqref{E:varzetaa}.
Furthermore, 
\begin{multline*}
        \tfrac\partial{\partial u}\tfrac1\kappa\zeta'_{\lambda,u}(0)
        =-\LIM_{t\to0}\str\bigl((\dot A_u+\dot A_u^*)Q_{\lambda,u}e^{-t\Delta_u}\bigr)
        \\=\str\bigl((\dot A_u+\dot A_u^*)P_{\lambda,u}\bigr)-\LIM_{t\to0}\str\bigl((\dot A_u+\dot A_u^*)e^{-t\Delta_u}\bigr).
\end{multline*}
Combining this with \eqref{E:asymAA} we obtain \eqref{E:varzetapa}.
\end{proof}

\begin{lemma}\label{L:varfinnorma}
If $\lambda\geq0$ is not contained in the spectrum of $\Delta_u$, then
\[
	\tfrac\partial{\partial v}\big|_{v=u}\log\left((\sdet(H(A_{v,u})))^*\|-\|^{\sdet(H^*(E,D_v))}_{[0,\lambda],h,\mu}\right)
	=\tfrac12\str\bigl((\dot A_u+\dot A_u^*)P_{\lambda,u}\bigr).
\]
\end{lemma}

\begin{proof}
We consider the linear map
\[
	\alpha_{v,u}\colon\img P_{\lambda,u}\to\img P_{\lambda,v},\qquad\alpha_{v,u}:=P_{\lambda,v}A_{v,u}.
\]
This is a map of finite dimensional complexes inducing \eqref{E:HAu} on cohomology.
Since $\alpha_{u,u}=\id$, the map $\alpha_{v,u}$ is an isomorphism of complexes, for $v$ sufficiently close to $u$.
Hence, for these $v$, we have
\[
	\frac{(\sdet(H(A_{v,u})))^*\|-\|^{\sdet(H^*(E,D_v))}_{[0,\lambda],h,\mu}}{\|-\|^{\sdet(H^*(E,D_u))}_{[0,\lambda],h,\mu}}
	=\sdet^{1/2}(\alpha_{v,u}^*\alpha_{v,u}).
\]
Differentiating and using $P_{\lambda,u}\dot P_{\lambda,u}P_{\lambda,u}=0$,
we obtain the lemma.
\end{proof}

Combining Lemma~\ref{L:varfinnorma} with Equation~\eqref{E:varzetapa}, we obtain Theorem~\ref{T:var.all}, cf.~\eqref{E:defRSmetric}.

\subsection{Elementary properties of the analytic torsion}\label{SS:elementary}

In this section we collect a few properties of the analytic torsion of Rockland complexes generalizing well known facts about the Ray--Singer and the Rumin--Seshadri analytic torsion.

Throughout this section $M$ denotes a closed filtered manifold, $(E,D)$ is a Rockland complex over $M$, $\mu$ is a volume density on $M$ and $h$ is a graded fiberwise Hermitian inner product on $E$.

\begin{proposition}\label{P:shift}
Let $(\tilde E,\tilde D)$ denote the Rockland complex obtained by shifting the grading by one.
More explicitly, $\tilde E^q=E^{q+1}$, $\tilde D_q=D_{q+1}$, and $\tilde h_q=h_{q+1}$.
Then, up to the canonical identification of determinant lines
\begin{equation}\label{E:sdet.shift}
	\sdet(H^*(\tilde E,\tilde D))
	=\sdet(H^*(E,D))^*
\end{equation}
induced by the canonical isomorphisms $H^q(\tilde E,\tilde D)=H^{q+1}(E,D)$, we have
\[
	\|-\|^{\sdet(H^*(\tilde E,\tilde D))}_{\tilde h,\mu}
	=\left(\|-\|^{\sdet(H^*(E,D))}_{h,\mu}\right)^{-1},
\]
where the right hand side denotes the induced norm on the dual (inverse) line.
\end{proposition}

\begin{proof}
Denoting the zeta function associated with the Rockland complex $(\tilde E,\tilde D)$ by $\tilde\zeta_\lambda$, see \eqref{E:zetalambda}, we clearly have $\tilde\zeta_\lambda(s)=-\zeta_\lambda(s)$, for any $\lambda\geq0$.
Moreover, up to the canonical identification in \eqref{E:sdet.shift} we have
\[
	\|-\|^{\sdet(H^*(\tilde E,\tilde D))}_{[0,\lambda],\tilde h,\mu}
	=\left(\|-\|^{\sdet(H^*(E,D))}_{[0,\lambda],h,\mu}\right)^{-1},
\]
see \cite[Remark~1.4(4)]{T01} for the acyclic case.
Combining these facts, the proposition follows at once, cf.~\eqref{E:defRSmetric}.
\end{proof}

For the torsion of a sum of two Rockland complexes we have:

\begin{proposition}\label{P:sumE}
Suppose $(E,D)$ and $(\tilde E,\tilde D)$ are two Rockland complexes over $M$ such that the Heisenberg orders of $D_q$ and $\tilde D_q$ coincide for all $q$.
Moreover, let $h$ and $\tilde h$ be graded, fiber wise Hermitian inner products on $E$ and $\tilde E$, respectively.
Then, up to the canonical identification of determinant lines
\begin{equation}\label{E:sdetoplus}
	\sdet\bigl(H^*(E\oplus\tilde E,D\oplus\tilde D)\bigr)
	=\sdet(H^*(E,D))\otimes\sdet(H^*(\tilde E,\tilde D))
\end{equation}
induced by the canonical isomorphisms $H^q(E\oplus\tilde E,D\oplus\tilde D)=H^q(E,D)\oplus H^q(\tilde E,\tilde D)$, we have
\[
	\|-\|^{\sdet(H^*(E\oplus\tilde E,D\oplus\tilde D))}_{h\oplus\tilde h,\mu}
	=\|-\|^{\sdet(H^*(E,D))}_{h,\mu}\otimes\|-\|^{\sdet(H^*(\tilde E,\tilde D))}_{\tilde h,\mu}.
\]
\end{proposition}

\begin{proof}
Since the Heisenberg orders of $D_q$ and $\tilde D_q$ coincide, $D\oplus\tilde D$ is a Rockland complex and we may use the same numbers $a_q$ and $N_q$ to compute the analytic torsion of all three complexes.
With theses choices, we clearly have
\[
	\Delta^{(E\oplus\tilde E,D\oplus\tilde D)}_{h\oplus\tilde h,\mu}
	=\Delta^{(E,D)}_{h,\mu}\oplus\Delta^{(\tilde E,\tilde D)}_{\tilde h,\mu}.
\]
In particular, 
\begin{equation}\label{E:Qoplus}
	Q^{(E\oplus\tilde E,D\oplus\tilde D)}_{\lambda,h\oplus\tilde h,\mu}
	=Q^{(E,D)}_{\lambda,h,\mu}\oplus Q^{(\tilde E,\tilde D)}_{\lambda,\tilde h,\mu}
\end{equation}
for each $\lambda\geq0$, and we obtain
\begin{equation}\label{E:zetaoplus}
	\zeta^{(E\oplus\tilde E,D\oplus\tilde D)}_{\lambda,h\oplus\tilde h,\mu}(s)
	=\zeta^{(E,D)}_{\lambda,h,\mu}(s)+\zeta^{(\tilde E,\tilde D)}_{\lambda,\tilde h,\mu}(s).
\end{equation}
From \eqref{E:Qoplus} we get a canonical isomorphism of finite dimensional complexes,
\[
	\left(\img\bigl(P^{(E\oplus\tilde E,D\oplus\tilde D)}_{\lambda,h\oplus\tilde h,\mu}\bigr),D\oplus\tilde D\right)
	=\left(\img\bigl(P^{(E,D)}_{\lambda,h,\mu}\bigr),D\right)
	\oplus\left(\img\bigl(P^{(\tilde E,\tilde D)}_{\lambda,\tilde h,\mu}\bigr),\tilde D\right).
\]
It is well known that up to the canonical identification in \eqref{E:sdetoplus} we have
\[
	\|-\|^{\sdet(H^*(E\oplus\tilde E,D\oplus\tilde D))}_{[0,\lambda],h\oplus\tilde h,\mu}
	=\|-\|^{\sdet(H^*(E,D))}_{[0,\lambda],h,\mu}\otimes\|-\|^{\sdet(H^*(\tilde E,\tilde D))}_{[0,\lambda],\tilde h,\mu},
\]
see \cite[Theorem~1.5]{T01} for the acyclic case.
Combining this with \eqref{E:zetaoplus}, the proposition follows at once, cf.~\eqref{E:defRSmetric}.
\end{proof}

Recall that the transposed of a differential operator $A\colon\Gamma^\infty(E)\to\Gamma^\infty(F)$ is the differential operator $A^t\colon\Gamma^\infty(F')\to\Gamma^\infty(E')$ characterized by
$$
(A^t\phi,\psi)_E=(\phi,A\psi)_F,\qquad\phi\in\Gamma^\infty_c(F'),\psi\in\Gamma^\infty(E),
$$
where the parentheses $(-,-)_E$ denote the canonical pairing between sections of $E$ and $E':=E^*\otimes|\Lambda|$.
The following observation will be useful in the discussion of Poincar\'e duality in Section~\ref{SS:PD} below.

\begin{proposition}\label{P:dual}
Let $(E',D^t)$ denote the transposed complex, i.e., $(E')^q=(E^{-q})^*\otimes|\Lambda|$ and $(D^t)_q=(D_{-q-1})^t$.
Equip $(E')^q$ with the induced fiberwise Hermitian inner product, $h'_q=(h_{-q})^{-1}\otimes\mu^{-2}$.
Then, up to the canonical identification 
\begin{equation}\label{E:sdett}
	\sdet(H^*(E',D^t))
	=\sdet(H^*(E,D))^*
\end{equation}
induced by the canonical isomorphism $H^q(E',D^t)=H^{-q}(E,D)^*$, we have
\[
	\|-\|^{\sdet(H^*(E',D^t))}_{h',\mu}
	=\left(\|-\|^{\sdet(H^*(E,D))}_{h,\mu}\right)^{-1},
\]
where the right hand side denotes the induced norm on the dual (inverse) line.
\end{proposition}

\begin{proof}
Since $(D^t)_q=(D_{-q-1})^t$ we have
\[
	(D^t)_q
	=(h_{-q-1}\otimes \mu)D_{-q-1}^*(h_{-q}\otimes \mu)^{-1}.
\]
Dualizing, we obtain
\[
	(D^t)^*_q
	=(h_{-q}\otimes \mu)D_{-q-1}(h_{-q-1}\otimes \mu)^{-1}.
\]
Note here that the vector bundle isomorphism $h_{-q}\otimes\mu\colon\bar E^{-q}\to(E^{-q})^*\otimes|\Lambda|=(E')^q$ intertwines the Hermitian metric $\bar h$ on $\bar E$ with the Hermitian metric $h'$ on $E'$, and thus, with respect to these Hermitian metrics, we have
\[
	(h_{-q}\otimes \mu)^*=(h_{-q}\otimes \mu)^{-1}.
\]
We obtain
\begin{align*}
	(D^t)^*_q(D^t)_q
	&=(h_{-q}\otimes \mu)D_{-q-1}D_{-q-1}^*(h_{-q}\otimes \mu)^{-1},
	\\
	(D^t)_{q-1}(D^t)^*_{q-1}
	&=(h_{-q}\otimes \mu)D_{-q}^*D_{-q}(h_{-q}\otimes \mu)^{-1}.
\end{align*}
Denoting the Heisenberg order of $(D^t)_q$ by $k'_q$, we have 
\begin{equation}\label{E:ktkq}
	k'_q=k_{-q-1}.
\end{equation}
Hence, we may use the numbers $a'_q=a_{-q-1}$ to compute the analytic torsion of $(E',D^t)$, cf.\ \eqref{E:kiai} and Lemma~\ref{L:ai}.
With this choice, we find
\begin{equation}\label{E:Deltat}
	\Delta_{q,h',\mu}^{(E',D^t)}
	=(h_{-q}\otimes \mu)\Delta^{(\bar E,D)}_{-q,\bar h,\mu}(h_{-q}\otimes \mu)^{-1}.
\end{equation}
In particular,
\begin{equation}\label{E:Qt}
	Q^{(E',D^t)}_{q,\lambda,h',\mu}
	=(h_{-q}\otimes \mu)Q_{-q,\lambda,\bar h,\mu}^{(\bar E,D)}(h_{-q}\otimes \mu)^{-1}.
\end{equation}
In view of \eqref{E:ktkq}, the numbers $N'_q:=-N_{-q}$ also satisfy the relation \eqref{E:Niik}, that is, $N'_{q+1}-N'_q=k'_q$.
Clearly,
\begin{equation}\label{E:Nt}
	N'_q
	=-(h_{-q}\otimes \mu)N_{-q}(h_{-q}\otimes \mu)^{-1}.
\end{equation}
Combining \eqref{E:Deltat}, \eqref{E:Qt} and \eqref{E:Nt}, we obtain
\begin{multline}\label{E:zetat}
	\zeta_{\lambda,h',\mu}^{(E',D^t)}(s)
	=\str\left(N'Q^{(E',D^t)}_{\lambda,h',\mu}\left(\Delta^{(E',D^t)}_{h',\mu}\right)^{-s}\right)
	\\=-\str\left(NQ^{(\bar E,D)}_{\lambda,\bar h,\mu}\left(\Delta^{(\bar E,D)}_{\bar h,\mu}\right)^{-s}\right)
	=-\zeta_{\lambda,\bar h,\mu}^{(\bar E,D)}(s).
\end{multline}
Recall here that according to Lemma~\ref{L:Ni}, we may use any sequence of numbers satisfying \eqref{E:Niik} to compute this zeta function.

Clearly, the adjoints of $D$ with respect to $h$ and $\bar h$ coincide.
Hence, $\Delta^{(\bar E,D)}_{\bar h,\mu}=\Delta^{(E,D)}_{h,\mu}$ and $Q^{(\bar E,D)}_{\lambda,\bar h,\mu}=Q^{(E,D)}_{\lambda,h,\mu}$.
Since the spectrum of $\Delta^{(E,D)}_{h,\mu}$ is real, we thus get
\begin{equation}\label{E:zetabar}
	\zeta_{\lambda,\bar h,\mu}^{(\bar E,D)}(s)=\zeta_{\lambda,h,\mu}^{(E,D)}(s).
\end{equation}

Note that from \eqref{E:Deltat} and \eqref{E:Qt} we also obtain 
\[
	\Delta_{q,h',\mu}^{(E',D^t)}
	=\left(\Delta^{(E,D)}_{-q,h,\mu}\right)^t
	\qquad\text{and}\qquad
	Q^{(E',D^t)}_{q,\lambda,h',\mu}
	=\left(Q_{-q,\lambda,h,\mu}^{(E,D)}\right)^t.
\]
Hence, the canonical weakly non-degenerate pairing $\Gamma^\infty((E')^q)\times\Gamma^\infty(E^{-q})\to\mathbb C$ restricts to a (weakly) non-degenerate pairing 
\[
	\img\left(P_{\lambda,q}^{(E',D^t)}\right)\otimes\img\left(P_{\lambda,-q}^{(E,D)}\right)\to\mathbb C.
\]
The latter pairing induces an isomorphism between finite dimensional complexes,
\[
	\left(\img\bigl(P_\lambda^{(E',D^t)}\bigr),D^t\right)
	=\left(\img\bigl(P_\lambda^{(E,D)}\bigr),D\right)^*,
\]
where the right hand side denotes the dual of the complex $\bigl(\img\bigl(P^{(E,D)}_\lambda\bigr),D\bigr)$ with the grading reversed.
Via this isomorphism, the inner product on the left hand side induced by the $L^2$ inner product on $\Gamma(E')$ coincides with the inner product on the right hand side induced by the $L^2$ inner product on $\Gamma(E)$.
Moreover, by the universal coefficient theorem, this identification induces isomorphisms in cohomology, $H^q(E',D^t)=H^{-q}(E,D)^*$, which in turn induce the canonical identification \eqref{E:sdett}.
Hence, via \eqref{E:sdett}, we have
\[
	\|-\|^{\sdet(H^*(E',D^t))}_{[0,\lambda],h',\mu}
	=\left(\|-\|^{\sdet(H^*(E,D))}_{[0,\lambda],h,\mu}\right)^{-1}
\]
see \cite[Theorem~1.9]{T01} for the acyclic case.
Combining this with \eqref{E:zetat} and \eqref{E:zetabar}, the proposition follows at once, see \eqref{E:defRSmetric}.
\end{proof}

With respect to finite coverings we have:

\begin{proposition}\label{P:PF}
Suppose $\pi\colon\tilde M\to M$ is a finite covering of filtered manifolds, let $(\tilde E,\tilde D)$ be a Rockland complex over $\tilde M$, and consider the Rockland complex $(\pi_*\tilde E,\pi_*\tilde D)$ over $M$, where $\pi_*\tilde E$ denotes the vector bundle over $M$ with fibers $(\pi_*\tilde E)_x=\bigoplus_{\tilde x\in\pi^{-1}(x)}\tilde E_{\tilde x}$ and $\pi_*\tilde D$ denotes the induced differential operators.
Moreover, suppose $\tilde h$ is a graded Hermitian inner product on $\tilde E$ and let $(\pi_*\tilde h)_x=\bigoplus_{\tilde x\in\pi^{-1}(x)}h_{\tilde x}$ denote the induced Hermitian inner product on $\pi_*\tilde E$.
Then, up to the canonical identification of determinant lines
\begin{equation*}\label{E:sdetPF}
	\sdet\bigl(H^*(\pi_*\tilde E,\pi_*\tilde D)\bigr)
	=\sdet(H^*(\tilde E,\tilde D))
\end{equation*}
induced by the canonical isomorphisms $H^q(\pi_*\tilde E,\pi_*\tilde D)=H^q(\tilde E,\tilde D)$, we have
\[
	\|-\|^{\sdet(H^*(\pi_*\tilde E,\pi_*\tilde D))}_{\pi_*\tilde h,\mu}
	=\|-\|^{\sdet(H^*(\tilde E,\tilde D))}_{\tilde h,\pi^*\mu}.
\]
\end{proposition}

\begin{proof}
Note that we have a canonical vector bundle map $\tilde E\to\pi_*\tilde E$ over $\pi$.
This vector bundle map induces a canonical isomorphism
\[
	\Gamma^\infty(\pi_*\tilde E)=\Gamma^\infty(\tilde E)
\]
which intertwines $\pi_*\tilde D$ with $\tilde D$ and induces the canonical identification in cohomology.
Moreover, this isomorphism intertwines the $L^2$ inner product induced by $\pi_*\tilde h$ and $\mu$ on $\Gamma^\infty(\pi_*\tilde E)$ with the $L^2$ inner product induced by $\tilde h$ and $\pi^*\mu$ on $\Gamma^\infty(\tilde E)$.
The proposition follows at once.
\end{proof}

\begin{proposition}\label{P:PB}
Suppose $\pi\colon\tilde M\to M$ is a finite covering of filtered manifolds and consider the pulled back Rockland complex $(\pi^*E,\pi^*D)$ over $\tilde M$.
Let $\mathcal V$ denote the flat vector bundle over $M$ with fibers 
\[
	\mathcal V_x=H_0(\pi^{-1}(x);\C)=\C[\pi^{-1}(x)],\qquad x\in M,
\]
and let $h^{\mathcal V}$ denote the canonical parallel fiber wise Hermitian inner product on $\mathcal V$.
Then, up to the canonical identification of determinant lines
\begin{equation*}\label{E:sdetPB}
	\sdet\bigl(H^*(\pi^*E,\pi^*D)\bigr)
	=\sdet\bigl(H^*(E\otimes\mathcal V,D^{\mathcal V})\bigr)
\end{equation*}
induced by the canonical isomorphisms $H^q(\pi^*E,\pi^*D)=H^q(E\otimes\mathcal V,D^{\mathcal V})$, we have
\[
	\|-\|^{\sdet(H^*(\pi^*E,\pi^*D))}_{\pi^*h,\pi^*\mu}
	=\|-\|^{\sdet(H^*(E\otimes\mathcal V,D^{\mathcal V}))}_{h\otimes h^{\mathcal V},\mu},
\]
cf.~Remark~\ref{R:twist}.
\end{proposition}

\begin{proof}
This follows from Proposition~\ref{P:PF}. Indeed, via the canonical isomorphism of vector bundles $\pi_*\pi^*E=E\otimes\mathcal V$ we have $\pi_*\pi^*D=D^{\mathcal V}$ and $\pi_*\pi^*h=h\otimes h^{\mathcal V}$.
\end{proof}

\section{Analytic torsion of Rumin complexes}\label{S:torRumin}

The Rumin complex \cite{R99,R01,R05} associated with a filtered manifold is a complex of higher order differential operators which is conjugate to a subcomplex of the de~Rham complex and computes the cohomology of the underlying manifold.
Rumin has shown that these are Rockland complexes in a graded sense.
Assuming that the cohomology of the osculating algebras is pure, we obtain a Rockland complex in the sense of Section~\ref{S:torRc}.
The aim of this section is to establish basic properties of the analytic torsion associated with these complexes.
In particular, we will address Poincar\'e duality in Theorem~\ref{T:PD}, metric dependence in Theorem~\ref{T:varg}, and the dependence on the filtration in Theorem~\ref{T:def.filt}.

For trivially filtered manifolds the Rumin complex coincides with the de~Rham complex and gives rise to the classical Ray--Singer torsion \cite{RS71,BZ92}.
For contact manifolds the Rumin complex \cite{R90,R94,R00} has been used by Rumin and Seshadri \cite{RS12} to define an analytic torsion.
Although these two classical torsions appear as special cases of our general construction, we will say noting new about them.
Cartan's $(2,3,5)$ geometries \cite{C10} constitute another class of $5$-dimensional filtered manifolds whose osculating algebras have pure cohomology.
Their analytic torsion will be discussed in Section~\ref{S:five} below.

Our assumption that the osculating algebras have pure cohomology appears to be very restrictive.
The three cases mentioned previously are the only types of filtered manifolds we know of which have this property.
In Section~\ref{SS:pure} we present a necessary condition which restricts the dimensions of such Lie algebras tremendously.

\subsection{Rumin complexes}

We continue to consider a closed filtered manifold $M$.
Recall that $\mathfrak tM$ denotes the bundle of osculating algebras and let $\partial_q\colon\Lambda^q\mathfrak t^*M\to\Lambda^{q+1}\mathfrak t^*M$ denote the fiberwise Chevalley--Eilenberg codifferential.
We assume that the dimension of the Lie algebra cohomology $\mathcal H^q(\mathfrak t_xM)=\ker\partial_{q,x}/\img\partial_{q,x}$ is locally constant in $x\in M$, cf.~\cite[Definition~2.4]{R01}.
Hence, $\mathcal H^q(\mathfrak tM)$ is a smooth vector bundle over $M$, for each $q$.
The filtration on $TM$ induces a filtration on $\Lambda^*T^*M$.
Since the filtration on $M$ is compatible with Lie brackets, the de~Rham differential $d$ on $\Omega^*(M)$ is filtration preserving and induces $\gr(d)=\partial$ on the associated graded, $\gr(\Lambda^*T^*M)=\Lambda^*\mathfrak t^*M$.

Let $\tilde g$ be a graded fiberwise Euclidean inner product on $\mathfrak tM=\bigoplus_p\mathfrak t^pM$.
We will denote the induced fiberwise Euclidean inner product on $\Lambda^q\mathfrak t^*M$ by $\tilde g^{-1}$.
Let $\partial^*_q\colon\Lambda^{q+1}\mathfrak t^*M\to\Lambda^q\mathfrak t^*M$ denote the corresponding fiberwise adjoint of $\partial_q$.
Fiberwise finite dimensional Hodge theory provides an orthogonal decomposition of vector bundles
\begin{equation}\label{E:LtMdeco}
	\Lambda^q\mathfrak t^*M=\img\partial_{q-1}\oplus\mathcal H^q(\mathfrak tM)\oplus\img\partial^*_q,
\end{equation}
where $\mathcal H^q(\mathfrak tM)=\ker\partial_q/\img\partial_{q-1}=\ker\partial_q\cap\ker\partial^*_{q-1}=\ker\partial_{q-1}^*/\img\partial_q^*$.

We fix a splitting of the filtration $S\colon\mathfrak tM\to TM$, i.e., a filtration preserving vector bundle isomorphism inducing the identity on the associated graded.
Then $(S^t)^{-1}\colon\mathfrak t^*M\to T^*M$ is a splitting for the dual filtration which will be denoted by $S$ too.
We extend it to a splitting $S\colon\Lambda^*\mathfrak t^*M\to\Lambda^*T^*M$ characterized by
\begin{equation}\label{E:Swedge}
	S(\alpha\wedge\beta)=S\alpha\wedge S\beta
\end{equation}
for $\alpha,\beta\in\Lambda^*\mathfrak t^*_xM$.
Hence $\delta_q\colon\Lambda^qT^*M\to\Lambda^{q-1}T^*M$, 
\begin{equation}\label{E:delta}
	\delta_q:=S\circ\partial_{q-1}^*\circ S^{-1}
\end{equation}
is a filtration preserving vector bundle homomorphism inducing $\gr(\delta)=\partial^*$ on the associated graded.
Actually, $\delta$ is a Kostant type codifferential \cite[Definition~4.8 and Remark~4.14]{DH17} for the de~Rham complex $\Omega^*(M)$.

Kostant's box operator $\Box=\delta d+d\delta$ is filtration preserving and induces $\gr(\Box)=\tilde\Box:=\partial^*\partial+\partial\partial^*$ on the associated graded.
If $0\neq z\in\mathbb C$ is sufficiently close to zero, then the vector bundle map $z-\tilde\Box$ is invertible.
This implies that $z-\Box$ is invertible too and its inverse is again a filtration preserving differential operator.
Indeed, the inverse can be expressed using a finite geometric series, cf.~\cite[Lemma~1]{R99}, \cite[Lemma~2.5]{R01}, \cite{CSS01}, \cite[Theorem 5.2]{CD01}, or \cite[Lemma~4.3]{DH17}.
Hence, for sufficiently small $\varepsilon>0$, 
\begin{equation}\label{E:Pi}
	\Pi:=\frac1{2\pi\mathbf i}\oint_{|z|=\varepsilon}(z-\Box)^{-1}dz
\end{equation}
is a filtration preserving differential projector, $\Pi^2=\Pi$, inducing $\gr(\Pi)=\tilde\Pi$ on the associated graded, where $\tilde\Pi$ denotes the orthogonal projection onto the subbundle $\ker\tilde\Box=\ker\partial^*\cap\ker\partial$ of $\Lambda^*\mathfrak t^*M$.
Since $\Box$ commutes with $d$ and $\delta$, the same is true for $\Pi$, that is, $d\Pi=\Pi d$, $\delta\Pi=\Pi\delta$.
The operator $\Pi$ coincides with Rumin's projector $\Pi_E$ in \cite[Theorem~1]{R99} and \cite[Theorem~2.6]{R01}.
For more details on the construction presented here we refer to \cite[Lemma~4.4(b)]{DH17}.

The differential operator $\mathbf L\colon\Gamma^\infty(\Lambda^*\mathfrak t^*M)\to\Omega^*(M)$,
\[
	\mathbf L=\Pi S\tilde\Pi+(1-\Pi)S(1-\tilde\Pi)
\]
is filtration preserving and induces the identity on the associated graded, $\gr(\mathbf L)=\id$.
Hence, $\mathbf L$ is invertible and its inverse is a filtration preserving differential operator.
Since $\Pi\mathbf L=\mathbf L\tilde\Pi$ and $\Pi d=d\Pi$, the differential operator $\mathbf L^{-1}d\mathbf L$ commutes with $\tilde\Pi$.
Hence, $\mathbf L^{-1}d\mathbf L$ decouples into a sum of two differential operators,
\begin{equation}\label{E:dDB}
	\mathbf L^{-1}d\mathbf L=\mathbf D\oplus\mathbf B
\end{equation}
with respect to the decomposition $\Lambda^*\mathfrak t^*M=\img(\tilde\Pi)\oplus\ker(\tilde\Pi)$, where 
\begin{align*}
	\mathbf D\colon\Gamma^\infty(\img\tilde\Pi)&\to\Gamma^\infty(\img\tilde\Pi),&\mathbf D&:=\mathbf L^{-1}d\mathbf L|_{\Gamma^\infty(\img\tilde\Pi)}\\
	\mathbf B\colon\Gamma^\infty(\ker\tilde\Pi)&\to\Gamma^\infty(\ker\tilde\Pi),&\mathbf B&:=\mathbf L^{-1}d\mathbf L|_{\Gamma^\infty(\ker\tilde\Pi)}.
\end{align*}
From $d^2=0$ we obtain $\mathbf D^2=0$ and $\mathbf B^2=0$.
The complex $\mathbf D$ is Rumin's complex denoted $(E_0,d_c)$ in \cite[Theorem~1]{R99} and \cite[Theorem~2.6]{R01}.
The complex $\mathbf B$ is acyclic and the restriction of $\mathbf L$ induces an isomorphism between the cohomology of $\mathbf D$ and the de~Rham cohomology \cite{R99,R01}.
Indeed, $\partial^*\mathbf B+\mathbf B\partial^*$ is invertible for it is filtration preserving and the induced homomorphism on the associated graded, $\gr(\partial^*\mathbf B+\mathbf B\partial^*)=\partial^*\partial+\partial\partial^*$, is invertible on $\ker\tilde\Pi$.
One can show that the complex $\mathbf B$ is actually conjugate, via an invertible filtration preserving differential operator, to the acyclic tensorial complex $\partial|_{\Gamma^\infty(\ker\tilde\Pi)}$, see \cite[Proposition~4.5(c)]{DH17}.

We use the decomposition in \eqref{E:LtMdeco} to identify $\mathcal H(\mathfrak tM)=\ker\partial/\img\partial$ with $\img\tilde\Pi=\ker\partial^*\cap\ker\partial$.
More precisely, we let $\iota\colon\mathcal H(\mathfrak tM)\to\img\tilde\Pi\subseteq\ker\partial$ denote the corresponding isomorphism of vector bundles, splitting the canonical projection $\ker\partial\to\ker\partial/\img\partial=\mathcal H(\mathfrak tM)$.
Putting $D:=\iota^{-1}\mathbf D\iota$, we obtain a complex of differential operators we will refer to as Rumin complex:
\begin{equation}\label{E:Rumin}
	\cdots\to\Gamma^\infty(\mathcal H^{q-1}(\mathfrak tM))\xrightarrow{D_{q-1}}\Gamma^\infty(\mathcal H^q(\mathfrak tM))
\xrightarrow{D_q}\Gamma^\infty(\mathcal H^{q+1}(\mathfrak tM))\to\cdots
\end{equation}
Moreover, $L\colon\Gamma^\infty\bigl(\mathcal H(\mathfrak tM)\bigr)\to\Omega(M)$, $L:=\mathbf L\iota$ is a chain map, 
\begin{equation}\label{E:dLLD}
	dL=LD
\end{equation} 
inducing an isomorphism in cohomology.
The inverse is induced by the differential operator $\iota^{-1}\tilde\Pi\mathbf L^{-1}\colon\Omega(M)\to\Gamma^\infty\bigl(\mathcal H(\mathfrak tM)\bigr)$ which satisfies $(\iota^{-1}\tilde\Pi\mathbf L^{-1})L=\id$, $L(\iota^{-1}\tilde\Pi\mathbf L^{-1})=\Pi$, and
\begin{equation}\label{E:itPL}
	D(\iota^{-1}\tilde\Pi\mathbf L^{-1})=(\iota^{-1}\tilde\Pi\mathbf L^{-1})d.
\end{equation}
The latter three equations follow from $L=\mathbf L\iota$, $\mathbf L\tilde\Pi=\Pi\mathbf L$, \eqref{E:dDB}, and $D=\iota^{-1}\mathbf D\iota$.
The construction of the Rumin complex presented here is motivated by the construction of natural differential operators \cite{CSS01,CS12,CD01,DH17} in parabolic geometry.

\begin{remark}\label{R:D.pdL}
The operator $L$ can be characterized as the unique differential operator $L\colon\Gamma^\infty(\mathcal H(\mathfrak tM))\to\Omega(M)$ for which 
\begin{equation}\label{E:Ldef}
	\delta L=0,\qquad\delta dL=0\qquad\text{and}\qquad\pi L=\id,
\end{equation}
where $\pi\colon\ker\delta\to\mathcal H(\mathfrak tM)$ denotes the composition of $S^{-1}\colon\ker\delta\to\ker\partial^*$ with the projection $\ker\partial^*\to\ker\partial^*/\img\partial^*=\ker\partial^*\cap\ker\partial=\mathcal H(\mathfrak tM)$.
This follows readily from 
\[
	\Gamma^\infty(\ker\delta)=\img(\Pi)\oplus\Gamma^\infty(\img\delta)
\]
and
\[
	\img(\Pi)=\ker(\Box)=\ker(\delta)\cap\ker(\delta d)
\]
see \cite{R99,R01} or \cite[Lemma~4.7]{DH17}.
Moreover, we have $D=\pi dL$ and, cf.~\eqref{E:itPL},
\begin{equation}\label{E:LPiS}
	\iota^{-1}\tilde\Pi\mathbf L^{-1}=\pi\Pi=\iota^{-1}\tilde\Pi S^{-1}\Pi.
\end{equation}
\end{remark}

Rumin has shown that the sequence \eqref{E:Rumin} becomes exact in every non-trivial unitary representation of the osculating group, see \cite[Theorem~3]{R99} or \cite[Theorem~5.2]{R01}.
Hence, the operators in \eqref{E:Rumin} form a \emph{graded} Rockland complex which has graded Heisenberg order zero, see also \cite[Corollary~4.20(b)]{DH17}.
In the flat case, where $M$ is locally diffeomorphic to a graded nilpotent Lie group, Rumin \cite{R99,R01} proved that this complex is C-C elliptic.

To obtain a Rockland complex, we assume from now on that the osculating algebras have pure cohomology \cite{R99,R01}, that is, $\mathcal H^q(\mathfrak t_xM)=\bigoplus_p\mathcal H^q(\mathfrak t_xM)_p$ is concentrated in a single degree for each $q$.
In other words, we assume that the grading automorphism $\phi_t$ given by multiplication with $t^p$ on $\mathfrak t^pM$, $t>0$, acts as a scalar in each cohomology group.
Hence, there exist numbers $p_q$ such that  
\begin{equation}\label{E:deltaHq}
	\phi_t=t^{p_q}\qquad\text{on}\qquad\mathcal H^q(\mathfrak tM),\qquad t>0.
\end{equation}
Denoting the Heisenberg order of $D_q$ by $k_q$, we then have 
\begin{equation}\label{E:kqpq}
	k_q=p_{q+1}-p_q.
\end{equation}

\begin{lemma}\label{L:Sindep}
If the osculating algebras have pure cohomology, then the Rumin complex \eqref{E:Rumin} does not depend on the choice of a graded fiberwise Euclidean inner product $\tilde g$ on $\mathfrak tM$ or the choice of a splitting $S\colon\mathfrak tM\to TM$.
Moreover, the isomorphism in cohomology induced by $L$ is independent of these choices too.
\end{lemma}

\begin{proof}
Suppose the Euclidean inner product $\tilde g_u$ and the splitting $S_u$ depend on a (discrete) parameter $u$.
Let $\tilde\Pi_u$, $\mathbf L_u$, $L_u$, $\iota_u$, and $D_u$ denote the associated operators.
Combining \eqref{E:dLLD} and \eqref{E:itPL}, we obtain
\[
	\bigl(\iota_v^{-1}\tilde\Pi_v\mathbf L_v^{-1}L_u\bigr)D_u=D_v\bigl(\iota_v^{-1}\tilde\Pi_v\mathbf L_v^{-1}L_u\bigr).
\]
It remains to show that $\iota_v^{-1}\tilde\Pi_v\mathbf L_v^{-1}L_u=\id$, for any two parameters $u$ and $v$.
Note that each factor is filtration preserving and the composition induces 
\[
	\gr\bigl(\iota_v^{-1}\tilde\Pi_v\mathbf L_v^{-1}L_u\bigr)=\iota_v^{-1}\tilde\Pi_v\iota_u=\id
\]
on the associated graded.
Indeed, $\gr(\mathbf L)=\id$, $\iota$ takes values in $\ker\partial$, and $\iota^{-1}\tilde\Pi|_{\ker\partial}$ coincides with the canonical projection $\ker\partial\to\ker\partial/\img\partial=\mathcal H(\mathfrak tM)$.
Since the filtration on $\mathcal H^q(\mathfrak tM)$ was assumed to be trivial, this implies $\iota_v^{-1}\tilde\Pi_v\mathbf L_v^{-1}L_u=\id$ on $\Gamma^\infty(\mathcal H^q(\mathfrak tM))$, for each $q$.
\end{proof}

\subsection{Analytic torsion}

We continue to consider a closed filtered manifold $M$ whose osculating algebras have pure cohomology of locally constant dimension.
Twisting the Rumin complex in \eqref{E:Rumin} with a flat vector bundle $F$, we obtain a Rockland complex
\begin{equation}\label{E:RuminF}
\cdots\to\Gamma^\infty(\mathcal H^{q-1}(\mathfrak tM)\otimes F)\xrightarrow{D_{q-1}^F}\Gamma^\infty(\mathcal H^q(\mathfrak tM)\otimes F)
\xrightarrow{D_q^F}\Gamma^\infty(\mathcal H^{q+1}(\mathfrak tM)\otimes F)\to\cdots
\end{equation}
computing the cohomology of $M$ with coefficients in the flat bundle $F$.
More explicitly, the twisted operator $L^F\colon\Gamma^\infty(\mathcal H(\mathfrak tM)\otimes F)\to\Omega(M;F)$ provides a chain map,
\begin{equation}\label{E:Lchain}
	L^F\circ D^F=d^F\circ L^F,
\end{equation}
which induces an isomorphism on cohomology,
\begin{equation}\label{E:HL}
	H^q\bigl(\mathcal H(\mathfrak tM)\otimes F,D^F\bigr)=H^q(M;F).
\end{equation}

Let $\tilde g$ be a graded fiberwise Euclidean inner product on $\mathfrak tM$ and let $\tilde g^{-1}$ denote the induced fiberwise Euclidean inner product on $\Lambda^q\mathfrak t^*M$.
Via the orthogonal decomposition in \eqref{E:LtMdeco} we obtain a fiberwise Euclidean inner product on $\mathcal H^q(\mathfrak tM)$ that will be denoted by $\tilde g^{-1}$ too.
If $h$ is a fiberwise Hermitian inner product on $F$, we obtain a fiberwise Hermitian inner product on $\mathcal H^q(\mathfrak tM)\otimes F$ we will denote by $\tilde g^{-1}\otimes h$.
Using the volume density $\mu_{\tilde g}$ on $M$ induced from the Euclidean inner product $\tilde g^{-1}$ via the canonical identification $\Lambda^m\mathfrak t^*M=\Lambda^mT^*M$ where $m=\dim M$, we obtain formal adjoints $(D^F)^*$ and Rumin--Seshadri operators, see \eqref{E:Deltai}, which will be denoted by 
\begin{equation}\label{E:RS}
	\Delta_{\tilde g,h}^F\colon\Gamma^\infty\bigl(\mathcal H(\mathfrak tM)\otimes F\bigr)\to\Gamma^\infty\bigl(\mathcal H(\mathfrak tM)\otimes F\bigr).
\end{equation}
These are Rockland operators of Heisenberg order $2\kappa$, see \eqref{E:kkiai}, which also depend on the choice of numbers $a_q$ as in \eqref{E:kiai}, but they are independent of the splitting $S$ according to Lemma~\ref{L:Sindep}.

Using the identification in \eqref{E:HL}, we obtain an analytic torsion, i.e., a norm 
\begin{equation}\label{E:deftorM}
	\|-\|^{\sdet(H^*(M;F))}_{\mathcal F,\tilde g,h}:=\|-\|^{\sdet(H^*(\mathcal H(\mathfrak tM)\otimes F,D^F))}_{\tilde g^{-1}\otimes h,\mu_{\tilde g}}
\end{equation}
on the graded determinant line,
\[
	\sdet(H^*(M;F)):=\bigotimes_q\bigl(\det H^q(M;F)\bigr)^{(-1)^q},
\]
cf.~\eqref{E:defRSmetric}.
The subscript $\mathcal F$ indicates the dependence on the filtration on $M$.
We will refer to this as the analytic torsion of the filtered manifold $M$ with coefficients in the flat bundle $F$.
It is defined for closed filtered manifolds whose osculating algebras have pure cohomology of locally constant dimension.

The analytic torsion defined above is a common generalization of the Ray--Singer torsion \cite{RS71,BZ92} and the Rumin--Seshadri \cite{RS12} analytic torsion.
In Section~\ref{S:five} we will discuss another instance of this torsion associated with a certain five dimensional geometry.
We are not aware of further filtered manifolds whose osculating algebras have pure cohomology, cf.~Section~\ref{SS:pure} below.

Let us spell out the following immediate consequence of Proposition~\ref{P:sumE}:

\begin{proposition}\label{P:sumF}
Suppose $F_1$ and $F_2$ are two flat vector bundles with Hermitian inner products $h_1$ and $h_2$, respectively.
Then, up to the canonical identification of determinant lines
\[
	\sdet\bigl(H^*(M;F_1\oplus F_2)\bigr)
        =\sdet(H^*(M;F_1))\otimes\sdet(H^*(M;F_2))
\]
induced by the canonical isomorphisms $H^q(M;F_1\oplus F_2)=H^q(M;F_1)\oplus H^q(M;F_2)$, we have
\[
	\|-\|^{\sdet(H^*(M;F_1\oplus F_2))}_{\mathcal F,\tilde g,h_1\oplus h_2}
	=\|-\|^{\sdet(H^*(M;F_1))}_{\mathcal F,\tilde g,h_1}\otimes\|-\|^{\sdet(H^*(M;F_2))}_{\mathcal F,\tilde g,h_2}.
\]
\end{proposition}

With respect to finite coverings we have

\begin{proposition}\label{P:PFF}
Suppose $\pi\colon\tilde M\to M$ is a finite covering of filtered manifolds, let $\tilde F$ be a flat vector bundle over $\tilde M$, and let $\pi_*\tilde F$ denote the flat vector bundle over $M$ with fibers $(\pi_*\tilde F)_x=\bigoplus_{\tilde x\in\pi^{-1}(x)}\tilde F_{\tilde x}$.
Moreover, let $\tilde h$ be a Hermitian inner product on $\tilde F$ and let $(\pi_*\tilde h)_x=\bigoplus_{\tilde x\in\pi^{-1}(x)}\tilde h_{\tilde x}$ denote the induced Hermitian inner product on $\pi_*\tilde F$.
Then, up to the canonical identification of determinant lines
\[
        \sdet\bigl(H^*(M;\pi_*\tilde F)\bigr)
	=\sdet\bigl(H^*(\tilde M;\tilde F)\bigr)
\]
induced by the canonical isomorphisms $H^q(M;\pi_*\tilde F)=H^q(\tilde M;\tilde F)$, we have
\[
	\|-\|^{\sdet(H^*(M;\pi_*\tilde F))}_{\mathcal F,\tilde g,\pi_*\tilde h}
	=\|-\|^{\sdet(H^*(\tilde M;\tilde F))}_{\pi^*\mathcal F,\pi^*\tilde g,\tilde h}.
\]
\end{proposition}

\begin{proof}
This follows from Proposition~\ref{P:PF}.
Indeed, since the covering map is a local diffeomorphism of filtered manifolds, we have canonical isomorphisms of vector bundles $\pi_*(\mathcal H^q(\mathfrak t\tilde M)\otimes\tilde F)=\mathcal H^q(\mathfrak tM)\otimes\pi_*\tilde F$ intertwining $\pi_*D_q^{\tilde F}$ with $D^{\pi_*\tilde F}_q$, cf.~\eqref{E:RuminF}, and intertwining the Hermitian inner product $\pi_*(\pi^*\tilde g^{-1}\otimes\tilde h)$ with the Hermitian inner product $\tilde g^{-1}\otimes\pi_*\tilde h$.
Moreover, $\pi^*\mu_{\tilde g}=\mu_{\pi^*\tilde g}$.
\end{proof}

\begin{proposition}\label{P:PBF}
Suppose $\pi\colon\tilde M\to M$ is a finite covering of filtered manifolds, let $\mathcal V$ denote the flat vector bundle over $M$ with fibers $\mathcal V_x=H_0(\pi^{-1}(x);\C)=\C[\pi^{-1}(x)]$, $x\in M$, and let $h^{\mathcal V}$ denote the canonical parallel Hermitian inner product on $\mathcal V$.
Then, up to the canonical identification of determinant lines
\[
	\sdet\bigl(H^*(\tilde M;\pi^*F)\bigr)
        =\sdet\bigl(H^*(M;F\otimes\mathcal V)\bigr)
\]
induced by the canonical isomorphisms $H^q(\tilde M;\pi^* F)=H^q(M;F\otimes\mathcal V)$, we have
\[
	\|-\|^{\sdet(H^*(\tilde M;\pi^*F))}_{\pi^*\mathcal F,\pi^*\tilde g,\pi^*h}
	=\|-\|^{\sdet(H^*(M;F\otimes\mathcal V))}_{\mathcal F,\tilde g,h\otimes h^{\mathcal V}}.
\]
\end{proposition}

\begin{proof}
This follows from Proposition~\ref{P:PFF}.
Indeed, via the canonical isomorphism of flat vector bundles $\pi_*\pi^*F=F\otimes\mathcal V$ we have $\pi_*\pi^*h=h\otimes h^{\mathcal V}$.
\end{proof}

\subsection{Poincar\'e duality}\label{SS:PD}

Rumin observed that his complex is Hodge $\star$ self-dual, see \cite[Section~2]{R99} or \cite[Proposition~2.8]{R01}.
In this section we recall this duality and discuss implications for the analytic torsion.

Let us denote the dimension of $M$ by 
\[
	m:=\dim M.
\]
Recall that the wedge product provides a fiberwise non-degenerate bilinear pairing
\begin{equation}\label{E:wedgeL}
	\Lambda^q\mathfrak t^*M\otimes\Lambda^{m-q}\mathfrak t^*M\xrightarrow{\,\,\wedge\,\,\,}\Lambda^m\mathfrak t^*M.
\end{equation}
The corresponding vector bundle isomorphism
\begin{equation}\label{E:wedgeq}
	\wedge_q\colon\Lambda^q\mathfrak t^*M\to\left(\Lambda^{m-q}\mathfrak t^*M\right)^*\otimes\Lambda^m\mathfrak t^*M
\end{equation}
is an isometry with respect to the Euclidean metric $\tilde g^{-1}$ on the left hand side and the induced Euclidean metric $\tilde g\otimes\tilde g^{-1}$ on the right hand side.
Hence,
\begin{equation}\label{E:wedgeiso}
	\wedge_q^*=\wedge_q^{-1}.
\end{equation}

Since the Chevalley--Eilenberg codifferential $\partial\colon\Lambda^*\mathfrak t_x^*M\to\Lambda^{*+1}\mathfrak t_x^*M$ is a graded derivation, the wedge product induces a fiberwise bilinear pairing
\begin{equation}\label{E:wedgeH}
	\mathcal H^q(\mathfrak tM)\otimes\mathcal H^{m-q}(\mathfrak tM)
	\xrightarrow{\,\,\wedge\,\,\,}\mathcal H^m(\mathfrak tM),
\end{equation}
where we regard $\mathcal H(\mathfrak tM)=\ker\partial/\img\partial$.
By Poincar\'e duality for the nilpotent Lie algebra $\mathfrak t_xM$, this pairing is fiberwise non-degenerate.
In particular, the codifferential in top degree, $\partial\colon\Lambda^{m-1}\mathfrak t^*M\to\Lambda^m\mathfrak t^*M$ vanishes.
Hence, we have a canonical identification
\begin{equation}\label{E:canHtM}
	\mathcal H^m(\mathfrak tM)=\Lambda^m\mathfrak t^*M
\end{equation}
and
\begin{equation}\label{E:partialwedge}
	\partial\alpha\wedge\beta+(-1)^q\alpha\wedge\partial\beta=0,
\end{equation}
for $\alpha\in\Lambda^q\mathfrak t_x^*M$ and $\beta\in\Lambda^{m-q-1}\mathfrak t_x^*M$.
Using the vector bundle isomorphism in \eqref{E:wedgeq}, this equation may be written in the form $\wedge_{q+1}\circ\partial_q=-(-1)^q(\partial^t_{m-q-1}\otimes\id)\circ\wedge_q$.
Dualizing, we obtain $\partial_q^*\circ\wedge_{q+1}^*=-(-1)^q\wedge_q^*\circ((\partial^t_{m-q-1})^*\otimes\id)$.
Using $(\partial^t)^*=(\partial^*)^t$ and \eqref{E:wedgeiso}, this yields $\wedge_q\circ\partial_q^*=-(-1)^q\circ((\partial^*_{m-q-1})^t\otimes\id)\circ\wedge_{q+1}$.
In other words, we also have
\begin{equation}\label{E:PDpartial*}
	\partial^*\alpha\wedge\beta+(-1)^q\alpha\wedge\partial^*\beta=0
\end{equation}
for $\alpha\in\Lambda^{q+1}\mathfrak t_x^*M$ and $\beta\in\Lambda^{m-q}\mathfrak t_x^*M$.

Combining \eqref{E:PDpartial*} with \eqref{E:delta} and \eqref{E:Swedge}, we obtain
\begin{equation}\label{E:PDdelta}
	\delta\phi\wedge\psi+(-1)^q\phi\wedge\delta\psi=0
\end{equation}
for $\phi\in\Lambda^{q+1}T_x^*M$ and $\psi\in\Lambda^{m-q}T^*_xM$.
Moreover,
\begin{equation}\label{E:PDpiwedge}
	\pi(\phi\wedge\psi)=\pi\phi\wedge\pi\psi
\end{equation}
for $\phi\in\ker\delta_{q,x}$ and $\psi\in\ker\delta_{m-q,x}$.
Here $\pi\colon\ker\delta\to\mathcal H(\mathfrak tM)$ is the vector bundle map from Remark~\ref{R:D.pdL}.
Indeed, \eqref{E:PDpiwedge} is obvious for $\phi\in S_q(\ker\partial^*_{q-1,x}\cap\ker\partial_{q,x})$ and $\psi\in S_{m-q}(\ker\partial^*_{m-q-1,x}\cap\ker\partial_{m-q,x})$, see \eqref{E:Swedge} and \eqref{E:delta}.
The general case follows from \eqref{E:PDdelta} since $\pi$ induces an isomorphism $\ker\delta/\img\delta\to\ker\partial^*/\img\partial^*=\ker\partial^*\cap\ker\partial=\mathcal H(\mathfrak tM)$.
Note that in top degree, $\pi$ is an isomorphism of line bundles that coincides with the canonical identification in \eqref{E:canHtM}.
Hence, combining \eqref{E:PDpiwedge} with the relations $\delta L=0$ and $\pi L=\id$ from \eqref{E:Ldef}, we obtain $\pi(L\alpha\wedge L\beta)=\alpha\wedge\beta$ and thus
\begin{equation}\label{E:PDL}
	L\alpha\wedge L\beta=L(\alpha\wedge\beta)
\end{equation}
for all $\alpha\in\Gamma^\infty(\mathcal H^q(\mathfrak tM))$ and $\beta\in\Gamma^\infty(\mathcal H^{m-q}(\mathfrak tM))$.
Note that in top degree, $L$ is algebraic, inducing the canonical identification in \eqref{E:canHtM} inverse to $\pi$.

For every flat vector bundle $F$, the wedge product in \eqref{E:wedgeH} induces a fiberwise non-degenerate bilinear pairing
\begin{equation}\label{E:wedgeHF}
	\bigl(\mathcal H^q(\mathfrak tM)\otimes F^*\otimes\mathcal O\bigr)\otimes\bigl(\mathcal H^{m-q}(\mathfrak tM)\otimes F\bigr)\xrightarrow{\,\,\wedge\,\,\,}|\Lambda|.
\end{equation}
Here $\mathcal O$ denotes the orientation bundle of $M$ and we are using the canonical identification of line bundles
\begin{equation}\label{E:Lmcan}
	\mathcal H^m(\mathfrak tM)\otimes\mathcal O
	=\Lambda^m\mathfrak t^*M\otimes\mathcal O
	=\Lambda^mT^*M\otimes\mathcal O
	=|\Lambda|.
\end{equation}
From \eqref{E:PDL} we obtain
\begin{equation}\label{E:PDLF}
	L^{F^*\otimes\mathcal O}\alpha\wedge L^F\beta=\alpha\wedge\beta
\end{equation}
for all $\alpha\in\Gamma^\infty\bigl(\mathcal H^q(\mathfrak tM)\otimes F^*\otimes\mathcal O\bigr)$ and $\beta\in\Gamma^\infty\bigl(\mathcal H^{m-q}(\mathfrak tM)\otimes F\bigr)$.

Variants of the following statement can be found in \cite[Section~2]{R99}, \cite[Proposition~2.8]{R01}, and \cite[Section~7]{CD01}.

\begin{lemma}\label{L:D*}
We have
\[
	D^{F^*\otimes\mathcal O}\alpha\wedge\beta+(-1)^q\alpha\wedge D^F\beta
	=d\bigl(L^{F^*\otimes\mathcal O}\alpha\wedge L^F\beta\bigr)
\]
for $\alpha\in\Gamma^\infty\bigl(\mathcal H^q(\mathfrak tM)\otimes F^*\otimes\mathcal O\bigr)$ and $\beta\in\Gamma^\infty\bigl(\mathcal H^{m-q-1}(\mathfrak tM)\otimes F\bigr)$.
\end{lemma}

\begin{proof}
By the Leibniz rule for the de~Rham differential, \eqref{E:Lchain} and \eqref{E:PDLF} we find:
\begin{align*}
	d\bigl(L^{F^*\otimes\mathcal O}\alpha\wedge L^{F}\beta\bigr)
	&=d^{F^*\otimes\mathcal O}L^{F^*\otimes\mathcal O}\alpha\wedge L^F\beta+(-1)^qL^{F^*\otimes\mathcal O}\alpha\wedge d^FL^F\beta
	\\&=L^{F^*\otimes\mathcal O}D^{F^*\otimes\mathcal O}\alpha\wedge L^F\beta+(-1)^qL^{F^*\otimes\mathcal O}\alpha\wedge L^FD^F\beta
	\\&=D^{F^*\otimes\mathcal O}\alpha\wedge\beta+(-1)^q\alpha\wedge D^F\beta
	\qedhere
\end{align*}
\end{proof}

By Lemma~\ref{L:D*} and Stokes' theorem, the pairing in \eqref{E:wedgeHF} induces a pairing
\begin{equation}\label{E:wedgeHH}
	H^q\bigl(\mathcal H(\mathfrak tM)\otimes F^*\otimes\mathcal O,D^{F^*\otimes\mathcal O}\bigr)\otimes H^{m-q}\bigl(\mathcal H(\mathfrak tM)\otimes F,D^F\bigr)
	\xrightarrow{\,\,\wedge\,\,\,}\mathbb C.
\end{equation}
In view of \eqref{E:PDLF}, the isomorphisms induced on cohomology by $L^{F^*\otimes\mathcal O}$ and $L^F$, see \eqref{E:HL}, intertwine this pairing with the Poincar\'e duality pairing
\[
	H^q(M;F^*\otimes\mathcal O)\otimes H^{m-q}(M;F)\xrightarrow{\,\,\wedge\,\,\,}\mathbb C.
\]
In particular, the pairing in \eqref{E:wedgeHH} is non-degenerate.

\begin{theorem}\label{T:PD}
Via the canonical isomorphism of determinant lines
\[
	\sdet\bigl(H^*(M;F^*\otimes\mathcal O)\bigr)
	=\bigl(\sdet(H^*(M;F)\bigr)^{(-1)^{m+1}}
\]
induced by Poincar\'e duality $H^q(M;F^*\otimes\mathcal O)=H^{m-q}(M;F)^*$, we have
\[
	\|-\|^{\sdet(H^*(M;F^*\otimes\mathcal O))}_{\mathcal F,\tilde g,h^{-1}}
	=\left(\|-\|^{\sdet(H^*(M;F))}_{\mathcal F,\tilde g,h}\right)^{(-1)^{m+1}}.
\]
Here $h^{-1}$ denotes the fiberwise Hermitian inner product on $F^*\otimes\mathcal O$ induced by the fiberwise Hermitian inner product $h$ on $F$.
\end{theorem}

\begin{proof}
From the preceding lemma and Stokes' theorem we obtain the following commutative diagram:
\[
\xymatrix{
	\Gamma^\infty\bigl(\mathcal H^q(\mathfrak tM)\otimes F^*\otimes\mathcal O\bigr)\ar[rrr]^-{-(-1)^qD_q^{F^*\otimes\mathcal O}}\ar[d]^-\cong_-{\wedge}
	&&&\Gamma^\infty\bigl(\mathcal H^{q+1}(\mathfrak tM)\otimes F^*\otimes\mathcal O\bigr)\ar[d]^-\cong_-{\wedge}
	\\
	\Gamma^\infty\bigl((\mathcal H^{m-q}(\mathfrak tM)\otimes F)^*\otimes|\Lambda|\bigr)\ar[rrr]^-{(D^F_{m-q-1})^t}
	&&&\Gamma^\infty\bigl((\mathcal H^{m-q-1}(\mathfrak tM)\otimes F)^*\otimes|\Lambda|\bigr)
	\\
	\Gamma^\infty\bigl(((\mathcal H(\mathfrak tM)\otimes F)')^{q-m}\bigr)\ar[rrr]^-{(D^F)^t_{q-m}}\ar@{=}[u]
	&&&\Gamma^\infty\bigl(((\mathcal H(\mathfrak tM)\otimes F)')^{q+1-m}\bigr)\ar@{=}[u]
}
\]
Here the vertical arrows denote the vector bundle isomorphisms
\[	
	\mathcal H^q(\mathfrak tM)\otimes F^*\otimes\mathcal O
	\xrightarrow{\,\,\wedge\,\,\,}\bigl(\mathcal H^{m-q}(\mathfrak tM)\otimes F\bigr)^*\otimes|\Lambda|
	=\left(\bigl(\mathcal H(\mathfrak tM)\otimes F\bigr)'\right)^{q-m}
\]
corresponding to the pairing in \eqref{E:wedgeHF}.
Via the latter vector bundle isomorphism, the Hermitian metric $\tilde g^{-1}\otimes h^{-1}$ on $\mathcal H^q(\mathfrak tM)\otimes F^*\otimes\mathcal O$ corresponds to the Hermitian metric $h':=(\tilde g^{-1}\otimes h)^{-1}\otimes\mu_{\tilde g}^{-2}$ on the right hand side, cf.\ \eqref{E:wedgeiso}.
Hence, by Proposition~\ref{P:shift} and Remark~\ref{R:signs},
\[
	\|-\|^{\sdet(H^*(M;F^*\otimes\mathcal O))}_{\mathcal F,\tilde g,h^{-1}}
	=\left(\|-\|_{h',\mu_{\tilde g}}^{\sdet(H^*((\mathcal H(\mathfrak tM)\otimes F)',(D^F)^t))}\right)^{(-1)^m}.
\]
Moreover, from Proposition~\ref{P:dual} we obtain
\[
	\|-\|^{\sdet(H^*((\mathcal H(\mathfrak tM)\otimes F)',(D^F)^t))}_{h',\mu_{\tilde g}}
	=\left(\|-\|^{\sdet(H^*(M;F))}_{\mathcal F,\tilde g,h}\right)^{-1}.
\]
Combining the preceding to equations, we obtain the theorem.
\end{proof}

Duality can also be understood in terms of an analogue of the Hodge star operator, cf.~\cite[Section~2]{R99} or \cite[Proposition~2.8]{R01}.
To this end, we let $\star\colon\Lambda^q\mathfrak t^*M\to\Lambda^{m-q}\mathfrak t^*M\otimes\mathcal O$ denote the star operator characterized by
\begin{equation}\label{E:stardef}
	\alpha\wedge\star\beta
	=\tilde g^{-1}(\alpha,\beta)\mu_{\tilde g}
\end{equation}
for $\alpha,\beta\in\Lambda^q\mathfrak t^*_xM$ where we are using the canonical identification given by the second and third equality in \eqref{E:Lmcan}.
Recall that the star operator is isometric, i.e.,
\begin{equation}\label{E:stariso}
	\star^*=\star^{-1}.
\end{equation}
Furthermore, \eqref{E:partialwedge} and \eqref{E:stardef} yield \cite{R99,R01}
$$
	\partial^*\alpha=(-1)^q\star^{-1}\partial^{\mathcal O}\star\alpha
$$
for $\alpha\in\Lambda^q\mathfrak t_x^*M$ where $\partial^{\mathcal O}:=\partial\otimes\id_{\mathcal O}$.

The star restricts to vector bundle isomorphisms
$$
\star\colon\mathcal H^q(\mathfrak tM)\to\mathcal H^{m-q}(\mathfrak tM)\otimes\mathcal O,
$$
cf.\ \eqref{E:LtMdeco}.
From Lemma~\ref{L:D*} we obtain \cite{R99,R01}
\[
	(D^F_q)^*
	=(-1)^{q+1}(\star\otimes h)^{-1}\circ D^{F^*\otimes\mathcal O}_{m-q-1}\circ(\star\otimes h).
\]
Using \eqref{E:stariso} this also yields
\[
	D^F_q=(-1)^{q+1}(\star\otimes h)^{-1}\circ\bigl(D^{F^*\otimes\mathcal O}_{m-q-1}\bigr)^*\circ(\star\otimes h).
\]
Hence,
\begin{align*}
	(D^F_q)^*D^F_q
	&=(\star\otimes h)^{-1}D_{m-q-1}^{F^*\otimes\mathcal O}\bigl(D_{m-q-1}^{F^*\otimes\mathcal O}\bigr)^*(\star\otimes h),
\\
	D^F_{q-1}(D^F_{q-1})^*
	&=(\star\otimes h)^{-1}\bigl(D_{m-q}^{F^*\otimes\mathcal O}\bigr)^*D_{m-q}^{F^*\otimes\mathcal O}(\star\otimes h).
\end{align*}
Since 
\begin{equation}\label{E:kkq}
	k_q=k_{m-q-1}
\end{equation}
we have $a_q=a_{m-q-1}$, see \eqref{E:kkiai}.
Hence, we may use the same numbers $a_q$ for the Rumin--Seshadri operator associated with flat bundle $F^*\otimes\mathcal O$, and obtain \eqref{E:Deltai} 
\begin{equation}\label{E:Delta*}
	\Delta^F_{\tilde g,h}
	=\Delta^{\bar F}_{\tilde g,\bar h}
	=(\star\otimes h)^{-1}\Delta^{F^*\otimes\mathcal O}_{\tilde g,h^{-1}}(\star\otimes h).
\end{equation}
In particular,
\begin{equation}\label{E:Q*}
	Q_{\lambda,\tilde g,h}^F
	=Q_{\lambda,\tilde g,\bar h}^{\bar F}
	=(\star\otimes h)^{-1}Q^{F^*\otimes\mathcal O}_{\lambda,\tilde g,h^{-1}}(\star\otimes h).
\end{equation}
In view of \eqref{E:kkq}, the numbers $\tilde N_q:=-N_{m-q}$ also satisfy the relation \eqref{E:Niik}, that is, $\tilde N_{q+1}-\tilde N_q=k_q$.
Clearly,
\begin{equation}\label{E:N*}
	N=-(\star\otimes h)^{-1}\tilde N(\star\otimes h).
\end{equation}
Combining \eqref{E:Delta*}, \eqref{E:Q*} and \eqref{E:N*}, we obtain 
\begin{equation}\label{E:PDzeta}
	\zeta_{\lambda,\tilde g,h}^F(s)
	=\zeta_{\lambda,\tilde g,\bar h}^{\bar F}(s)
	=(-1)^{m+1}\zeta_{\lambda,\tilde g,h^{-1}}^{F^*\otimes\mathcal O}(s),
\end{equation}
cf.~\eqref{E:zetalambda} and \eqref{E:zetabar}.
Recall here that according to Lemma~\ref{L:Ni}, we may use any sequence of numbers satisfying \eqref{E:Niik} to compute this zeta function.

The relation~\eqref{E:PDzeta} immediately leads to a slightly different proof of Theorem~\ref{T:PD}.
Let us spell out two more consequences.

\begin{proposition}\label{P:zetaeven}
If $M$ is an orientable manifold of even dimension and $h$ is a parallel Hermitian metric on $F$, then $\zeta_{\lambda,\tilde g,h}^F(s)=0$.
\end{proposition}

\begin{proof}
Since $h$ is parallel, it provides, together with an orientation of $M$, an isomorphism of flat vector bundles $h\colon\bar F\to F^*\otimes\mathcal O$ which maps the Hermitian metric $\bar h$ on $\bar F$ to the Hermitian metric $h^{-1}$ on $F^*\otimes\mathcal O$.
Hence, $\zeta^{\bar F}_{\lambda,\tilde g,\bar h}(s)=\zeta^{F^*\otimes\mathcal O}_{\lambda,\tilde g,h^{-1}}(s)$.
The proposition now follows from \eqref{E:PDzeta}.
\end{proof}

From \eqref{E:Delta*} we also obtain
\[
	p^F_{q,j}=p^{\bar F}_{q,j}=(\star\otimes h)^{-1}p^{F^*\otimes\mathcal O}_{m-q,j}(\star\otimes h)
\]
where $p_{q,j}^F\in\Gamma^\infty\bigl(\eend(\mathcal H^q(\mathfrak tM)\otimes F)\otimes|\Lambda|\bigr)$ denote the degree $q$ part of the local quantities in the heat kernel asymptotics associated with the Rumin--Seshadri operator $\Delta^F_{\tilde g,h}$, cf.~\eqref{E:heatkernasymp}.
Using locality, this gives
\begin{equation}\label{E:pFF*}
	\tr\left(p^F_{q,j}\right)
	=\tr\left(p^{\bar F}_{q,j}\right)
	=\tr\left(p^{F^*}_{m-q,j}\right).
\end{equation}

\begin{proposition}\label{P:strpodd}
If the dimension of $M$ is odd and if $h$ is a parallel Hermitian metric on $F$, then $\str(p_j^F)=0$ for all $j\in\mathbb N_0$.
\end{proposition}

\begin{proof}
Since $h$ is parallel, it provides an isomorphism of flat vector bundles $h\colon\bar F\to F^*$ which maps the Hermitian metric $\bar h$ on $\bar F$ to the Hermitian metric $h^{-1}$ on $F^*$.
Hence, $p^{\bar F}_{q,j}=p^{F^*}_{q,j}$.
Combining this with \eqref{E:pFF*}, we see that the summands in $\str(p^F_j)$ cancel pairwise if the dimension is odd.
\end{proof}

\subsection{Variation of the metrics}\label{SS:RSvar}

In this section we apply Theorem~\ref{T:var} to determine how the analytic torsion of a Rumin complex depends on the fiberwise inner products $\tilde g$ on $\mathfrak tM$ and $h$ on $F$.

If $A\in\Gamma^\infty(\Aut(\mathfrak tM))$ is a fiberwise automorphism of graded (nilpotent) Lie algebras, we let $\mathcal H(A)\in\Gamma^\infty(\Aut(\mathcal H(\mathfrak tM)))$ denote the automorphism in fiberwise cohomology induced by the inverse $A^{-1}$. 
Hence, we have covariant functoriality $\mathcal H(AB)=\mathcal H(A)\mathcal H(B)$ just like for the dual representation of a group, $A,B\in\Gamma^\infty(\Aut(\mathfrak tM))$.
If $\dot A\in\Gamma^\infty(\der(\mathfrak tM))$ is a fiberwise derivation of graded (nilpotent) Lie algebras, we define $\mathcal H(\dot A)\in\Gamma^\infty(\eend(\mathcal H(\mathfrak tM)))$ by $\mathcal H(\dot A):=\frac\partial{\partial t}|_{t=0}\mathcal H(\exp(t\dot A))$ so that $\mathcal H([\dot A,\dot B])=[\mathcal H(\dot A),\mathcal H(\dot B)]$ for $\dot A,\dot B\in\Gamma^\infty(\der(\mathfrak tM))$.

The following generalizes anomaly formulas for the Ray--Singer \cite{RS71} torsion \cite[Theorem~0.1]{BZ92} and the Rumin--Seshadri analytic torsion, see \cite[Corollary~3.7]{RS12}.

\begin{theorem}\label{T:varg}
Suppose the graded fiberwise Euclidean inner product $\tilde g_u$ on $\mathfrak tM$ and the fiberwise Hermitian inner product $h_u$ on $F$ depend smoothly on a real parameter $u$ such that $\tilde g_v^{-1}\tilde g_u\in\Gamma^\infty(\Aut(\mathfrak tM))$ is a fiberwise automorphism of graded (nilpotent) Lie algebras, for all $u$ and $v$.
Then
\[
	\tfrac\partial{\partial u}\log\|-\|^{\sdet(H^*(M;F))}_{\mathcal F,\tilde g_u,h_u}
	=\frac12\int_M\str\Bigl(\bigl(\mathcal H(\dot g_u)+\tfrac12\tr(\dot g_u)+\dot h_u\bigr)p_{u,n}^F\Bigr)
\]
where $\dot g_u:=\tilde g_u^{-1}\frac\partial{\partial u}\tilde g_u\in\Gamma^\infty\bigl(\der(\mathfrak tM)\bigr)$, 
$\dot h_u:=h_u^{-1}\frac\partial{\partial u}h_u\in\Gamma^\infty(\eend(F))$, and $p_{u,n}^F\in\Gamma^\infty\bigl(\eend(\mathcal H(\mathfrak tM)\otimes F)\otimes|\Lambda|\bigr)$ denotes the constant term in the heat kernel expansion associated with the Rumin--Seshadri operator $\Delta^F_{\tilde g_u,h_u}$.
\end{theorem}

\begin{proof}
Recall that we are using the fiberwise Hermitian inner products $\tilde g_u^{-1}\otimes h_u$ on the Rumin complex $\mathcal H(\mathfrak tM)\otimes F$, where $\tilde g_u^{-1}$ denotes the fiberwise Euclidean inner product on $\mathcal H(\mathfrak tM)$ obtained by restriction via the fiberwise Hodge decomposition \eqref{E:LtMdeco} from the induced Euclidean inner product on $\Lambda^*\mathfrak t^*M$ which is also denoted by $\tilde g^{-1}_u$.
Put $A_{v,u}=\tilde g_v^{-1}\tilde g_u\in\Gamma^\infty(\Aut(\mathfrak tM))$ and let $A_{v,u}^t=\tilde g_u\tilde g_v^{-1}\in\Gamma^\infty(\Aut(\mathfrak t^*M))$ denote the fiberwise dual automorphism.
Extend this to an automorphism of the exterior bundle, $A_{v,u}^t\in\Gamma^\infty(\Aut(\Lambda^*\mathfrak t^*M))$ characterized by
\[
	A_{v,u}^t(\alpha\wedge\beta)=A_{v,u}^t\alpha\wedge A_{v,u}^t\beta,
\]
for $\alpha,\beta\in\Lambda^*\mathfrak t^*_xM$.
Note that the equality 
\begin{equation}\label{E:elvis3}
	A_{v,u}^t=\tilde g_u\tilde g_v^{-1}
\end{equation} 
holds on $\Lambda^*\mathfrak t^*M$ for we are using the induced Euclidean inner products on the exterior bundle.
Since $A_{v,u}$ is a fiberwise automorphism of Lie algebras, $A_{v,u}^t$ commutes with fiberwise Chevalley--Eilenberg codifferential $\partial$, the fiberwise adjoint $\partial^*$ does not depend on $u$, and $A_{v,u}^t$ commutes with $\partial^*$ too.
Hence, the decomposition in \eqref{E:LtMdeco} is independent of $u$ and invariant under $A_{v,u}^t$.
Hence, 
\begin{equation}\label{E:elvis4}
	A_{v,u}^t|^{-1}_{\mathcal H(\mathfrak tM)}=\mathcal H(A_{v,u}).
\end{equation}
Combining \eqref{E:elvis3} and \eqref{E:elvis4}, we conclude
\[
	(\tilde g_v^{-1}\otimes h_v)^{-1}(\tilde g_u^{-1}\otimes h_u)|_{\mathcal H(\mathfrak tM)\otimes F}=\mathcal H(\tilde g_v^{-1}\tilde g_u)\otimes h_v^{-1}h_u.
\]
Differentiating, we find
\[
	(\tilde g_u^{-1}\otimes h_u)^{-1}\tfrac\partial{\partial u}(\tilde g_u^{-1}\otimes h_u)|_{\mathcal H(\mathfrak tM)\otimes F}=\mathcal H(\dot g_u)\otimes\id_F+\id_{\mathcal H(\mathfrak tM)}\otimes\dot h_u.
\]
Furthermore, we clearly have $\mu_{\tilde g_u}={\det}^{1/2}(\tilde g_v^{-1}\tilde g_u)\mu_{\tilde g_v}$ and thus 
\[
	\mu^{-1}_{\tilde g_u}\tfrac\partial{\partial u}\mu_{\tilde g_u}=\tfrac12\tr(\dot g_u).
\]
The statement thus follows from Theorem~\ref{T:var}.
\end{proof}

Recall that $\phi_t$ denotes the grading automorphism of $\mathfrak t_xM=\bigoplus_p\mathfrak t^p_xM$ given by multiplication with $t^p$ on the summand $\mathfrak t_x^pM$, $t>0$.
For a positive smooth function $f$ on $M$, we let $\phi_f\in\Gamma^\infty(\Aut(\mathfrak tM))$ denote the fiberwise Lie algebra automorphism acting by $\phi_{f(x)}$ on the fiber $\mathfrak t_xM$.

\begin{corollary}\label{C:varg2}
Suppose $f$ is a real valued smooth function on $M$ and consider a family of Euclidean inner products of the form $\tilde g_u=\tilde g\phi_{\exp(uf)}$ on $\mathfrak tM$.
Then
\[
	\tfrac\partial{\partial u}\log\|-\|^{\sdet(H^*(M;F))}_{\mathcal F,\tilde g_u,h}
	=\frac12\int_Mf\str\left((N-N_0-\tfrac n2)p_{u,n}^F\right),
\]
where $p_{u,n}^F\in\Gamma^\infty\bigl(\eend(\mathcal H(\mathfrak tM)\otimes F)\otimes|\Lambda|\bigr)$ denotes the constant term in the heat kernel expansion associated with the Rumin--Seshadri operator $\Delta^F_{\tilde g_u,h}$.
\end{corollary}

\begin{proof}
Since $\dot g_u=\tilde g_u^{-1}\tfrac\partial{\partial u}\tilde g_u$ acts by $pf$ on $\mathfrak t^pM$, we have cf.~\eqref{E:n}
\[
	\tr(\dot g_u)=-nf.
\]
Using \eqref{E:Niik}, \eqref{E:kqpq} and $p_0=0$, we find $N_q-N_0=p_q$.
Combining this with \eqref{E:deltaHq}, we obtain $\mathcal H^q(\phi_{\exp(uf)})=(e^{uf})^{p_q}=e^{u(N_q-N_0)f}$.
Hence, $\mathcal H^q(\dot g_u)=(N_q-N_0)f$ and
\[
	\mathcal H(\dot g_u)=(N-N_0)f.
\]
The statement thus follows from Theorem~\ref{T:varg}.
\end{proof}

\begin{remark}
Conformal invariance of the analytic torsion, that is, independence under scaling of the Euclidean inner product on $\mathfrak tM$ as in Corollary~\ref{C:varg2}, thus is equivalent to the pointwise vanishing of the local quantity $\str\bigl((N-N_0-\frac n2)p_n^F\bigr)$.
Using Lemma~\ref{L:Euler} and \eqref{E:zeta00}, we find
\begin{equation}\label{E:intzeta0}
	\int_M\str\bigl((N-N_0-\tfrac n2)p^F_n\bigr)=\chi'(M;F)-(N_0+\tfrac n2)\chi(M;F)+\zeta^F_{\lambda=0}(0),
\end{equation}
where $\chi(M;F)=\sum_q(-1)^q\dim H^q(M;F)=\rk(F)\chi(M)$ denotes the Euler characteristics and $\chi'(M;F)=\sum_q(-1)^qN_q\dim H^q(M;F)$.
In view of \eqref{E:dduzeta} the integral in \eqref{E:intzeta0} is independent of $\tilde g$ and $h$, hence this is a smooth invariant of the filtered manifold $M$ and the flat bundle $F$.
Actually, this invariant only depends on $F$ and the homotopy class of the underlying filtration, see Remark~\ref{R:homotop.inv} below for a more precise statement.
In the trivially filtered case, this vanishes according to \cite[Theorem~7.10]{BZ92}.
A discussion of the contact case can be found at the end of Section~3.2 in \cite{RS12}.
Whether \eqref{E:intzeta0} vanishes in general is unclear.
If $m$ is odd and $h$ is parallel then $\str(p_n^F)=0$ according to Proposition~\ref{P:strpodd}.
\end{remark}

\begin{remark}\label{R:hpara}
If $h_u$ are parallel Hermitian inner products on $F$, then the integrands in the preceding statements can be simplified somewhat:
\begin{align*}
	\str\bigl(\mathcal H(\dot g_u)p_{u,n}^F\bigr)
	&=\rk(F)\str\bigl(\mathcal H(\dot g_u)p_{u,n}\bigr)
	\\
	\str\bigl(\dot h_up_{u,n}^F\bigr)
	&=\tr(\dot h_u)\str(p_{u,n})
	\\
	\str\bigl(p_{u,n}^F\bigr)
	&=\rk(F)\str(p_{u,n})
	\\
	\str\bigl(Np_{u,n}^F\bigr)
	&=\rk(F)\str\bigl(Np_{u,n}\bigr)
\end{align*}
Here $p_{u,n}\in\Gamma^\infty\bigl(\eend(\mathcal H(\mathfrak tM))\otimes|\Lambda|\bigr)$ denotes the constant term in the heat kernel expansion for the corresponding Rumin--Seshadri operator associated with the trivial flat line bundle over $M$ equipped with the canonical Hermitian metric.
This follows from Remark~\ref{R:twist}.
\end{remark}

\subsection{Deformation of the filtration}\label{SS:def.filt}

In this section we consider a family of closed filtered manifolds $M_u$ smoothly depending on a real parameter $u$.
More explicitly, we assume the underlying manifold $M$ remains fixed, and the filtration subbundles $T^pM_u\subseteq TM$ depend smoothly on $u$.
As before, we assume that the osculating algebras of $M_u$ have pure cohomology of locally constant dimension.
Moreover, suppose $\tilde g_u$ is a smooth family of graded fiberwise Euclidean metrics on $\mathfrak tM_u$.
For any flat bundle $F$ with Hermitian metric $h$, we obtain an analytic torsion $\|-\|^{\sdet(H^*(M;F))}_{\mathcal F_u,\tilde g_u,h}$ on the determinant line $\sdet(H^*(M;F))$.
Below, we will describe qualitatively how the this torsion depends on the parameter $u$.

Clearly, the bundle of osculating algebras $\mathfrak tM_u$ depends smoothly on $u$ also.
For simplicity, we will assume that these are isomorphic as bundles of graded Lie algebras.
Hence, there exist isomorphisms of graded vector bundles $\tilde\psi_{v,u}\colon\mathfrak tM_u\to\mathfrak tM_v$ which intertwine the fiberwise Lie algebra structures and depend smoothly on the parameters $u$ and $v$.
W.l.o.g.\ we may assume $\tilde\psi_{w,v}\tilde\psi_{v,u}=\tilde\psi_{w,u}$ and in particular $\tilde\psi_{u,u}=\id$.
Let $S_u\colon\mathfrak tM_u\to TM_u$ be splittings of the filtration which depend smoothly on $u$ and define $\psi_{v,u}\in\Gamma^\infty(\Aut(TM))$ by the equation
\begin{equation}\label{E:Spsiuv}
        \psi_{v,u}S_u=S_v\tilde\psi_{v,u}.
\end{equation}
Then $\psi_{v,u}\colon TM_u\to TM_v$ is a filtration preserving isomorphism, $\psi_{v,u}(T^pM_u)=T^pM_v$, which induces $\tilde\psi_{v,u}\colon\mathfrak tM_u\to\mathfrak tM_v$ on the associated graded.
Moreover, $\psi_{w,v}\psi_{v,u}=\psi_{w,u}$ and $\psi_{u,u}=\id$.
Since we have already discussed the dependence on the metric before, we will consider a family of fiberwise Euclidean metrics $\tilde g_u$ on $\mathfrak tM_u$ such that
\begin{equation}\label{E:gpsiuv}
        \tilde\psi_{v,u}^*\tilde g_v=\tilde g_u.
\end{equation}

\begin{theorem}\label{T:def.filt}
In this situation we have
\[
        \tfrac\partial{\partial u}\log\|-\|^{\sdet(H^*(M;F))}_{\mathcal F_u,\tilde g_u,h}=\frac12\int_M\str(\tilde p_{n,u}^F)
\]
where $\tilde p^F_{n,u}\in\Gamma^\infty\bigl(\eend(\mathcal H(\mathfrak tM_u)\otimes F)\otimes|\Lambda|\bigr)$ is the constant term in the asymptotic expansion, as $t\to0$,
\begin{equation}\label{E:asym.filt}
        \Bigl(\bigl(\dot A_u^F+(\dot A^F_u)^*-\tr(\dot\psi_u)\bigr)k^F_{t,u}\Bigr)(x,x)\sim\sum_{j=-r_u}^\infty t^{(j-n)/2\kappa}\tilde p_{j,u}^F(x).
\end{equation}
Here $\dot A^F_u$ is the differential operator given by the expression \eqref{E:dotAuF} below, $\dot\psi_u:=\frac\partial{\partial v}|_{v=u}\psi_{v,u}\in\Gamma^\infty(\eend(TM))$, and $k^F_{t,u}$ denotes the heat kernel of the Rumin--Seshadri operator $\Delta^F_{\tilde g_u,h}$.
In particular, $\tilde p^F_{n,u}$ is locally computable.
More precisely, $\tilde p^F_{n,u}(x)$ can be computed from the germ of $(M,\mathcal F_u,\tilde g_u,F,h,S_u,\dot\psi_u)$ at $x$.
Furthermore, $\tilde p^F_{n,u}=0$ if the homogeneous dimension $n$ is odd.
\end{theorem}

\begin{proof}
The differential operators $A_{v,u}\colon\Gamma^\infty(\mathcal H(\mathfrak tM_u))\to\Gamma^\infty(\mathcal H(\mathfrak tM_v))$
\[
        A_{v,u}:=\iota_v^{-1}\tilde\Pi_v\mathbf L_v^{-1}L_u
\]
satisfy $A_{u,u}=\id$ and $A_{v,u}D_u=D_vA_{v,u}$, see \eqref{E:dLLD} and \eqref{E:itPL}.
Twisting with the flat bundle $F$, we get differential operators $A^F_{v,u}\colon\Gamma^\infty(\mathcal H(\mathfrak tM_u)\otimes F)\to\Gamma^\infty(\mathcal H(\mathfrak tM_v)\otimes F)$ such that $A_{v,u}^FD_u^F=D_v^FA_{v,u}^F$.
Clearly, $A_{v,u}^F$ induces the identity on $H^*(M;F)$, up to the canonical identifications in \eqref{E:HL}.
In order to apply Theorem~\ref{T:var.all}, we have to analyze $\dot A^F_u$.
It will be convenient to rewrite $A_{v,u}$ in the form, see \eqref{E:LPiS},
\begin{equation}\label{E:Auv.var}
        A_{v,u}=\iota_v^{-1}\tilde\Pi_vS_v^{-1}\Pi_vL_u.
\end{equation}

We let $\psi_{v,u}\colon\Lambda^*T^*M_u\to\Lambda^*T^*M_v$ and $\tilde\psi_{v,u}\colon\Lambda^*\mathfrak t^*M_u\to\Lambda^*\mathfrak t^*M_v$ also denote the isomorphisms induced by $\psi_{v,u}\colon TM_u\to TM_v$ and $\tilde\psi_{v,u}\colon\mathfrak tM_u\to\mathfrak tM_v$, respectively.
Then equation~\ref{E:Spsiuv} remains true on $\Lambda^*\mathfrak t^*M_u$ if $S_u\colon\Lambda^*\mathfrak t^*M_u\to\Lambda^*T^*M_u$ denotes the splitting induced by $S_u\colon\mathfrak tM_u\to TM_u$ as in \eqref{E:Swedge}.
Furthermore, \eqref{E:gpsiuv} yields
\begin{equation}\label{E:psi.tgu}
        \tilde\psi_{v,u}^*\tilde g_v^{-1}=\tilde g_u^{-1}
\end{equation}
where $\tilde g_u^{-1}$ denotes the induced Euclidean metric on $\Lambda^*\mathfrak t^*M_u$ as before.
We have
\begin{equation}\label{E:partialpsiuv}
        \tilde\psi_{v,u}\partial_u=\partial_v\tilde\psi_{v,u}
\end{equation}
since $\tilde\psi_{v,u}$ is an isomorphism of graded Lie algebra bundles.
By \eqref{E:psi.tgu} this gives
\begin{equation}\label{E:partial*psiuv}
        \tilde\psi_{v,u}\partial_u^*=\partial_v^*\tilde\psi_{v,u}
\end{equation}
and combination with \eqref{E:Spsiuv} also yields
\begin{equation}\label{E:deltapsiuv}
        \psi_{v,u}\delta_u=\delta_v\psi_{v,u},
\end{equation}
see~\eqref{E:delta}.
Moreover,
\begin{equation}\label{E:tPipsiuv}
        \tilde\psi_{v,u}\tilde\Pi_u=\tilde\Pi_v\tilde\psi_{v,u}
        \qquad\text{and}\qquad
        \tilde\psi_{v,u}\iota_u=\iota_v\mathcal H(\tilde\psi_{v,u})
\end{equation}
where $\mathcal H(\tilde\psi_{v,u})\colon\mathcal H(\mathfrak tM_u)\to\mathcal H(\mathfrak tM_v)$ denotes the induced isomorphism, cf.~\eqref{E:partialpsiuv}.
From \eqref{E:Spsiuv}, \eqref{E:Auv.var} and \eqref{E:tPipsiuv} we obtain
\begin{equation}\label{E:HpsiA}
        \mathcal H(\tilde\psi_{v,u})^{-1}A_{v,u}=\iota_u^{-1}\tilde\Pi_uS_u^{-1}(\psi_{v,u}^{-1}\Pi_v\psi_{v,u})\psi_{v,u}^{-1}L_u.
\end{equation}
When applying Theorem~\ref{T:var.all} we use an auxiliary linear connection on the bundle $\bigsqcup_u\mathcal H(\mathfrak tM_u)\otimes F$ such that $\dot A^F_u=\bigl(\frac\partial{\partial v}|_{v=u}\mathcal H(\tilde\psi_{v,u})^{-1}A_{v,u}\bigr)^F$.
Differentiating \eqref{E:HpsiA} yields
\begin{equation}\label{E:dotAu}
        \dot A_u^F=\Bigl(\iota_u^{-1}\tilde\Pi_uS_u^{-1}\tfrac\partial{\partial v}|_{v=u}(\psi_{v,u}^{-1}\Pi_v\psi_{v,u})L_u-\iota_u^{-1}\tilde\Pi_uS_u^{-1}\Pi_u\dot\psi_uL_u\Bigr)^F.
\end{equation}
where $\dot\psi_u=\frac\partial{\partial v}|_{v=u}\psi_{v,u}\in\Gamma^\infty(\der(\Lambda^*T^*M))$ denotes the fiberwise graded derivation given by contraction with $\dot\psi_u\in\Gamma^\infty(\eend(TM))$.
Since $\Box_v=[\delta_v,d]$ we obtain from \eqref{E:deltapsiuv}
\begin{equation}\label{E:boxpsiuv}
        \psi_{v,u}^{-1}\Box_v\psi_{v,u}=[\delta_u,\psi_{v,u}^{-1}d\psi_{v,u}]
\end{equation}
and thus
\begin{equation}\label{E:ddboxpsiuv}
        \tfrac\partial{\partial v}|_{v=u}\bigl(\psi_{v,u}^{-1}\Box_v\psi_{v,u}\bigr)
        =[\delta_u,[d,\dot\psi_u]].
\end{equation}
Combining this with \eqref{E:Pi}, we find
\begin{equation}\label{E:dotPi}
        \tfrac\partial{\partial v}|_{v=u}\bigl(\psi_{v,u}^{-1}\Pi_v\psi_{v,u}\bigr)
        =\frac1{2\pi\mathbf i}\oint_{|z|=\varepsilon}(z-\Box_u)^{-1}[\delta_u,[d,\dot\psi_u]](z-\Box_u)^{-1}dz
\end{equation}
Combining this with $[\delta_u,\Box_u]=0$, $\delta_uL_u=0$ and $\tilde\Pi_uS_u^{-1}\delta_u=0$, we conclude that the first summand in \eqref{E:dotAu} vanishes.
Hence, using \eqref{E:LPiS} we obtain
\begin{equation}\label{E:dotAuF}
        \dot A_u^F
        =-\Bigl(\iota_u^{-1}\tilde\Pi_uS_u^{-1}\Pi_u\dot\psi_uL_u\Bigr)^F
        =-\bigl(\pi_u\Pi_u\dot\psi_uL_u\bigr)^F.
\end{equation}
In particular, the coefficients of the differential operator $\dot A_u^F$ at $x$ can be expressed in terms of the germ of $(M,\mathcal F_u,\tilde g_u,F,h,S_u,\dot\psi_u)$ at $x$.

Clearly, $\mu_{\tilde g_v}=\psi_{v,u}(\mu_{\tilde g_u})=\det_{TM}(\psi_{v,u}^{-1})\mu_{\tilde g_u}$ and thus
\[
        \mu_{\tilde g_u}^{-1}\tfrac\partial{\partial u}\mu_{\tilde g_u}=-\tr_{TM}(\dot\psi_u).
\]
Moreover, \eqref{E:psi.tgu} and \eqref{E:tPipsiuv} give $\mathcal H(\tilde\psi_{v,u})^*(\tilde g_v^{-1}|_{\mathcal H(\mathfrak tM_v)})=\tilde g_u^{-1}|_{\mathcal H(\mathfrak tM_u)}$ and therefore
\[
        (\tilde g_u^{-1}\otimes h)^{-1}\tfrac\partial{\partial v}\big|_{v=u}(\mathcal H(\tilde\psi_{v,u})\otimes\id_F)^*\bigl((\tilde g_v^{-1}\otimes h)\big|_{\mathcal H(\mathfrak tM_v)\otimes F}\bigr)=0.
\]
The statement thus follows from Theorem~\ref{T:var.all}.
\end{proof}

\begin{remark}\label{R:dotpsi}
Note that $\dot A^F_u$, $\tr(\dot\psi_u)$ and $\tilde p_{j,u}^F$ in Theorem~\ref{T:def.filt} remain unchanged if $\psi_{v,u}$ is replaced by any other filtration preserving family $\psi_{v,u}'\colon TM_u\to TM_v$ satisfying $\psi'_{w,v}\psi'_{v,u}=\psi'_{w,u}$ and inducing the same automorphism $\mathfrak tM_u\to\mathfrak tM_v$ on the associated graded, that is, $\gr(\psi'_{v,u})=\gr(\psi_{v,u})=\tilde\psi_{v,u}$.
Indeed, in this case there exists a smooth family of filtration preserving automorphism $B_{v,u}$ of $TM_u$ such that $\psi_{v,u}'=\psi_{v,u}B_{v,u}$ and $\gr(B_{v,u})=\id$.
Differentiating, we obtain $\dot\psi'_u=\dot\psi_u+\dot B_u$ where $\dot B_u:=\frac\partial{\partial v}|_{v=u}B_{v,u}$ is a filtration preserving endomorphism of $TM_u$ inducing $\gr(\dot B_u)=0$ on $\mathfrak tM_u$.
In particular, $\tr(\dot B_u)=0$ and thus $\tr(\dot\psi'_u)=\tr(\dot\psi_u)$.
Moreover, $\gr(\pi_u\Pi_u\dot B_uL_u)=\gr(\pi_u\Pi_u)\gr(\dot B_u)\gr(L_u)=0$ since each factor is filtration preserving.
As $\mathfrak tM_u$ has pure cohomology, we conclude $\pi_u\Pi_u\dot B_uL_u=0$, cf.~the proof of Lemma~\ref{L:Sindep}.
We obtain $\pi_u\Pi_u\dot\psi'_uL_u=\pi_u\Pi_u\dot\psi_uL_u$ and thus $\dot A^F_u$ remains unchanged, see \eqref{E:dotAuF}.
Clearly this implies that $\tilde p_{j,u}^F$ remains unchanged as well, see \eqref{E:asym.filt}.
\end{remark}

\begin{remark}\label{R:homotop.inv}
Using \eqref{E:dduzetaa}, we conclude that the quantity in \eqref{E:intzeta0} remains unchanged if the filtration on $M$ is smoothly deformed as above, i.e., in such a way that the bundles of osculating algebras remain isomorphic.
Hence, this quantity only depends on $F$ and the corresponding homotopy class of the filtration.
\end{remark}

\subsection{Comparison with the Ray--Singer torsion}\label{S:compRS}

We continue to use the notation set up above.
Moreover, we consider a Riemannian metric $g$ on $M$.
Comparing the torsion of the Rumin complex constructed above, see \eqref{E:deftorM}, with the Ray--Singer torsion \cite{RS71,BZ92}, we obtain a positive real number
\begin{equation}\label{E:Rdef}
	R_{\tilde g;g}^{F,h}(M,\mathcal F):=\frac{\|-\|^{\sdet(H^*(M;F))}_{\mathcal F,\tilde g,h}}{\|-\|^{\sdet(H^*(M;F))}_{\RS,g,h}}.
\end{equation}
For trivially filtered manifolds the analytic torsion constructed above coincides, by definition, with the Ray--Singer torsion.
We will add the decoration RS to our notation if an object is associated to the underlying trivially filtered manifold. 

For contact manifolds, Albin and Quan \cite[Corollary~3]{AQ19} have shown that the quotient \eqref{E:Rdef} can be expressed as an integral over local quantities.
Possibly, their analysis of the sub-Riemannian limit can be extended to filtered manifolds whose osculating algebras have pure cohomology.
We will not pursue this here though.
Returning to the general case we have.

\begin{proposition}\label{P:cont}
Suppose $F_u$ is a complex vector bundle over $M$ with a flat connection that depends smoothly on a real parameter $u$.
Then $R_{\tilde g;g}^{F_u,h}(M,\mathcal F)$ is smooth in $u$.
\end{proposition}

\begin{proof}
Let $\Delta_u=\Delta_{\tilde g,h}^{F_u}$ denote the Rumin--Seshadri operator \eqref{E:RS} and let $\Delta_{\RS,u}=\Delta^{F_u}_{\RS,g,h}$ denote the Hodge Laplacian.
Given $u_0$, we choose $\lambda>0$ such that all non-zero eigenvalues of $\Delta_{u_0}$ and all non-zero eigenvalues of $\Delta_{\RS,u_0}$ are strictly larger than $\lambda$.
By definition, see \eqref{E:defRSmetric} and \eqref{E:deftorM}, we have
\begin{equation}\label{E:RRS}
	R_{\tilde g;g}^{F_u,h}(M,\mathcal F)
	=\frac{\exp(-\frac1{2\kappa}\zeta'_{\lambda,u}(0))}{\exp(-\frac12\zeta'_{\RS,\lambda,u}(0))}
	\cdot\frac{\|-\|^{\sdet(H^*(M;F_u))}_{[0,\lambda],\tilde g,h}}{\|-\|^{\sdet(H^*(M;F_u))}_{\RS,[0,\lambda],g,h}}.
\end{equation}
If $u$ is sufficiently close to $u_0$, then the spectra of $\Delta_u$ and $\Delta_{\RS,u}$ do not contain $\lambda$.
For these $u$, the spectral projectors $P_{\lambda,u}$ and $P_{\RS,\lambda,u}$ depend smoothly on $u$, and so do $\zeta'_{\lambda,u}(0)$ and $\zeta'_{\RS,\lambda,u}(0)$, cf.~\eqref{E:zetalambda}.
Hence, numerator and denominator of the first factor in \eqref{E:RRS} depend smoothly on $u$, provided $u$ remains sufficiently close to $u_0$.
To analyze the second factor in \eqref{E:RRS}, we consider the chain map, cf.~\eqref{E:Lchain},
$$
	\ell_u\colon\bigl(\img(P_{\lambda,u}),D^{F_u}\bigr)\to\bigl(\img(P_{\RS,\lambda,u}),d^{F_u}\bigr),\qquad\ell_u:=P_{\RS,\lambda,u}\circ L^{F_u}|_{\img(P_{\lambda,u})}.
$$
Clearly, $\ell_u$ induces the isomorphism \eqref{E:HL} in cohomology.
By positivity, and since $\Delta_{u_0}$ vanishes on $\img(P_{\lambda,u_0})$, the differential $D^{F_{u_0}}$ vanishes on $\img(P_{\lambda,u_0})$, cf.~\eqref{E:Deltai}.
Similarly, $d^{F_{u_0}}$ vanishes on $\img(P_{\RS,\lambda,u_0})$.
Hence, $\ell_{u_0}$ is an isomorphism of (trivial) chain complexes.
We conclude that $\ell_u$ is an isomorphism of chain complexes, provided $u$ is sufficiently close to $u_0$.
In this situation we have the well known formula 
$$
	\frac{\|-\|^{\sdet(H^*(M;F_u))}_{[0,\lambda],\tilde g,h}}{\|-\|^{\sdet(H^*(M;F_u))}_{\RS,[0,\lambda],g,h}}
	=\sdet^{-1/2}(\ell_u^*\ell_u),
$$
where, for these $u$, the right hand side clearly depends smoothly on $u$.
\end{proof}

From Propositions~\ref{P:sumF}, \ref{P:PFF}, \ref{P:PBF}, and Theorem~\ref{T:PD} we immediately obtain

\begin{proposition}\label{P:R}
In this situation we have:
\begin{enumerate}[(a)]
\item\label{I:Rsum} If $F_1$ and $F_2$ are two flat vector bundles with fiberwise Hermitian inner products $h_1$ and $h_2$, respectively, then
\[
	R_{\tilde g;g}^{F_1\oplus F_2,h_1\oplus h_2}(M,\mathcal F)=R_{\tilde g;g}^{F_1,h_1}(M,\mathcal F)\cdot R_{\tilde g;g}^{F_2,h_2}(M,\mathcal F).
\]
\item\label{I:Rind}
If $\pi\colon\tilde M\to M$ is a finite covering of filtered manifolds and $\tilde F$ is a flat vector bundle over $\tilde M$ with fiberwise Hermitian inner product $\tilde h$, then
\[
	R_{\tilde g;g}^{\pi_*\tilde F,\pi_*\tilde h}(M,\mathcal F)=R_{\pi^*\tilde g;\pi^*g}^{\tilde F,\tilde h}(\tilde M,\pi^*\mathcal F).
\]
\item\label{I:Rres} 
If $\pi\colon\tilde M\to M$ is a finite covering of filtered manifolds, then 
\[
	R_{\pi^*\tilde g;\pi^*g}^{\pi^*F,\pi^*h}(\tilde M,\pi^*\mathcal F)=R_{\tilde g;g}^{F\otimes\mathcal V,h\otimes h^{\mathcal V}}(M,\mathcal F).
\]
\item\label{I:Rdual} With respect to Poincar\'e duality, we have 
\begin{equation}\label{E:Rdual}
	R_{\tilde g;g}^{F^*\otimes\mathcal O,h^{-1}}(M,\mathcal F)
	=\left(R_{\tilde g;g}^{\bar F,\bar h}(M,\mathcal F)\right)^{(-1)^{m+1}}
	=\left(R_{\tilde g;g}^{F,h}(M,\mathcal F)\right)^{(-1)^{m+1}}.
\end{equation}
\end{enumerate}
\end{proposition}

\begin{remark}
For even dimensional orientable $M$ and parallel Hermitian metrics $h$, Equation~\eqref{E:Rdual} yields $R_{\tilde g;g}^{F,h}(M,\mathcal F)=1$, for in this case, $h$ together with an orientation of $M$ provides an isomorphism of flat bundles $\bar F\cong F^*\otimes\mathcal O$ that intertwines the Hermitian metric $\bar h$ with $h^{-1}$, cf.\ Proposition~\ref{P:zetaeven}.
Hence, in this situation	
\[
	\|-\|^{\sdet(H^*(M;F))}_{\mathcal F,\tilde g,h}=\|-\|^{\sdet(H^*(M;F))}_{\RS,g,h},
\]
and, in particular, the analytic torsion on the left hand side is independent of the metric $\tilde g$.
However, besides the trivially filtered case, we do not know of any even dimensional filtered manifold whose osculating algebras have pure cohomology: contact manifolds are odd dimensional and so are the five dimensional geometries discussed in Section~\ref{S:five} below.
\end{remark}

We expect that the analytic torsion proposed in this paper can be refined to a complex bilinear form on the determinant line $\sdet(H^*(M;F))$ such that the comparison with the refined complex valued Ray--Singer torsion \cite{BK05,BK07a,BK07b,BK07c,BK08,BH07,BH10,SZ08} analogous to \eqref{E:Rdef} yields a ratio that depends holomorphically on the flat connection.

\subsection{Nilpotent Lie algebras with pure cohomology}\label{SS:pure}

The analytic torsion discussed above is only defined for filtered manifolds which have osculating algebras with pure cohomology.
This assumption appears to be very restrictive.
The only examples we are aware of are: trivially filtered manifolds corresponding to abelian osculating algebras giving rise to the classical Ray--Singer torsion; contact manifolds corresponding to Heisenberg algebras giving rise to the Rumin--Seshadri torsion; and the five dimensional geometries discussed in Section~\ref{S:five} below.

The aim of this section is to present a necessary condition which the dimensions of the grading components of a finite dimensional graded nilpotent Lie algebra
\[
        \goe=\goe_{-r}\oplus\dotsb\oplus\goe_{-2}\oplus\goe_{-1}
\]
have to satisfy, if the cohomology of $\goe$ is pure.

To this end let $\phi_t$ denote the grading automorphism acting by multiplication with $t^p$ on the summand $\goe_p$ and consider the polynomial \cite{DS88}
\[
        P(t):=\str\bigl(\phi_t|H^*(\goe)\bigr).
\]
If $\goe$ has pure cohomology, then there exist numbers $p_q\in\mathbb N_0$ such that $\phi_t$ acts by the scalar $t^{p_q}$ on $H^q(\goe)$.
Hence, in this case the polynomial
\[
	P(t)=\sum_{q=0}^d(-1)^qt^{p_q}\dim H^q(\goe)
\]
has at most $d+1$ non-trivial coefficients, where $d:=\dim\goe$.

Clearly, $\str(\phi_t|H^*(\goe))=\str(\phi_t|\Lambda^*\goe^*)=\prod_p\str(\phi_t|\Lambda^*\goe^*_p)$ and thus
\[
	P(t)=(1-t)^{n_1}(1-t^2)^{n_2}\dotsm(1-t^r)^{n_r}
\]
where $n_p:=\dim\goe_{-p}$.
This is a polynomial of degree $n:=\sum_ppn_p$.
Moreover,
\begin{align*}
        P(t)
        &=(1-t)^d(1+t)^{n_2}\bigl(1+t+t^2\bigr)^{n_3}\dotsm\bigl(1+t+t^2+\dotsb+t^{r-1}\bigr)^{n_r}
        \\&=\sum_{i=0}^n(-t)^i\sum_{l=0}^{n-d}(-1)^l\binom{d}{i-l}a_l,
\end{align*}
where the (positive, integral) coefficients $a_l$ are defined by
\[
        \sum_{l=0}^{n-d}a_lt^l:=(1+t)^{n_2}\bigl(1+t+t^2\bigr)^{n_3}\dotsm\bigl(1+t+t^2+\dotsb+t^{r-1}\bigr)^{n_r}
\]
and we are using the convention $\binom dm=0$ if $m<0$ or $m>d$.
Using
\[
        (-1)^l\binom{d}{i-l}=\frac{d!}{i!(n-i)!}\prod_{j=0}^{l-1}(j-i)\prod_{j=d+l+1}^n(j-i),
\]
which is valid for all integers $0\leq i\leq n$ and $0\leq l\leq n-d$, we obtain
\[
        P(t)=\sum_{i=0}^n(-t)^i\frac{d!}{i!(n-i)!}c(i)
\]
where
\[
        c(i):=\sum_{l=0}^{n-d}a_lQ_l(i)
\]
and
\[
        Q_l(i):=\prod_{j=0}^{l-1}(j-i)\prod_{j=d+l+1}^n(j-i)=\frac{\prod_{j=0}^n(j-i)}{\prod_{j=l}^{d+l}(j-i)}
\]
are polynomials of degree $n-d$ in $i$.
Hence, the polynomial $P(t)$ has at least $d+1$ non-trivial coefficients.
Moreover, we obtain

\begin{lemma}\label{L:2}
If $\goe$ has pure cohomology, then $P(t)$ has precisely $d+1$ non-trivial coefficients and $c(i)$ has $n-d$ mutually different integral zeros located in the range $\{0,\dotsc,n\}$.
More explicitly,
\begin{equation}\label{E:cr}
	c(i)=2^{n_2}3^{n_3}\dotsm r^{n_r}\prod_{j\in\mathcal P'}(j-i),
\end{equation}
where $\mathcal P:=\{p_q:q=0,\dotsc,d\}$ and $\mathcal P':=\{0,\dotsc,n\}\setminus\mathcal P$.
\end{lemma}

\begin{example}
If $\goe=\goe_{-2}\oplus\goe_{-1}$ with $n_2=2$, then $a_l=1,2,1$ and
$$
        c(i)
        =(n-2i)^2-n
$$
with zeros $i=\frac{n\pm\sqrt n}2$.
Hence, if $\goe$ has pure cohomology, then $n$ must be a square according to Lemma~\ref{L:2}.
The case $n=4$ corresponds to $n_1=0$ and is realized by an abelian Lie algebra concentrated in degree $2$.
The case $n=9$ corresponds to $n_1=5$.
It is not hard to see that this cannot be realized by a nilpotent Lie algebra though.
We do no know if larger squares $n$ can be realized by nilpotent Lie algebras.
\end{example}

\begin{example}
If $\goe=\goe_{-2}\oplus\goe_{-1}$ with $n_2=3$, then $a_l=1,3,3,1$ and
$$
        c(i)=(n-2i)\bigl((n-2i)^2-(3n-2)\bigr)
$$
Hence, if $\goe$ has pure cohomology, then $n$ must be even and $3n-2$ must be a square according to Lemma~\ref{L:2}.
The case $n=6$ corresponds to $n_1=0$ and is realized by an abelian Lie algebra concentrated in degree $2$.
The general $n$ consistent with Lemma~\ref{L:2} is of the form $n=(m^2+2)/3$ where $m\geq4$ is an integer congruent to $2$ or $4$ mod $6$.
Which of these can be realized by nilpotent Lie algebras is unclear.
\end{example}

\begin{example}
If $\goe=\goe_{-2}\oplus\goe_{-1}$ with $n_2=4$, then $a_l=1,4,6,4,1$ and
\[
       c(i)=\bigl((n-2i)^2-(3n-4)\bigr)^2-(6n^2-18n+16).
\]
Hence, if $\goe$ has pure cohomology, then the four roots of $c(r)$,
\[
	i=\frac{n\pm\sqrt{3n-4\pm\sqrt{6n^2-18n+16}}}2,
\]
must all be integral according to Lemma~\ref{L:2}.
The case $n=8$ corresponds to $n_1=0$ and is realized by an abelian Lie algebra concentrated in degree $2$.
We do not know of any other $n$ for which this happens.
Using a computer, we ruled out all $8<n\leq10000$.
For some $n$, e.g.\ $n=17$, $n=66$, or $n=1521$, two roots are integral, but never all four, in this range.
\end{example}

\begin{example}
If $\goe=\goe_{-2}\oplus\goe_{-1}$ with $n_2=5$, then $a_l=1,5,10,10,5,1$ and
\[
	c(i)=(n-2i)\Bigl(\bigl((n-2i)^2-(5n-10)\bigr)^2-(10n^2-50n+76)\Bigr)
\]
Hence, if $\goe$ has pure cohomology, then the root $i=n/2$ must be integral, hence $n$ must be even and the other four roots,
\[
	i=\frac{n\pm\sqrt{5n-10\pm\sqrt{10n^2-50n+76}}}2,
\]
must all be integral according to Lemma~\ref{L:2}.
The case $n=10$ corresponds to $n_1=0$ and is realized by an abelian Lie algebra concentrated in degree $2$.
We do not know of any other $n$ for which this happens.
Using a computer, we ruled out all $10<n\leq10000$.
For some $n$, e.g. $n=17$, $n=36$, or $n=289$, two roots are integral, but never all five.
For $n=67$ four roots are integral, but this one is odd.
\end{example}

Using computer algebra, we also ruled out Lie algebras of the form $\goe=\goe_{-2}\oplus\goe_{-1}$ in the range $0\leq n_1\leq1000$ and $6\leq n_2\leq11$ --- non of these dimensions are consistent with the condition in Lemma~\ref{L:2}.

\begin{example}
If $\goe=\goe_{-3}\oplus\goe_{-2}\oplus\goe_{-1}$ with $n_2=n_3=1$, then $a_l=1,2,2,1$ and
\[
	c(i)=(n-2i)\Bigl(\tfrac34(n-2i)^2+\bigl(\tfrac{n^2}4-3n+2\bigr)\Bigr)
\]
Hence, if $\goe$ has pure cohomology, then the root $i=n/2$ must be integral, hence $n$ must be even and the other two roots,
\[
	i=\frac{n\pm\sqrt{-(n^2-12n+8)/3}}2,
\]
must be integral according to Lemma~\ref{L:2}.
All $n\geq12$ can be ruled out since in this case these roots are complex.
Hence, $n=10$ is the only case consistent with Lemma~\ref{L:2}.
However, one readily shows that this can not be realized by a nilpotent Lie algebra.
\end{example}

\begin{example}
If $\goe=\goe_{-3}\oplus\goe_{-2}\oplus\goe_{-1}$ with $n_2=2$ and $n_3=1$, then we have $a_l=1,3,4,3,1$ and
\[
	4c(i)=3(n-2i)^4+(n^2-23n+30)(n-2i)^2-n(n^2-14n+24).
\]
If this polynomial has four real zeros, then it must take a positive value at $i=n/2$, hence we must have $n^2-14n+24<0$.
This rules out any $n\geq12$.
The remaining finitely many $n$ can easily be excluded.
Hence, no nilpotent Lie algebra of this form is consistent with the condition in Lemma~\ref{L:2}.
\end{example}

\begin{example}
If $\goe=\goe_{-3}\oplus\goe_{-2}\oplus\goe_{-1}$ with $n_2=1$ and $n_3=2$, then we have $a_l=1,3,5,5,3,1$ and
\begin{multline*}
	16c(i)=9(n-2i)^5+(6n^2-132n+252)(n-2i)^3\\+(n^4-28n^3+308n^2-800n+384)(n-2i).
\end{multline*}
If $\goe$ has pure cohomology, then the zero $i=n/2$ must be integral, hence $n$ must be even.
Moreover, in order to obtain five real zeros, the polynomial $c(i)$ must have a local maximum at $i=n/2$.
Hence we must have $6n^2-132n+252<0$, which rules out any $n\geq20$.
Checking the remaining $n$ by computer, we see that the only case consistent with Lemma~\ref{L:2} is $n=10$ corresponding to $n_1=2$.
The latter can indeed be realized by a nilpotent Lie algebra, see Section~\ref{S:five} below.
\end{example}

Using computer algebra, we also searched $\goe=\goe_{-5}\oplus\goe_{-4}\oplus\goe_{-3}\oplus\goe_{-2}\oplus\goe_{-1}$ in the range $0\leq n_1\leq100$, $0\leq n_2,n_3,n_4,n_5\leq5$.
The only cases consistent with Lemma~\ref{L:2} are the ones mentioned above, perhaps with a different grading (abelian concentrated in any degree, or contact in even degrees).
We also searched $\goe=\goe_{-3}\oplus\goe_{-2}\oplus\goe_{-1}$ in the range $0\leq n_1\leq200$, $0\leq n_2\leq50$, $0\leq n_3\leq20$ and did not find anything new which is consistent with Lemma~\ref{L:2}.
The algorithm we used expands the polynomial $P(t)=(1-t)^{n_1}(1-t^2)^{n_2}\cdots(1-t^r)^{n_r}$ and counts the number of non-zero coefficients.

\section{Generic rank two distributions in dimension five}\label{S:five}

Recall that a generic rank two distribution on a 5-manifold $M$ is a smooth rank two subbundle $\mathcal D$ in the tangent bundle $TM$ such that Lie brackets of sections of $\mathcal D$ span a rank three subbundle $[\mathcal D,\mathcal D]$ in $TM$ and triple brackets span all of the tangent bundle.
Hence, putting $T^{-1}M:=\mathcal D$, $T^{-2}M:=[\mathcal D,\mathcal D]$ and $T^{-3}M:=TM$, a generic rank two distribution turns $M$ into a filtered manifold, 
\[
	TM=T^{-3}M\supseteq T^{-2}M\supseteq T^{-1}M.
\]
We will denote this filtration by $\mathcal F_{\mathcal D}$.

The osculating algebras $\mathfrak t_xM=\mathfrak t^{-3}_xM\oplus\mathfrak t^{-2}_xM\oplus\mathfrak t_x^{-1}M$
of such a filtration are all isomorphic to the five dimensional graded Lie algebra 
\begin{equation}\label{E:goefive}
	\goe=\goe_{-3}\oplus\goe_{-2}\oplus\goe_{-1}
\end{equation}
with graded basis $X_1,X_2\in\goe_{-1}$, $X_3\in\goe_{-2}$, $X_4,X_5\in\goe_{-3}$ and non-trivial brackets
\begin{equation}\label{E:brackets}
	[X_1,X_2]=X_3,\qquad[X_1,X_3]=X_4,\qquad[X_2,X_3]=X_5.
\end{equation}
Note that these geometries have even homogeneous dimension
\[
	n=10.
\]

\begin{lemma}\label{L:Autgoe}
Suppose $Y_1,Y_2\in\goe$ project to a basis of $\goe/[\goe,\goe]$.
Then there exists a unique (in general ungraded) Lie algebra automorphism $\varphi\colon\goe\to\goe$ such that $\varphi(X_1)=Y_1$ and $\varphi(X_2)=Y_2$.
\end{lemma}

\begin{proof}
Define a linear map $\varphi\colon\goe\to\goe$ by $\varphi(X_1):=Y_1$, $\varphi(X_2):=Y_2$, $\varphi(X_3):=[Y_1,Y_2]$, $\varphi(X_4):=[Y_1,[Y_1,Y_2]]$, $\varphi(X_5)=[Y_2,[Y_1,Y_2]]$.
Using the relations in~\eqref{E:brackets} one readily checks that $\varphi$ is a homomorphism of Lie algebras.
With respect to the decomposition $\goe=\goe_{-3}\oplus\goe_{-2}\oplus\goe_{-1}$ the linear map $\varphi$ has upper triangular block form with invertible diagonal blocks.
Hence, $\varphi$ is invertible.
\end{proof}

In particular, this lemma shows that restriction provides an isomorphism 
\begin{equation}\label{E:Autfive}
	\Aut(\goe)=\GL(\goe_{-1})\cong\GL_2(\mathbb R)
\end{equation}
where the left hand side denotes the group of graded Lie algebra automorphisms.

For the Betti numbers $b^q:=\dim H^q(\goe)$ we find $b^0=b^5=1$, $b^1=b^4=2$, $b^2=b^3=3$.
Moreover, $\goe$ has pure cohomology.
Indeed, the grading automorphism acts by multiplication $t^{p_q}$ on $H^q(\goe)$ where $p_0=0$, $p_1=1$, $p_2=4$, $p_3=6$, $p_4=9$, $p_5=10$.
Hence, the Rumin complex is a Rockland complex of the form
\begin{multline*}
	\Gamma^\infty\bigl(\mathcal H^0(\mathfrak tM)\bigr)\xrightarrow{D_0}
	\Gamma^\infty\bigl(\mathcal H^1(\mathfrak tM)\bigr)\xrightarrow{D_1}
	\Gamma^\infty\bigl(\mathcal H^2(\mathfrak tM)\bigr)\\\xrightarrow{D_2}
	\Gamma^\infty\bigl(\mathcal H^3(\mathfrak tM)\bigr)\xrightarrow{D_3}
	\Gamma^\infty\bigl(\mathcal H^4(\mathfrak tM)\bigr)\xrightarrow{D_4}
	\Gamma^\infty\bigl(\mathcal H^5(\mathfrak tM)\bigr)
\end{multline*}
with ranks $\rk\mathcal H^0(\mathfrak tM))=\rk\mathcal H^5(\mathfrak tM))=1$, $\rk\mathcal H^1(\mathfrak tM))=\rk\mathcal H^4(\mathfrak tM))=2$, and $\rk\mathcal H^2(\mathfrak tM))=\rk\mathcal H^3(\mathfrak tM))=3$.
Moreover, the differential operators $D_0$ and $D_4$ have Heisenberg order $1$; the operators $D_1$ and $D_3$ have Heisenberg order $3$; and $D_2$ has Heisenberg order $2$, see \cite[Section~5]{BEN11} or \cite[Example~4.24]{DH17}.

To obtain an analytic torsion, it suffices to choose a sub-Riemannian metric $\tilde g$ on the rank two bundle $\mathcal D=T^{-1}M=\mathfrak t^{-1}M$.
We can use the fiberwise Lie algebra structure to extend this to a graded Euclidean inner product on $\mathfrak tM$, which will be denoted by $\tilde g$ too.
Indeed, in view of \eqref{E:brackets}, the Levi bracket induces canonical isomorphisms of vector bundles
\[
	\Lambda^2\mathfrak t^{-1}M=\mathfrak t^{-2}M\qquad\text{and}\qquad\mathfrak t^{-1}M\otimes\mathfrak t^{-2}M=\mathfrak t^{-3}M.
\]
More explicitly, using a graded basis as in \eqref{E:brackets}, we extend $\tilde g$ such that
\begin{align}
	\tilde g(X_3,X_3)&=4\bigl(\tilde g(X_1,X_1)\tilde g(X_2,X_2)-\tilde g(X_1,X_2)^2\bigr)\label{E:gext1}\\
	\tilde g([X_i,X_3],[X_j,X_3])&=3\tilde g(X_3,X_3)\tilde g(X_i,X_j),\qquad i,j=1,2.\label{E:gext2}
\end{align}
This choice of peculiar constants $4$ and $3$ is motivated by parabolic geometry,  see \cite[Section~3.3.1]{CS09} and \cite[Equations~(3.21) and (3.39)]{S08}.
Hence, a sub-Riemannian metric $\tilde g$ on the rank two bundle $\mathcal D$ and a Hermitian metric $h$ on a flat bundle $F$ over $M$ give rise to an analytic torsion 
\[
	\|-\|^{\sdet(H^*(M;F))}_{\mathcal D,\tilde g,h}:=\|-\|^{\sdet(H^*(M;F))}_{\mathcal F_{\mathcal D},\tilde g,h}
\]
i.e.\ a norm on the graded determinant line, cf.~\eqref{E:deftorM},
\[
	\sdet(H^*(M;F)):=\bigotimes_q\bigl(\det H^q(M;F)\bigr)^{(-1)^q}.
\]

For the purpose of this paper, any other pair of positive constants in \eqref{E:gext1} and \eqref{E:gext2} would work equally well. 
We will not discuss here to what extent the torsion depends on the choice of these constants.
The following will be crucial though:
If $\tilde g_u$ is a family of sub-Riemannian metrics on $\mathcal D$, then their extensions differ by a graded automorphism of the bundle of osculating algebras $\mathfrak tM$, that is, $\tilde g_v^{-1}\tilde g_u\in\Gamma^\infty(\Aut(\mathfrak tM))$, see \eqref{E:Autfive}.
Moreover, $\tr_{\mathfrak tM}(\dot g_u)=5\tr_{\mathcal D}(\dot g_u)$.
Hence, Theorem~\eqref{T:varg} is applicable and specializes Theorem~\ref{T:var235}.



\subsection{Proof of Theorem~\ref{T:var.all235}}\label{SS:proof}

We begin by noting that a sub-Riemannian metric $\tilde g$ on a generic rank two distribution $\mathcal D$ induces a splitting of the filtration $S_{\mathcal D,\tilde g}\colon\mathfrak tM\to TM$.
Indeed,	the sub-Riemannian metric induces a nowhere vanishing section of the line bundle $|\Lambda^2\mathcal D^*|$ that provides a scale which is closely related to the generalized contact forms considered in \cite[Section~3.1]{S08}.
It induces \cite[Corollary~5.1.6]{CS09} an (exact) Weyl structure \cite[Definition~5.1.7]{CS09} for the canonical (regular and normal) Cartan connection \cite{C10}.
In particular, it provides a splitting of the filtration such that $S_{\mathcal D,\tilde g}(x)$ only depends on the germ of $(M,\mathcal D,\tilde g)$ at $x$.
Using these splittings, the coefficients at $x$ of the differential operators $L\colon\Gamma^\infty(\mathcal H(\mathfrak tM))\to\Omega(M)$ and $\pi\Pi\colon\Omega(M)\to\Gamma^\infty(\mathcal H(\mathfrak tM))$ only depend on the germ of $(M,\mathcal D,\tilde g)$ at $x$.

We denote the filtered manifold corresponding to the generic rank two distribution $\mathcal D_u$ by $M_u$.
Let $\tilde\psi_u\colon\mathfrak tM\to\mathfrak tM_u$ denote the unique isomorphism of graded Lie algebra bundles which restricts to $\Theta_u|_{\mathcal D}\colon\mathcal D\to\mathcal D_u$.
Clearly, $(\tilde\psi_u)_*\tilde g=\tilde g_u$, cf.~\eqref{E:gpsiuv}, where $\tilde g$ and $\tilde g_u$ denote the graded Euclidean metrics on $\mathfrak tM$ and $\mathfrak tM_u$ which are induced from the sub-Riemannian metrics on $\mathcal D$ and $\mathcal D_u$ as indicated in \eqref{E:gext1} and \eqref{E:gext2}.
We use the splittings discussed above to define $\psi_u\in\Gamma^\infty(\Aut(TM))$ by $\psi_uS_{\mathcal D,\tilde g}=S_{\mathcal D_u,\tilde g_u}\tilde\psi_u$, cf.~\eqref{E:Spsiuv}.
Put $\dot\psi=\frac\partial{\partial u}|_{u=0}\psi_u$ and consider the differential operator, cf.~\eqref{E:dotAuF},
\begin{equation}\label{E:dotAF235}
	\dot A=-\pi\Pi\dot\psi L.
\end{equation}

We will now show that $\tr(\dot\psi)(x)$ and the coefficients of $\dot A$ at $x$ only depend on the germ of $(M,\mathcal D,\tilde g,\dot\Theta|_{\mathcal D})$ at $x$.
To see this let $X,Y$ be a local frame of $\mathcal D$ defined on an open subset $U\subseteq M$.
Then $X,Y,[X,Y],[X,[X,Y]],[Y,[X,Y]]$ is a frame of $TM|_U$ and we may define a smooth family of filtration preserving automorphisms $\psi_u'\colon TM|_U\to TM_u|_U$ by:
\begin{align*}	
	\psi'_u(X)&:=\Theta_u(X)\\
	\psi'_u(Y)&:=\Theta_u(Y)\\
	\psi'_u([X,Y])&:=[\Theta_u(X),\Theta_u(Y)]\\
	\psi'_u([X,[X,Y]])&:=[\Theta_u(X),[\Theta_u(X),\Theta_u(Y)]]\\
	\psi'_u([Y,[X,Y]])&:=[\Theta_u(Y),[\Theta_u(X),\Theta_u(Y)]]
\end{align*}
Clearly $\gr(\psi_u')=\tilde\psi_u$ and $\psi'_0=\id$.
Differentiating, we see that $\dot\psi'(x)=\frac\partial{\partial u}|_{u=0}\psi_u'(x)$ only depends on the germ of $(M,\mathcal D,\dot\Theta|_{\mathcal D},X,Y)$ at $x$.
According to Remark~\ref{R:dotpsi} we have $\tr(\dot\psi)|_U=\tr(\dot\psi')$ and $\dot A|_U=-\pi\Pi\dot\psi'L|_U$.
We conclude that $\tr(\dot\psi)(x)$ and the coefficients of $\dot A$ at $x$ only depend on the germ of $(M,\mathcal D,\tilde g,\dot\Theta|_{\mathcal D})$ at $x$.

From the preceding paragraph we see that the coefficients at $x$ of the differential operator $\dot A^F+(\dot A^F)^*-\tr(\dot\psi)$ only depend on the germ of $(M,\mathcal D,\tilde g,F,h,\dot\Theta|_{\mathcal D})$ at $x$.
Hence, in the asymptotic expansion
\begin{equation}\label{E:tpj}
	\Bigl(\bigl(\dot A^F+(\dot A^F)^*-\tr(\dot\psi)\bigr)k^F_t\Bigr)(x,x)\sim\sum_{j=-r}^\infty t^{(j-10)/2\kappa}\tilde p_j^F(x).
\end{equation}
the terms $\tilde p_j^F(x)$ can be computed from the germ of $(M,\mathcal D,\tilde g,F,h,\dot\Theta|_{\mathcal D})$ at $x$.
Hence, in view of Theorem~\ref{T:def.filt}, the density 
\begin{equation}\label{E:alpha}
	\alpha=\str(\tilde p^F_{10})
\end{equation} 
has the desired property.

\subsection{Computation for nilmanifolds}\label{SS:nilmf}

Let $G$ denote the simply connected nilpotent Lie group with Lie algebra \eqref{E:goefive}.
Since the structure constants are rational, this group admits lattices \cite[Theorem~2.12]{R72}.
Suppose $\Gamma\subseteq G$ is a lattice and let $\rho\colon\Gamma\to U(k)$ be a finite dimensional unitary representation of $\Gamma$.
Then $F_\rho:=G\times_\Gamma\C^k$ is a flat vector bundle over the nilmanifold $G/\Gamma$ which comes with a parallel Hermitian metric $h_\rho$.
Let $\mathcal D$ be a right $G$-invariant generic rank two distribution on $G$ and let $\mathcal D_\Gamma$ denote the induced generic rank two distribution on $G/\Gamma$.
Note that a 2-dimensional subspace in $\goe$ gives rise to a right invariant generic rank two distribution on $G$ iff it intersects $[\goe,\goe]=\goe_{-3}\oplus\goe_{-2}$ trivially, cf.~\eqref{E:brackets}.
We consider a right invariant sub-Riemannian metric $\tilde g$ on $\mathcal D$ and let $\tilde g_\Gamma$ denote the induced sub-Riemannian metric on $\mathcal D_\Gamma$.

%

\begin{lemma}\label{L:Rrhotg}
The torsion $\|-\|^{\sdet H^*(G/\Gamma;F_\rho)}_{\mathcal D_\Gamma,\tilde g_\Gamma,h_\rho}$ does not depend on the right invariant generic rank two distribution $\mathcal D$ on $G$ nor does it depend on the right invariant sub-Riemannian metric $\tilde g$ on $\mathcal D$.
\end{lemma}

\begin{proof}
We start by showing independence of $\tilde g$.
By convexity, any two right invariant sub-Riemannian metrics on $\mathcal D$ can be connected by a straight path of right invariant sub-Riemannian metrics on $\mathcal D$.
In view of Theorem~\ref{T:var235}, it therefore suffices to show that the constant term in the asymptotic expansion of the heat kernel of the Rumin--Seshadri operator $\Delta_{\mathcal D_\Gamma,\tilde g_\Gamma,h_\rho}^{F_\rho}$ vanishes.
Since this asymptotic expansion can be computed locally, we may work on the universal covering $p\colon G\to G/\Gamma$.
Over $G$, the flat bundle $F_\rho$ may be (canonically) trivialized and the Hermitian metric $h_\rho$ becomes constant in this trivialization.
In view of Lemma~\ref{L:Autgoe} there exists a group isomorphism $\phi\colon G\to G$ such that $\phi^*\mathcal D=\mathcal D_0$, the right invariant standard distribution obtained by translating $\goe_{-1}$.
Moreover, $\tilde g_0:=\phi^*\tilde g$ is a right invariant sub-Riemannian metric on $\mathcal D_0$.
Hence, by naturality, 
\begin{equation}\label{E:trivG}
	\phi^*p^*\Delta_{\mathcal D_\Gamma,\tilde g_\Gamma,h_\rho}^{F_\rho}=\Delta_{\mathcal D_0,\tilde g_0,h_0}^{\mathbb C^{\rk(\rho)}}
\end{equation}
where $h_0$ denotes the standard (constant) Hermitian metric on the trivial flat bundle $\phi^*p^*F_\rho=G\times\mathbb C^{\rk(\rho)}$.
Note that $\mathcal D_0$ and $\tilde g_0$ are also homogeneous with respect to the grading automorphism of $G$.
Hence, $\Delta_{\mathcal D_0,\tilde g_0,h_0}^{\mathbb C^{\rk(\rho)}}$ is right invariant and homogeneous of degree $2\kappa$.
The same is true for the heat operator $\partial_t+\Delta_{\mathcal D_0,\tilde g_0,h_0}^{\mathbb C^{\rk(\rho)}}$ on $\mathbb R\times G$, after assigning degree $2\kappa$ to the time direction $\mathbb R$.
Consequently, this heat operator admits a parametrix which is homogeneous of degree $-2\kappa$.
Combining this with \eqref{E:trivG}, we see that all but the leading term in the asymptotic expansion of the heat kernel of $\Delta_{\mathcal D_\Gamma,\tilde g_\Gamma,h_\rho}^{F_\rho}$ vanish, cf.\ the proof of Lemma~\ref{L:asymA}. 

To show independence of $\mathcal D$, we proceed analogously using Theorem~\ref{T:var.all235}.
Since the space of 2-dimensional subspaces in $\goe$ intersecting $\goe_{-3}\oplus\goe_{-2}$ trivially is connected, any two right invariant generic rank two distributions on $G$ can be connected by a smooth family of rank two distributions of the same type.
According to Theorem~\ref{T:var.all235} and the formula in~\eqref{E:alpha} it thus suffices to show that for every differential operator $B$ of Heisenberg order at most $r$ on $G/\Gamma$ we have an asymptotic expansion of the form, as $t\to0$,
\[
	\bigl(Be^{-t\Delta^{F_\rho}_{\mathcal D_\Gamma,\tilde g_\Gamma,h_\rho}}\bigr)(x,x)\sim\sum_{j=-r}^0t^{(j-10)/2\kappa}\tilde p_j(x)
\]
and, in particular, the constant term $\tilde p_{10}$ vanishes.
This in turn follows from \eqref{E:trivG}.
Indeed, since $\Delta_{\mathcal D_0,\tilde g_0,h_0}^{\mathbb C^{\rk(\rho)}}$ is homogeneous and right invariant, $(\phi^*p^*B)e^{-\Delta_{\mathcal D_0,\tilde g_0,h_0}^{\mathbb C^{\rk(\rho)}}}$ has an asymptotic expansion of the same form.
\end{proof}

Since $G/\Gamma$ is odd dimensional, the Ray--Singer torsion $\|-\|^{\sdet(H^*(G/\Gamma;F_\rho))}_{\RS,g,h_\rho}$ is independent of the Riemannian metric $g$ and the fiberwise Hermitian inner product $h$ on $F_\rho$, see \cite{RS71} and \cite[Theorem~0.1]{BZ92}. 
Combining this with Lemma~\ref{L:Rrhotg}, we see that the relative torsion in \eqref{E:Rdef},
\begin{equation}\label{E:Rrho}
	R_\Gamma(\rho)
	:=R_{\tilde g_\Gamma;g}^{F_\rho,h_\rho}(G/\Gamma,\mathcal F_{\mathcal D_\Gamma})
	=\frac{\|-\|^{\sdet(H^*(G/\Gamma;F_\rho))}_{\mathcal D_\Gamma,\tilde g_\Gamma,h_\rho}}{\|-\|^{\sdet(H^*(G/\Gamma;F_\rho))}_\RS}
\end{equation}
only depends on the lattice $\Gamma$ in $G$ and the unitary representation $\rho\colon\Gamma\to U(k)$.

\begin{lemma}\label{L:Rrho}
Let $\rho$, $\rho_1$ and $\rho_2$ be a finite dimensional unitary representations of a lattice $\Gamma$ in $G$.
Moreover, let $\tilde\rho$ be a finite dimensional unitary representations of a sublattice $\tilde\Gamma\subseteq\Gamma$. 
Then the following hold true:
\begin{enumerate}[(a)]
	\item $R_\Gamma(\rho_1\oplus\rho_2)=R_\Gamma(\rho_1)\cdot R_\Gamma(\rho_2)$.
	\item $R_\Gamma(\rho)=R_{\tilde\Gamma}(\tilde\rho)$ if $\rho=\Ind_{\tilde\Gamma}^\Gamma(\tilde\rho)$ is the induced representation of $\Gamma$.
	\item $R_{\tilde\Gamma}(\rho|_{\tilde\Gamma})=R_\Gamma(\rho\otimes\mathbb C[\Gamma/\tilde\Gamma])$.
	\item $R_\Gamma(\rho^*)=R_\Gamma(\bar\rho)=R_\Gamma(\rho)$.
	\item $R_{\phi^{-1}(\Gamma)}(\rho\phi)=R_\Gamma(\rho)$ for every Lie group isomorphism $\phi\colon G\to G$.
\end{enumerate}
\end{lemma}

\begin{proof}
With respect to the canonical identification $F_{\rho_1\oplus\rho_2}=F_{\rho_1}\oplus F_{\rho_2}$ we have $h_{\rho_1\oplus\rho_2}=h_{\rho_1}\oplus h_{\rho_2}$.
Hence (a) follows from Proposition~\ref{P:R}(a).
	
The canonical projection $\pi\colon G/\tilde\Gamma\to G/\Gamma$ is a finite covering of filtered manifolds, where we use the generic rank two distributions $\mathcal D_{\tilde\Gamma}$ and $\mathcal D_\Gamma$ on $G/\tilde\Gamma$ and $G/\Gamma$, respectively.
If $\rho=\Ind_{\tilde\Gamma}^\Gamma(\tilde\rho)$, then with respect to the canonical identification $\pi_*F_{\tilde\rho}=F_\rho$ we have $\pi_*h_{\tilde\rho}=h_\rho$.
Moreover, $\pi^*\tilde g_\Gamma=\tilde g_{\tilde\Gamma}$.
Hence (b) follows from Proposition~\ref{P:R}(b).
	
(c) follows from (b) for we have $\Ind_{\tilde\Gamma}^\Gamma(\rho|_{\tilde\Gamma})=\rho\otimes\C[\Gamma/\tilde\Gamma]$.

(d) follows from Proposition~\ref{P:R}(d).

To see (e), note that $\mathcal D':=\phi^{-1}(\mathcal D)$ is a right invariant generic rank two distribution on $G$ and $\tilde g':=\phi^*\tilde g$ is a right invariant sub-Riemannian metric on $\phi^{-1}(\mathcal D)$.
Hence, $\phi$ induces a diffeomorphism $\bar\phi\colon G/\Gamma'\to G/\Gamma$ where we consider the lattice $\Gamma':=\phi^{-1}(\Gamma)$.
Clearly, $\bar\phi(\mathcal D'_{\Gamma'})=\mathcal D_\Gamma$ and $\bar\phi^*\tilde g_\Gamma=\tilde g'_{\Gamma'}$.
Moreover, via the canonical identification $\bar\phi^*F_\rho=F_{\rho'}$ we have $\bar\phi^*h_\rho=h_{\rho'}$, where we abbreviate $\rho':=\rho\phi$.
Hence, by naturality
\[
	\|-\|^{\sdet(H^*(G/\Gamma',F_{\rho'}))}_{\mathcal D'_{\Gamma'},\tilde g'_{\Gamma'},h_{\rho'}}
	=\bar\phi^*\|-\|^{\sdet(H^*(G/\Gamma,F_\rho))}_{\mathcal D_\Gamma,\tilde g_\Gamma,h_\rho}
\]
via the identification $\sdet(H^*(G/\Gamma';F_{\rho'}))\cong\sdet(H^*(G/\Gamma;F_\rho))$ induced by $\bar\phi$.
Combining this with a corresponding identity for the Ray--Singer torsion, we obtain $R_{\Gamma'}(\rho')=R_\Gamma(\rho)$, see \eqref{E:Rrho} and Lemma~\ref{L:Rrhotg}.
\end{proof}

The exponential map provides a diffeomorphism $\exp\colon\goe\to G$.
Using the Baker--Campbell--Hausdorff formula we find
\[
        \exp\left(\sum_{i=1}^5x_iX_i\right)\exp\left(\sum_{i=1}^5y_iX_i\right)
        =\exp\left(\sum_{i=1}^5z_iX_i\right)
\]
where
\begin{equation}\label{E:zxy}
        z=x\cdot y
        :=\left(\begin{array}{l}
                x_1+y_1\\
                x_2+y_2\\
		x_3+y_3+\frac{x_1y_2-x_2y_1}2\\x_4+y_4+\frac{x_1y_3-x_3y_1}2+\frac{(x_1-y_1)(x_1y_2-x_2y_1)}{12}\\
		x_5+y_5+\frac{x_2y_3-x_3y_2}2+\frac{(x_2-y_2)(x_1y_2-x_2y_1)}{12}
        \end{array}\right)
\end{equation}
We let $\Gamma_0$ denote the subgroup generated by $\gamma_1:=\exp(X_1)$ and $\gamma_2:=\exp(X_2)$.

\begin{lemma}
$\Gamma_0$ is a lattice in $G$ and
\begin{equation}\label{E:Gamma0}
	\log\Gamma_0=\left\{\sum_{i=1}^5x_iX_i\middle|
	\begin{array}{rr}
		x_1,x_2\in\mathbb Z\\
		x_3-\frac{x_1x_2}2\in\Z\\
		x_4-\tfrac{x_1^2x_2}{12}-\frac{x_1+1}2(x_3-\frac{x_1x_2}2)\in\Z\\
		x_5+\tfrac{x_1x_2^2}{12}+\frac{x_2+1}2(x_3-\frac{x_1x_2}2)\in\Z
	\end{array}\right\}.
\end{equation}
\end{lemma}

\begin{proof}
With $z$ as in \eqref{E:zxy} we have:	
\begin{align*}
	z_3-\tfrac{z_1z_2}2&=x_3-\tfrac{x_1x_2}2+y_3-\tfrac{y_1y_2}2-x_2y_1\\
	z_4-\tfrac{z_1^2z_2}{12}-\tfrac{z_1+1}2(z_3-\tfrac{z_1z_2}2)
	&=x_4-\tfrac{x_1^2x_2}{12}-\tfrac{x_1+1}2(x_3-\tfrac{x_1x_2}2)\\
	&\qquad+y_4-\tfrac{y_1^2y_2}{12}-\tfrac{y_1+1}2(y_3-\tfrac{y_1y_2}2)\\
	&\qquad-y_1(x_3-\tfrac{x_1x_2}2)+\tfrac{y_1(y_1+1)x_2}2\\
	z_5+\tfrac{z_1z_2^2}{12}+\tfrac{z_2+1}2(z_3-\tfrac{z_1z_2}2)
	&=x_5+\tfrac{x_1x_2^2}{12}+\tfrac{x_2+1}2(x_3-\tfrac{x_1x_2}2)\\
	&\qquad+y_5+\tfrac{y_1y_2^2}{12}+\tfrac{y_2+1}2(y_3-\tfrac{y_1y_2}2)\\
	&\qquad+x_2(y_3-\tfrac{y_1y_2}2)-\tfrac{x_2(x_2+1)y_1}2
\end{align*}
Using these relations, one readily checks that the right hand side in \eqref{E:Gamma0} defines a lattice in $G$ that contains $\gamma_1$ and $\gamma_2$.
Using the computations
\begin{align*}
	\log(\gamma_1^k\gamma_2^l)&=\begin{pmatrix}k\\l\\kl/2\\k^2l/12\\-kl^2/12\end{pmatrix},
	&\log[\gamma_1,\gamma_2]&=\begin{pmatrix}0\\0\\1\\1/2\\1/2\end{pmatrix},
\\	\log[\gamma_1,[\gamma_1,\gamma_2]]&=\begin{pmatrix}0\\0\\0\\1\\0\end{pmatrix},
	&\log[\gamma_2,[\gamma_1,\gamma_2]]&=\begin{pmatrix}0\\0\\0\\0\\1\end{pmatrix},
\end{align*}
it is easy to see that this lattice is generated by $\gamma_1$ and $\gamma_2$.
Here the formula
\[
	\log[\exp x,\exp y]=\begin{pmatrix}0\\0\\x_1y_2-x_2y_1\\x_1y_3-x_3y_1+\frac{(x_1+y_1)(x_1y_2-x_2y_1)}2\\x_2y_3-x_3y_2+\frac{(x_2+y_2)(x_1y_2-x_2y_1)}2\end{pmatrix}
\]
for commutators is helpful.
\end{proof}

\begin{lemma}\label{L:GG0}
If $\Gamma$ is a lattice in $G$, then there exists a, not necessarily graded Lie group automorphism $\phi\colon G\to G$ such that $\phi(\Gamma)\subseteq\Gamma_0$.
\end{lemma}

\begin{proof}
Every lattice $\Gamma$ in $G$ can be generated by five elements $\omega_1,\dotsc,\omega_5$.
These generators may be chosen such that $\omega_1,\omega_2\in\Gamma$, $\omega_3\in\Gamma\cap[G,G]$ and $\omega_4,\omega_5\in\Gamma\cap[G,[G,G]]$, see \cite[Theorem~2.21]{R72}.
According to Lemma~\ref{L:Autgoe} that there exists a unique Lie algebra automorphism $\varphi\colon\goe\to\goe$ such that $\varphi(X_1)=\log\omega_1$ and $\varphi(X_2)=\log\omega_2$.
Then $\exp\circ\varphi\circ\log\colon G\to G$ is a Lie group automorphism mapping $\gamma_1$ to $\omega_1$ and $\gamma_2$ to $\omega_2$.
W.l.o.g.\ we may thus assume $\omega_1=\gamma_1$ and $\omega_2=\gamma_2$.
Then $\Gamma_0$ is a subgroup of finite index in $\Gamma$. 
Hence, there exists a positive integer $k$ such that $\omega_3^k,\omega_4^k,\omega_5^k$ are all contained in $\Gamma_0$.
Writing
\[
	\log\omega_3=\begin{pmatrix}0\\0\\a\\b\\c\end{pmatrix},\qquad
	\log\omega_4=\begin{pmatrix}0\\0\\0\\d\\e\end{pmatrix},\qquad
	\log\omega_5=\begin{pmatrix}0\\0\\0\\f\\g\end{pmatrix},
\]
we see from \eqref{E:Gamma0} that the numbers $ka,kb-\frac{ka}2,kc+\frac{ka}2,kd,ke,kf,kg$ must all be integral.
For the action by the grading automorphism we find  
\[
	\log\phi_r(\omega_3)=\begin{pmatrix}0\\0\\r^2a\\r^3b\\r^3c\end{pmatrix},\quad
	\log\phi_r(\omega_4)=\begin{pmatrix}0\\0\\0\\r^3d\\r^3e\end{pmatrix},\quad
	\log\phi_r(\omega_5)=\begin{pmatrix}0\\0\\0\\r^3f\\r^3g\end{pmatrix}.
\]
Taking $r=2k$, the numbers $r,r^2a,r^3b-\frac{r^2a}2,r^3c+\frac{r^2a}2,r^3d,r^3e,r^3f,r^3g$ are all integral.
Hence, using \eqref{E:Gamma0} again, we obtain $\phi_r(\omega_i)\in\Gamma_0$ for $i=1,\dotsc,5$.
As $\Gamma$ is generated by $\omega_1,\dotsc,\omega_5$, this implies $\phi_r(\Gamma)\subseteq\Gamma_0$, whence the lemma.
\end{proof}

\begin{lemma}\label{L:tG0}
If $\rho_0$ is a finite dimensional unitary representation of\/ $\Gamma_0$, then 
\[
	R_{\Gamma_0}(\rho_0\otimes\chi_0)=R_{\Gamma_0}(\rho_0)
\]
for all unitary characters $\chi_0\colon\Gamma_0\to U(1)$.
\end{lemma}

\begin{proof}
W.l.o.g.\ we may assume $\rho_0$ irreducible, see Lemma~\ref{L:Rrho}(a).
According to \cite[Proposition~1]{H77} there exists a normal subgroup $\Gamma_1\subseteq\Gamma_0$ of finite index, a unitary representation $\tau_0$ of $\Gamma_0/\Gamma_1$, and a unitary character $\sigma_0\colon\Gamma_0\to U(1)$ such that 
\begin{equation}\label{E:rho0}
	\rho_0=\tau_0\otimes\sigma_0.
\end{equation}
We consider the group of graded automorphisms preserving $\Gamma_0$,
\[
	\Aut(G,\Gamma_0)
	\cong\GL(2,\mathbb Z).
\]
Since $\Gamma_0$ only has finitely many subgroups of given finite index, see \cite[Section~II]{H77} or \cite[Lemma~3]{B73}, the subgroup
\[
	A:=\bigl\{\phi\in\Aut(G,\Gamma_0):\phi(\Gamma_1)=\Gamma_1,\phi|_{\Gamma_0/\Gamma_1}=1\bigr\}
\]
has finite index in $\Aut(G,\Gamma_0)$.
For $\phi\in A$ we have $\tau_0\phi=\tau_0$ and Lemma~\ref{L:Rrho}(e) gives
\begin{equation}\label{E:jkl}
	R_{\Gamma_0}(\tau_0\otimes\sigma\phi)=R_{\Gamma_0}(\tau_0\otimes\sigma)
\end{equation}
for all unitary characters $\sigma\colon\Gamma_0\to U(1)$.
Recall that the action of $\Aut(G,\Gamma_0)\cong\GL(2,\mathbb Z)$ on the dual group $\hom(\Gamma_0,U(1))\cong U(1)\times U(1)$ has dense orbits \cite{D76}.
Since $A$ is a finite index subgroup, there exist finitely many $A$-orbits whose closures cover all of $U(1)\times U(1)$.
By \eqref{E:jkl} and Proposition~\ref{P:cont} the function $\sigma\mapsto R_{\Gamma_0}(\tau_0\otimes\sigma)$ is constant on each of these orbit closures.
As $U(1)\times U(1)$ is connected, $R_{\Gamma_0}(\tau_0\otimes\sigma)$ has to be independent of $\sigma$.
We conclude 
\[
	R_{\Gamma_0}(\tau_0\otimes\sigma_0)=R_{\Gamma_0}(\tau_0\otimes\sigma_0\chi_0).
\]
Combining this with \eqref{E:rho0}, we obtain the lemma.
\end{proof}

%

Combining these observations we will now show

\begin{theorem}\label{T:main}
Suppose\/ $\Gamma$ is a lattice in $G$ and let $\rho$ denote a finite dimensional unitary representation of\/ $\Gamma$.
Then
\[
	R_\Gamma(\rho\otimes\chi)=R_\Gamma(\rho)
\]
for all unitary characters $\chi\colon\Gamma\to U(1)$ that vanish on $\Tor H_1(\Gamma)$.
\end{theorem}

\begin{proof}
In view of Lemma~\ref{L:GG0} and Lemma~\ref{L:Rrho}(e) we may, w.l.o.g.\ assume $\Gamma\subseteq\Gamma_0$.
Since $\chi$ vanishes on the torsion part of $H_1(\Gamma)$ it can be extended to a unitary character $\chi_0\colon\Gamma_0\to U(1)$, see \cite{H77} for instance.
Let $\rho_0=\Ind_\Gamma^{\Gamma_0}(\rho)$ denote the induced representation of $\Gamma_0$.
Using Lemma~\ref{L:Rrho}(b) we find
\[
	R_\Gamma(\rho)=R_{\Gamma_0}(\rho_0)\qquad\text{and}\qquad R_\Gamma(\rho\otimes\chi)=R_{\Gamma_0}(\rho_0\otimes\chi_0).
\]
The theorem thus follows form Lemma~\ref{L:tG0}.
\end{proof}

Note that Theorem~\ref{T:intro} in the introduction follows immediately.
Indeed, a lattice $\Gamma$ in $G$ is generated by two elements of and only if $H_1(\Gamma)=0$ is torsionfree.
Moreover, every flat line bundle $F$ with parallel Hermitian metric $h$ on $G/\Gamma$ is of the form $F=F_\rho$ with $h=h_\rho$ for some unitary representation $\rho\colon\Gamma\to U(k)$.

\end{document}